\RequirePackage{fix-cm}
\documentclass[smallextended]{svjour3}       % onecolumn (second format)
\smartqed  % flush right qed marks, e.g. at end of proof
\usepackage{graphicx}
\usepackage{amsmath}
\usepackage{epstopdf} % needed if you have eps figures in a pdflatex manuscript
\usepackage{mathptmx}      % use Times fonts if available on your TeX system
\usepackage{amsmath,ulem}
\usepackage{stmaryrd}
\usepackage{amssymb}
\usepackage{empheq}
\usepackage{cases}
\usepackage{cite}
\usepackage{caption,graphicx}
\usepackage[none]{hyphenat}%去掉换行连字符，配合\sloppy{}放松间距限制

\usepackage{amsmath}
\usepackage{graphicx}
\usepackage{graphicx,epstopdf}
\usepackage{wrapfig}
\usepackage[caption=false]{subfig}
\usepackage{amssymb}
\usepackage{cases}

\usepackage{enumerate}
\usepackage{stmaryrd}
\usepackage{mathrsfs}
\usepackage{multirow}

\usepackage{indentfirst}

\usepackage{array} % <- Preamble

\usepackage[colorlinks,linkcolor=blue]{hyperref}
\hypersetup{backref=true,citebordercolor={1 1 1},citecolor=blue}
\hypersetup{linkbordercolor={1 1 1},linkcolor=blue,linktocpage=true}

\input psfig.sty

\newcommand{\be}{\begin{equation}}
\newcommand{\ee}{\end{equation}}
\newcommand{\ba}{\begin{array}}
\newcommand{\ea}{\end{array}}
\newcommand{\bea}{\begin{eqnarray}}
\newcommand{\eea}{\end{eqnarray}}
\newcommand{\beas}{\begin{eqnarray*}}
\newcommand{\eeas}{\end{eqnarray*}}

\smartqed  % flush right qed marks, e.g. at end of proof
\numberwithin{equation}{section}
\begin{document}

\title{A structure-preserving parametric finite element method for area-conserved generalized mean curvature flow%\thanks{Grants or other notes
%about the article that should go on the front page should be
%placed here. General acknowledgments should be placed at the end of the article.}
}
%\subtitle{Do you have a subtitle?\\ If so, write it here}

\titlerunning{SP-PFEM for area-conserved generalized curvature flow}        % if too long for running head

\author{Lifang Pei         \and
        Yifei Li %etc.
}

\authorrunning{L.~Pei, and Y.~Li} % if too long for running head

\institute{L.~Pei \at
              School of Mathematics and Statistics, Zhengzhou University, Zhengzhou, P. R. China 450001\\
              \email{plf5801@zzu.edu.cn}           %  \\
%             \emph{Present address:} of F. Author  %  if needed
           \and
           Y.~Li \at
              Department of Mathematics, National
              University of Singapore, Singapore 119076\\
              \email{e0444158@u.nus.edu}
}

\date{Received: date / Accepted: date}
% The correct dates will be entered by the editor

\maketitle

\begin{abstract}
\sloppy{}
We propose and analyze a structure-preserving parametric finite element method (SP-PFEM) to simulate the motion of closed curves governed by area-conserved generalized mean curvature flow in two dimensions (2D). We first present a variational formulation and rigorously prove that it preserves two fundamental geometric structures of the flows, i.e., (a) the conservation of the area enclosed by the closed curve; (b) the decrease of the perimeter of the curve. Then the variational formulation is approximated by using piecewise linear parametric finite elements in space to develop the semi-discrete scheme. With the help of the discrete Cauchy's inequality and discrete power mean inequality, the area conservation and perimeter decrease properties of the semi-discrete scheme are shown. On this basis, by combining the backward Euler method in time and a proper approximation of the unit normal vector, a structure-preserving fully discrete scheme is constructed successfully, which can preserve the two essential geometric structures simultaneously at the discrete level. Finally, numerical experiments test the convergence rate, area conservation, perimeter decrease and mesh quality, and depict the evolution of curves. Numerical results indicate that the proposed SP-PFEM provides a reliable and powerful tool for the simulation of area-conserved generalized mean curvature flow in 2D.
\keywords{area-conserved generalized mean curvature flow \and parametric finite element method \and area conservation \and perimeter decrease}
% \PACS{PACS code1 \and PACS code2 \and more}
\subclass{ 65M60 \and 65M12 \and 35K55 \and 53C44}
\end{abstract}

\section{Introduction}
\sloppy{}
Let $\Gamma(t)$ be a closed curve in two dimensions, which is represented by $\Gamma(t)=\boldsymbol{X}(\cdot,t)$ with $ \boldsymbol{X}:= \boldsymbol{X}(s,t)=
(x(s,t), y(s,t))^T \in \mathbb{R}^2,$ where $t$ denoting the time and $s$ being the arc length parametrization of $\Gamma(t)$. The motion of $\Gamma(t)$ under the area-conserved generalized mean curvature flow is governed by the following geometric partial differential equation:
\begin{align}\label{Model,cont,eq}
\left\{\begin{array}{ll}
\displaystyle \partial_{t}\boldsymbol{X}
=(\beta\kappa^{\alpha}-\lambda(t))\boldsymbol{n},~~0<s<L(t),~~t>0,\\
\kappa
=-\partial_{ss}\boldsymbol{X}\cdot \boldsymbol{n},
\end{array}\right.
\end{align}
where $L(t):=\int_{\Gamma(t)} 1ds$ and $\kappa:= \kappa(s,t)$ are the perimeter and the mean curvature of $\Gamma(t)$, respectively, $\boldsymbol{n}:=-(\partial_{s}\boldsymbol{X})^{\perp}$ is the outer unit normal to $\Gamma(t)$ with the notation $^{\perp}$ denoting clockwise rotation by $\frac{\pi}{2}$, $\alpha$ and $\beta$ are real numbers satisfying $\alpha\beta<0$, and
\begin{align}\label{Model,cont,lag}
\displaystyle \lambda(t)=\frac{\int_{\Gamma(t) } \beta\kappa^{\alpha} ds}{L(t)}=\frac{\int_{\Gamma(t) } \beta\kappa^{\alpha} ds}{\int_{\Gamma(t)} 1ds},
\end{align} which is specifically chosen such that the area enclosed by $\Gamma(t)$ is conserved.
The initial data for \eqref{Model,cont,eq} is given by
\begin{equation}\label{Model,cont,initial}
\boldsymbol{X}(s,0)=\boldsymbol{X}_{0}(s)=\left(x_{0}(s),y_0(s)\right)^T, \hspace{0.5cm}  0\leq s\leq {L({0})},
\end{equation}
where ${L({0})}$ is the perimeter of the initial curve $\Gamma(0):=\Gamma_{0}$.

Flows of the form \eqref{Model,cont,eq} without the term ${\lambda(t)}$, which are called the generalized mean curvature flow, have been widely studied in the literature. For example, when $\alpha=1$ and $\beta=-1$, this flow is the well-known mean curvature flow. It has been shown in \cite{gage1986heat,grayson1987heat,huisken1984flow} that the embedded curve evolving by the mean curvature flow shrinks to a point in finite time. Taking $0<\alpha\neq1$ and $\beta=-1$, we obtain the powers of mean curvature flow, which can be regarded as natural generalizations of the classical mean curvature flow (see \cite{andrews1998evolving,sapiro1994affine,schulze2005evolution,schulze2008nonlinear,sevcovic2001evolution}) and are fully nonlinear contracting flows for convex curves. In the case $\alpha=-1$ and $\beta=1$, this flow is an expanding flow and coincides with the inverse mean curvature flow \cite{andrews1998evolving,huisken2001inverse,ilmanen2008higher,urbas1991expansion}, a weak notion of which was used to prove the Riemannian Penrose inequality. Urbas proved in \cite{urbas1991expansion} that the curve with the positive mean curvature $\kappa>0$ stays strictly convex and smooth for all time, and exponentially converges, under appropriate rescaling, to a circle. This result can be extended to the case $\beta>0$ and $-1<\alpha<0$ for infinite time or $\alpha<-1$ for finite time (see \cite{andrews1998evolving,gerhardt2014non,scheuer2016pinching,lin2009expanding}). However, all these contracting and expanding flows mentioned above have the property of some sort of singularity formation.

A tempting approach is to directly define a flow by adding a constraining term to prevent the singularity phenomenon. When the constraining term ${\lambda(t)}$ as \eqref{Model,cont,lag} adopted, it can be shown that the flow \eqref{Model,cont,eq} possess two fundamental geometric properties, i.e., (a) the area of the region enclosed by the curve is conserved; (b) the perimeter of the curve decreases in time. In this case, the initial curve governed by this flow exists for all times and converges smoothly to a round circle. For more details, we refer \cite{andrews2021volume,cabezas2010volume,dittberner2021curve,escher1998volume,gage1986area,guo2019family,
huisken1987volume,li2009volume,ma2014non,sinestrari2015convex,tsai2015length} and reference therein. Thus \eqref{Model,cont,eq}--\eqref{Model,cont,lag} can be interpreted as a unified flow model with fixed areas and perimeters decrease properties, which has important applications in many research areas, such as material science and shape recovery in image processing \cite{dolcetta2002area}.

In this paper, we will investigate numerical approximations of the area-conserved generalized mean curvature flow \eqref{Model,cont,eq}--\eqref{Model,cont,lag}, paying particular attention to the two essential geometric structure preserving aspects. In the last decades, for the special case $\alpha=1$ and $\beta=-1$, i.e., area-conserved mean curvature flow, there have been various numerical methods in the literature, e.g., the finite difference method \cite{mayer2000numerical}, the level set method \cite{ruuth2003simple}, MBO method \cite{kublik2011algorithms} and crystalline algorithm \cite{ushijima2004convergence}. However, among them, structure-preserving properties were barely investigated except for the
crystalline algorithm in \cite{ushijima2004convergence}, which was shown to be area-preserving and perimeter decreasing. Since the crystalline algorithm involves the crystalline motion governed by a system of ordinary differential equations, it is no easy task to apply to the generalized model \eqref{Model,cont,eq}--\eqref{Model,cont,lag}. Recently, a parametric finite element
method (PFEM) was developed for the flow \eqref{Model,cont,eq}--\eqref{Model,cont,lag} with $\alpha=1$ and $\beta=-1$ in \cite{barrett2020parametric}. The PFEM was presented by Barrett, Garcke and N\"{u}rnberg (BGN) first for surface diffusion flows in \cite{barrett2007parametric} based on ideas of Dziuk \cite{dziuk1990algorithm}. Different from the parametric methods in \cite{bansch2005finite,deckelnick1995shortening,deckelnick2005computation,m2017approximations}, the PFEM involves an elegant variational formulation that allows tangential degrees of freedom so that the mesh quality can be improved significantly. Up to now, this method has been extensively applied to simulate geometric flows \cite{bao2021symmetrized,bao2017parametric,barrett2007variational,barrett2008numerical,barrett2008parametric,
li2021energy,zhao2021energy}. It is shown in \cite{barrett2020parametric} that the semi-discrete version of the PFEM leads to the equidistribution of mesh points along curves during the evolution, i.e., equal mesh distribution. This property prevents the possible mesh distortion and avoids the complex and undesired re-meshing steps, which plays a vital role in numerical simulations. Moreover, the PFEM in \cite{barrett2020parametric} enjoys unconditional perimeter decrease law. But the fully discrete scheme fails to exactly conserve the enclosed area of the discrete solution, which may lead to a challenge on the accuracy of the numerical solutions. In this situation, it is a natural requirement to develop novel numerical schemes that can inherit the good properties of the PFEM and preserve the enclosed area for the area-conserved generalized mean curvature flow \eqref{Model,cont,eq}--\eqref{Model,cont,lag}.

Recently, for surface diffusion flow, which also has area-conservation and perimeter decrease properties, we notice that a structure-preserving fully discrete PFEM was introduced in \cite{jiang2021perimeter} to preserve the two geometric properties for the discretized solutions. Nevertheless, the mesh quality is not well preserved during time evolution. Subsequently, motivated by the original ideas of \cite{jiang2021perimeter} and BGN \cite{barrett2007parametric}, Bao and Zhao constructed a structure-preserving PFEM in \cite{bao2021structure} for surface diffusion flow. The proposed method not only achieves the exact conservation of the enclosed area and decrease of perimeter, but also has asymptotic equal mesh distribution property, that is to say, mesh points on the polygonal curve tend to be equally distributed during the evolution and eventually be equidistributed. Very recently, by combining the skills from \cite{bao2021structure,jiang2021perimeter} and axisymmetric variational discretizations, a structure-preserving PFEM was developed in \cite{bao2022volume} for axisymmetric geometric evolution equations.

Inspired by \cite{bao2021structure,jiang2021perimeter}, we try to propose a structure-preserving PFEM (SP-PFEM) for the area-conserved generalized mean curvature flow \eqref{Model,cont,eq}--\eqref{Model,cont,lag} such that its two geometric structures are well preserved.
%The main aim of this paper is to propose a structure-preserving PFEM (SP-PFEM) for the area-conserved generalized mean curvature flow \eqref{Model,cont,eq}--\eqref{Model,cont,lag} such that its two geometric structures are well preserved.
We need to deal with troubles and challenges brought by the nonlinear term $\kappa^{\alpha}$ and the nonlocal term $\lambda(t)$, and hope that the SP-PFEM are available for all cases of $\alpha$. We first present a variational formulation and show the area conservation and perimeter decrease properties by utilizing Cauchy's inequality and the power mean inequality.
Then we discretize the variational formulation with the PFEM in space, and prove rigorously that the semi-discrete scheme preserves the two geometric structures with the help of the discrete Cauchy's inequality and discrete power mean inequality. On this basis, by transferring the novel approach in \cite{bao2021structure} to the approximation of the flow \eqref{Model,cont,eq}--\eqref{Model,cont,lag}, we choose the backward Euler method in time and a proper approximation of the unit normal vector combining the information of the curves at the current and next time step to construct a fully discrete scheme. It can be shown that the proposed scheme is area preserving and unconditional perimeter decreasing. Numerical experiments reveal that it also enjoys asymptotic equal mesh distribution property. The main contribution of our work lies in: (I) we establish a unified structure-preserving PFEM for the area-conserved generalized mean curvature flow, which is accurate and efficient in practical simulations; (II) the area conservation and perimeter decrease properties are rigorously proved for variational formulation, semi-discrete scheme and fully discrete scheme.

The rest of the paper is organized as follows. In section 2, we investigate the two geometric structures properties of the variational formulation. In section 3, a PFEM semi-discrete scheme is developed, and its area conservation and perimeter dissipation properties are proved rigorously. In section 4, we propose a SP-PFEM fully-discrete scheme, provide the proof of its structure preserving property, and present iterative methods for solving the resulting nonlinear system. Furthermore, numerical results are reported in section 5 to demonstrate the accuracy and efficiency of the proposed SP-PFEM, and some conclusions are drawn in section 6.

\section{ The variational formulation and its properties}
\setcounter{section}{2} \setcounter{equation}{0}
\sloppy{}
In this section, we present a variational formulation for the area-conserved generalized mean curvature flow \eqref{Model,cont,eq}--\eqref{Model,cont,lag} and establish its area conservation and perimeter decrease properties.

To obtain the variational formulation, a time independent variable $\rho$ is introduced such that $\Gamma(t)$ can be parameterized over the fixed domain $\rho\in \mathbb{I} =\mathbb{R}/\mathbb{Z}=[0,1]$ as
\begin{equation}
\Gamma(t):=\boldsymbol{X}(\cdot,t),\quad \boldsymbol{X}(\cdot, t): \mathbb{I}\to \mathbb{R}^2, \quad(\rho,t)\mapsto \boldsymbol{X}(\rho,t)=\left(x(\rho,t),y(\rho,t)\right)^T.
\end{equation}
Based on this parametrization, the arc length parameter ${s}$ is
computed by ${s}(\rho,t)=\displaystyle \int_ {0}^{\rho}|\partial_{q}\boldsymbol{X}| dq$, and there holds $\partial_{\rho}s=\left|\partial_{\rho}\boldsymbol{X}\right|$, $ds=\partial_{\rho}s d\rho=|\partial_{\rho}\boldsymbol{X}| d\rho$. During the evolution, we assume that the mean curvature $\kappa(\cdot,t)$ of $\Gamma(t)$ satisfies
\begin{equation}\label{assumption on kappa, vari}
    \kappa(\cdot,t)\geq 0\quad \forall\, 0<\alpha\neq 1  \qquad \text{  and  }\qquad \kappa(\cdot,t)>0 \quad\forall\, \alpha<0, \qquad \quad \forall t\geq 0.
\end{equation}
Besides, we further assume that during the evolution, there is a constant $C>1$ such that $1/C\leq |\partial_\rho \boldsymbol{X}|\leq C$, i.e., $\Gamma(t)$ is regular for all $t$.

As a preparation, we define the functional space as
\begin{equation}
L^2_p(\mathbb{I})=\left\{u:\mathbb{I}\rightarrow\mathbb{R}\Big|\int_{\mathbb{I}}|u(\rho)|^2 d\rho < +\infty\right\}
\end{equation}
equipped with the weighted ${L}^2$-inner product with respect to the closed curve $\Gamma(t)$
\begin{equation}
(u,v)_{\Gamma(t)}:=\int_{\Gamma(t)}u(s)v(s) ds=\int_{\mathbb{I}}u(s(\rho,t))v(s(\rho,t))\partial_{\rho}s d\rho,  ~~\forall u,v \in\ {L}_p^2(\mathbb{I}).
\end{equation}
We note that the assumption $1/C\leq |\partial_\rho \boldsymbol{X}|\leq C$ ensures the weighted inner products with respect to different $\Gamma(t)$s are equivalent, and thus the $L^2_p(\mathbb{I})$ can be interpreted as the usual $L^2_p(\mathbb{I})$ space.

At the same time, we introduce the Sobolev space
\begin{align*}
{H}^1_{{p}}(\mathbb{I})=\left\{u:\mathbb{I}\rightarrow\mathbb{R}\Big|u\in{L}^2_p(\mathbb{I}), \partial_{\rho}u\in{L}_p^2(\mathbb{I})\right\}.
\end{align*}
These notations can be easily extended to $\left[{L}^2_{{p}}(\mathbb{I})\right]^2$ and $\left[{H}^1_{{p}}(\mathbb{I})\right]^2$ for vector-valued functions.

By applying the identity $\kappa\boldsymbol{n}=-\partial_{ss}\boldsymbol{X}$ given in \cite{dziuk1990algorithm}, we reformulate \eqref{Model,cont,eq} into
\begin{subequations}
\begin{align}
\label{Model,cont,reformulate1}
&\displaystyle \partial_{t}\boldsymbol{X}\cdot
\boldsymbol{n}=\beta\kappa^{\alpha}-\lambda(t),~~0<s<L(t),\hspace{0.5cm} t>0,\\
\label{Model,cont,reformulate2}
&\kappa \boldsymbol{n}=-\partial_{ss}\boldsymbol{X}.
\end{align}
\end{subequations}

Then by taking the ${L}^2$-inner product of \eqref{Model,cont,reformulate1} with a test function $\phi\in {H}^1_{{p}}(\mathbb{I})$, we have
\begin{equation}
(\partial_{t}\boldsymbol{X}\cdot\boldsymbol{n},\phi)_{\Gamma(t)}-\left(\beta\kappa^{\alpha}-\lambda(t),\phi\right)_{\Gamma(t)}=0.
\end{equation}
Similarly, take the ${L}^2$-inner product of \eqref{Model,cont,reformulate2} with a test function $\boldsymbol{\omega}=(\omega_{1},\omega_{2})^T\in \left[{H}^1_{{p}}(\mathbb{I})\right]^2$ and integrating by parts, we get
\begin{equation}
(\kappa\boldsymbol{n},\boldsymbol{\omega})_{\Gamma(t)}-(\partial_{s}\boldsymbol{X},\partial_{s}\boldsymbol{\omega})_{\Gamma(t)}=0.
\end{equation}
Therefore, the variational formulation of \eqref{Model,cont,eq}--\eqref{Model,cont,lag} with the initial condition \eqref{Model,cont,initial} can be stated as follows: Given the initial curve $\Gamma(0):=\Gamma_0$ with $ \boldsymbol{X}(\cdot, 0)\in \left[{H}^1_{{p}}(\mathbb{I})\right]^2$, find the curve $\Gamma(t):=\boldsymbol{X}(\cdot,t)=\left(x(\cdot,t),y(\cdot,t)\right)^T\in \left[{H}^1_{{p}}(\mathbb{I})\right]^2$ and the mean curvature $\kappa(\cdot, t)\in {H}^1_{{p}}(\mathbb{I})$ such that
\begin{align}\label{Model,vari,eq}
\left\{\begin{array}{ll}
(\partial_{t}\boldsymbol{X}\cdot\boldsymbol{n},\phi)_{\Gamma(t)}-\left(\beta\kappa^{\alpha}-\lambda(t),\phi\right)_{\Gamma(t)}=0,~~~\phi\in {H}^1_{{p}}(\mathbb{I}),\\
(\kappa\boldsymbol{n},\boldsymbol{\omega})_{\Gamma(t)}-(\partial_{s}\boldsymbol{X},\partial_{s}\boldsymbol{\omega})_{\Gamma(t)}=0,
~~~~~~~~~~\boldsymbol{\omega}\in [{H}^1_{{p}}(\mathbb{I})]^2.
\end{array}\right.
\end{align}

Let $\Omega(t)$ be the region enclosed by the curve $\Gamma(t)$ and
\begin{align}
A(t):=\int_{\Omega(t)} dxdy=\int_{0}^{{L}(t)}{y}(s,t)\partial_{s}x(s,t)ds,\hspace{0.5cm}t\geq 0
\end{align}
be the area of $\Omega(t)$. Then we have

\begin{theorem}[Area conservation and perimeter decrease]\label{Them: Area-perimeter vari} \\
Let $\left(\boldsymbol{X}(\cdot,t),\kappa(\cdot,t)\right)\in \left[{H}^1_{{p}}(\mathbb{I})\right]^2\times {H}^1_{{p}}(\mathbb{I})$ be a solution of the variational
formulation \eqref{Model,vari,eq} with the assumption \eqref{assumption on kappa, vari} holds on $\kappa{(\cdot,t)}$. Then the area $A(t)$ is conserved and the perimeter $L(t)$ is decreasing, i.e., $\forall t\geq t_{1} \geq 0$, there holds
\begin{equation}\label{preserve, vari}
A(t)=A(t_{1})\equiv A(0):=\int_{0}^{L({0})}y_{0}(s)\partial_{s}x_{0}(s) ds,\qquad L(t)\leq L(t_{1})\leq L(0).
\end{equation}
\end{theorem}
\begin{proof}\quad Differentiating $A(t)$ with respect to $t$ and integrating by parts, we obtain
\begin{equation}\label{Area-different,vari}
\begin{aligned}
\frac{dA(t)}{dt}&=\frac{d}{dt}\int_{0}^{L(t)}y(s,t)\partial_{s}x(s,t)ds=\frac{d}{dt}\int_{0}^{1}y(\rho, t)\partial_{\rho}x(\rho,t)d\rho\\
&=\int_{0}^{1}(\partial_{t}y\partial_{\rho}x+y\partial_{t}\partial_{\rho } x)d\rho
=\int_{0}^{1}(\partial_{t}y\partial_{\rho}x-\partial_{\rho} y\partial_{t}x)d\rho+(y\partial_{t}x)|_{\rho=0}^{\rho=1}\\
&=\int_{0}^{1}(\partial_{t}\boldsymbol{X})\cdot(-\partial_{\rho}y,\partial_{\rho}x)d\rho=(\partial_{t}\boldsymbol{X}\cdot\boldsymbol{n},1)_{\Gamma(t)}
=(\beta\kappa^{\alpha}-\lambda(t),1)_{\Gamma(t)}\\&=\beta\int_{\Gamma(t)}\kappa^{\alpha}ds-\beta\frac{\int_{\Gamma(t) } \kappa^{\alpha} ds}{L(t)}\int_{\Gamma(t)}1 ds\equiv 0,
\end{aligned}
\end{equation}
where we used the first equation in \eqref{Model,vari,eq} by taking $\phi\equiv1$. Obviously, formula \eqref{Area-different,vari} immediately implies the area-preserving property in \eqref{preserve, vari}.

Similarly, differentiating $L(t)$ with respect to $t$ and integrating by parts, we have
\begin{equation}\label{perimeter-different,vari}
\begin{aligned}
\frac{dL(t)}{dt}&=\frac{d}{dt}\int_0^{L(t)}1ds=\frac{d}{dt}\int_{0}^{1}\partial_{\rho}sd\rho=\frac{d}{dt}\int_{0}^{1}|\partial_{\rho}\boldsymbol{X}|d\rho\\
&=\int_{0}^{1}\frac{\partial_{\rho}\boldsymbol{X}\cdot(\partial_{\rho}\partial_{t}\boldsymbol{X})}{|\partial_{\rho}\boldsymbol{X}|}d\rho=\int_{0}^{1}\partial_{\rho}\boldsymbol{X}\partial_{s}\rho\cdot(\partial_{s}\partial_{t}\boldsymbol{X})\partial_{\rho}sd\rho\\
&=\int_{\Gamma(t)}\partial_{s}\boldsymbol{X}\cdot(\partial_{s}\partial_{t}\boldsymbol{X})ds=(\partial_{s}\boldsymbol{X},\partial_{s}\partial_{t}\boldsymbol{X})_{\Gamma(t)}.
\end{aligned}
\end{equation}
Setting $\boldsymbol{\omega}=\partial_{t}\boldsymbol{X}$ and $\phi=\kappa$ in the variational formulation \eqref{Model,vari,eq}, we get
 \begin{equation}\label{perimeter-different,vari2}
\begin{aligned}
\frac{dL(t)}{dt}&=(\kappa\boldsymbol{n},\partial_{t}\boldsymbol{X})_{\Gamma(t)}=(\partial_{t}\boldsymbol{X}\cdot\boldsymbol{n},\kappa)_{\Gamma(t)}
=(\beta\kappa^{\alpha}-\lambda(t),\kappa)_{\Gamma(t)}\\&\displaystyle =L(t)\beta\left[\frac{1}{L(t)}\int_{\Gamma(t)}\kappa^{\alpha+1}ds
-\frac{1}{L(t)^2}\left(\int_{\Gamma(t)}\kappa^{\alpha}ds\right)\left(\int_{\Gamma(t)}\kappa ds\right)\right].
\end{aligned}
\end{equation}

When $\alpha=1$ and $\beta<0$, for any $\kappa\in \mathbb{R}$, by Cauchy's inequality, we have
\begin{equation}\label{int alpha=1}
\left(\int_{\Gamma(t)}\kappa ds\right)^2\leq \int_{\Gamma(t)}1 ds\,\int_{\Gamma(t)}\kappa^2 ds,
\end{equation}
thus
\begin{equation}
\begin{array}{ll}
\displaystyle \frac{dL(t)}{dt}& \displaystyle=L(t)\beta\left[\frac{1}{L(t)}\int_{\Gamma(t)}\kappa^{2}ds
-\frac{1}{L(t)^2}\left(\int_{\Gamma(t)}\kappa ds\right)\left(\int_{\Gamma(t)}\kappa ds\right)\right]\\[4mm]
&\displaystyle \leq L(t)\beta\left[\frac{1}{L(t)}\int_{\Gamma(t)}\kappa^{2}ds
-\frac{1}{L(t)^2}\left(\int_{\Gamma(t)}1 ds\right)\left(\int_{\Gamma(t)}\kappa^2 ds\right)\right]
{=}0.
\end{array}
\end{equation}

When $\alpha=-1, \kappa>0$ and $\beta>0$, by Cauchy's inequality, we have
\begin{equation}\label{int alpha=-1}
\left(\int_{\Gamma(t)}1 ds\right)^2\leq \int_{\Gamma(t)}\frac{1}{\kappa} ds\int_{\Gamma(t)}\kappa ds,
\end{equation}
which combines with $\beta>0$  also implies
\begin{equation}
\begin{array}{ll}
\displaystyle \frac{dL(t)}{dt}=L(t)\beta\left[\frac{1}{L(t)}\int_{\Gamma(t)}1ds
-\frac{1}{L(t)^2}\left(\int_{\Gamma(t)}\frac{1}{\kappa}ds\right)\left(\int_{\Gamma(t)}\kappa ds\right)\right]
\leq0.
\end{array}
\end{equation}
Therefore the perimeter decrease property in \eqref{preserve, vari} holds for $|\alpha|=1$.

When $0<|\alpha|\neq 1$, to show $\frac{dL(t)}{dt}\leq0$, we need the power mean inequality \cite{hardy1934inequalities}[Page 144]:
\begin{lemma}[Power mean inequality]\label{lem: power mean int}
For any finite interval $E=(a, b)$, and the non-zero parameters $p,q$ satisfy $-\infty<p<q<\infty$. Suppose $f(x)$ is bounded and $f\geq 0$ in $E$ if $p>0$, and $f>0$ in $E$ if $p<0$,  we have
\begin{equation}\label{eq: power mean int}
    \left(\frac{\int_a^b f^p(x)dx}{|b-a|}\right)^{\frac{1}{p}}\leq \left(\frac{\int_a^b f^q(x)dx}{|b-a|}\right)^{\frac{1}{q}}.
\end{equation}
\end{lemma}

For the case $\alpha>0, \alpha\neq 1$ and $\kappa\geq 0$, we know $\beta<0$. By taking $E:=\Gamma(t),$ $(f, p, q)=(\kappa^\alpha, 1, \frac{\alpha+1}{\alpha})$ and $(\kappa, 1, \alpha+1)$, respectively, we obtain
%\begin{subequations}
%\label{int alpha>0}
%\begin{align}
%\label{int alpha>0 eq1}
%    &\left(\frac{\int_{\Gamma(t)} \kappa^\alpha ds}{L(t)}\right)\leq \left(\frac{\int_{\Gamma(t)} \kappa^{\alpha+1}ds}{L(t)}\right)^{\frac{\alpha}{\alpha+1}},\\
%\label{int alpha>0 eq2}
%    &\left(\frac{\int_{\Gamma(t)} \kappa ds}{L(t)}\right)\leq \left(\frac{\int_{\Gamma(t)} \kappa^{\alpha+1}ds}{L(t)}\right)^{\frac{1}{\alpha+1}}.
%\end{align}
%\end{subequations}
\begin{equation} \label{int alpha>0}
\frac{\int_{\Gamma(t)} \kappa^\alpha ds}{L(t)}\leq \left(\frac{\int_{\Gamma(t)} \kappa^{\alpha+1}ds}{L(t)}\right)^{\frac{\alpha}{\alpha+1}},\quad
\frac{\int_{\Gamma(t)} \kappa ds}{L(t)}\leq \left(\frac{\int_{\Gamma(t)} \kappa^{\alpha+1}ds}{L(t)}\right)^{\frac{1}{\alpha+1}}.
\end{equation}
By substituting \eqref{int alpha>0} in \eqref{perimeter-different,vari2}, we have
%\begin{align}
%    \frac{dL(t)}{dt}&=L(t)\beta \left[\left(\frac{\int_{\Gamma(t)} \kappa^{\alpha+1}ds}{L(t)}\right)^{\frac{\alpha}{\alpha+1}}\left(\frac{\int_{\Gamma(t)} \kappa^{\alpha+1}ds}{L(t)}\right)^{\frac{1}{\alpha+1}}\right.\\
%    &-\frac{1}{L(t)^2}\left(\int_{\Gamma(t)}\kappa^{\alpha}ds\right)\left(\int_{\Gamma(t)}\kappa ds\right)\right]\leq 0.
%\end{align}
\begin{equation}
\begin{aligned}
    \frac{dL(t)}{dt}&=L(t)\beta \left[\left(\frac{\int_{\Gamma(t)} \kappa^{\alpha+1}ds}{L(t)}\right)^{\frac{\alpha}{\alpha+1}}\left(\frac{\int_{\Gamma(t)} \kappa^{\alpha+1}ds}{L(t)}\right)^{\frac{1}{\alpha+1}}\right.\\
    &\quad\quad\left.-\frac{1}{L(t)^2}\left(\int_{\Gamma(t)}\kappa^{\alpha}ds\right)\left(\int_{\Gamma(t)}\kappa ds\right)\right]\leq 0.
\end{aligned}
\end{equation}
For the case $\alpha<0, \alpha\neq-1$ and $\kappa>0$, we know $\beta>0$. By taking $E:=\Gamma(t),$ $(f, p, q)=(\kappa^\alpha, \frac{\alpha+1}{\alpha},1)$ and $(\kappa, \alpha+1,1)$, respectively, we obtain
\begin{equation} \label{int alpha<0}
\left(\frac{\int_{\Gamma(t)} \kappa^{\alpha+1}ds}{L(t)}\right)^{\frac{\alpha}{\alpha+1}}\leq \frac{\int_{\Gamma(t)} \kappa^\alpha ds}{L(t)},\quad
    \left(\frac{\int_{\Gamma(t)} \kappa^{\alpha+1}ds}{L(t)}\right)^{\frac{1}{\alpha+1}}\leq \frac{\int_{\Gamma(t)} \kappa ds}{L(t)}.
\end{equation}
Similarly, by substituting \eqref{int alpha<0} in \eqref{perimeter-different,vari2}, we have $\frac{dL(t)}{dt}\leq 0$.
Thus, the perimeter decrease property in \eqref{preserve, vari} is proven for all possible $\alpha$ cases. The proof is completed.
\end{proof}
\section{ A PFEM semi-discrete scheme and its properties}
\setcounter{section}{3} \setcounter{equation}{0}
\sloppy{}
In this section, we present a parametric finite element method $\text{(PFEM)}$ with conforming piecewise linear elements to discretize the variational formulation \eqref{Model,vari,eq} in space direction, and show that the scheme preserves the area conservation and perimeter decrease properties.
\subsection{The semi-discretization in space}
\par
Assume that the interval $\mathbb{I}$ is divided into a uniform partition $\mathbb{I}=\bigcup_{j=1}^{N}I_{j}$ with the grid points $\left\{\rho_{j}:=jh\right\}_{j=0}^N$, where $\rho_{0}=\rho_{N}$ by periodic, the subintervals ${I}_{j}=[\rho_{j-1},\rho_{j}]$ for $j=1,2,\cdot\cdot\cdot,N$, $N\geq 3$ is a positive integer and $h=1/N$. Let $\mathcal{P}_{1}({I}_{j})$ be the space of all polynomials with degree at most $1$ over ${I}_{j}$, then the conforming linear finite element space can be described by
\begin{align*}
\mathbb{V}^h:&=\left\{u^h\in C(\mathbb{I})\Big|\,\color{black}u^h|_{I_{j}}\in\mathcal{P}_{1}({I}_{j}),\quad\forall j=1,2,\cdot\cdot\cdot,N\right\}.
\end{align*}
\par
Let $\Gamma^h(t):=\boldsymbol{X}^h(\cdot,t)=\left({x}^h(\cdot,t),{y}^h(\cdot,t)\right)^T\in \left[\mathbb{V}^h\right]^2$ and $\kappa^h(\cdot, t)\in \mathbb{V}^h$ be the piecewise linear finite element approximation of the solution $(\boldsymbol{X}(\cdot,t),\kappa(\cdot, t))\in \left[{H}^1_{{p}}(\mathbb{I})\right]^2\times {H}^1_{{p}}(\mathbb{I})$ of the variational formulation \eqref{Model,vari,eq}. In fact, the polygonal curve $\Gamma^h(t)$ consists of ordered line segments $\left\{\boldsymbol{h}_{j}\right\}_{j=1}^{N}$, and we always assume that they satisfy
\begin{equation}
 \begin{aligned}
h_{\min}(t):=\underset{1 \leq j \leq N}{\min}|\boldsymbol{h}_{j}(t)|>0~~\mathrm{with}~~
\boldsymbol{h}_{j}(t):=\boldsymbol{X}^h(\rho_{j},t)-\boldsymbol{X}^h(\rho_{j-1},t),
 \end{aligned}
\end{equation}
where $|\boldsymbol{h}_{j}(t)|$ denotes the length of the vector $\boldsymbol{h}_{j}(t)$. Similar to the continuous equation, we further propose the following assumption on $\kappa^h(\cdot, t)$:
\begin{equation}\label{assumption on kappa, semi}
     \kappa^h(\rho_j,t)\geq 0~ \forall\, 0<\alpha\neq 1 ~ \text{  and  } ~ \kappa^h(\rho_j, t)>0 ~\forall\, \alpha<0, \quad \forall j=0,1,\cdot\cdot\cdot,N, ~t\geq 0.
 \end{equation}

Meanwhile, it is easy to check that the unit tangential vector $\boldsymbol{\tau}^h$ and the outward unit normal vector $\boldsymbol{n}^h$ of the curve $\Gamma^h(t)$ are constant vectors on each interval $I_{j}$. They can be computed on each interval $I_{j}$ as
\begin{equation}
\begin{aligned}
\boldsymbol{\tau}^h|_{I_{j}}=\frac{\boldsymbol{h}_{j}}{\left|\boldsymbol{h}_j\right|}:=\boldsymbol{\tau}_{j}^{h},~~
\boldsymbol{n}^h|_{I_{j}}=-(\boldsymbol{\tau}_{j}^{h})^\perp=-\frac{(\boldsymbol{h}_{j})^\perp}{|\boldsymbol{h}_{j}|}:=\boldsymbol{n}_{j}^{h},\quad\forall j=1,2,\cdot\cdot\cdot,N.
\end{aligned}
\end{equation}
Furthermore, for two piecewise continuous scalar/vector functions $u$ and $v$ defined on $\mathbb{I}$, with possible jumps at the nodes $\{\rho_{j}\}_{j=0}^N$, we introduce the mass lumped inner product $(\cdot,\cdot)_{\Gamma^{h}(t)}^{h}$ over $\Gamma^{h}(t)$ as
\begin{equation}
(u,v)_{\Gamma^{h}}^{h}:=\frac{1}{2}\sum\limits_{j=1}^{N}|\boldsymbol{h}_{j}|\left[(u\cdot v)(\rho_{j}^{-})+(u\cdot v)(\rho_{j-1}^{+})\right],
\end{equation}
where $u(\rho_{j}^{\pm})=\lim\limits_{\rho\rightarrow\rho_{j}^{\pm}}u(\rho)$ for $j=0,1,\cdot\cdot\cdot,N$ are the one-sided limits.

Then a semi-discrete scheme based on the variational formulation \eqref{Model,vari,eq} can be constructed as: Given the initial curve $\Gamma^h(0):=\Gamma_0^{h}$ with $\boldsymbol{X}^h(\cdot,0)\in \left[\mathbb{V}^h\right]^2$, find the closed curve $\Gamma^h(t):=\boldsymbol{X}^h(\cdot,t)=\left(x^h(\cdot,t),y^h(\cdot,t)\right)^T\in \left[\mathbb{V}^h\right]^2$ and the mean curvature $\kappa^h(\cdot,t)\in \mathbb{V}^h$ such that
\begin{align}\label{Model,semi,eq}
\left\{\begin{array}{ll}
(\partial_{t}\boldsymbol{X}^{h}\cdot\boldsymbol{n}^{h},\phi^{h})_{\Gamma^{h}}^h-(\beta(\kappa^{h})^{\alpha}-\lambda^h,\phi^{h})_{\Gamma^{h}}^h=0,~~~\phi^{h}\in \mathbb{V}^h,\\
(\kappa^{h}\boldsymbol{n}^{h},\boldsymbol{\omega}^{h})_{\Gamma^{h}}^h-(\partial_{s}\boldsymbol{X}^{h},\partial_{s}\boldsymbol{\omega}^{h})_{\Gamma^{h}}^h=0,~~~~~~~~~~~~\boldsymbol{\omega}^{h}\in [\mathbb{V}^h]^2,
\end{array}\right.
\end{align}
where $\Gamma^h(0)=\boldsymbol{X}^h(\cdot,0)\in \left[\mathbb{V}^h\right]^2$ is the interpolation of the initial curve $\Gamma_0$ satisfying $\boldsymbol{X}^h(\rho_j,0)=\boldsymbol{X}(\rho_j,0), \forall j=0, 1,2,\ldots, N,$ and
\begin{align}\label{Model,semi,lag}
\displaystyle \lambda^h=\frac{(\beta(\kappa^{h})^{\alpha},1)_{\Gamma^{h}}^h}{(1,1)_{\Gamma^{h}}^h}.
\end{align}
\subsection{Area conservation and perimeter decrease}

Denote the area of the region enclosed by the closed curve $\Gamma^{h}(t)$ by $A^{h}(t)$ and the perimeter of $\Gamma^{h}(t)$ by $L^{h}(t)$, which are defined as follows:
\begin{equation}
A^{h}(t):=\frac{1}{2}\sum_{j=1}^{N}(x_{j}-x_{j-1})(y_{j}+y_{j-1}),\qquad
L^{h}(t):=\sum_{j=1}^{N}|\boldsymbol{h}_{j}(t)|=(1,1)_{\Gamma^{h}}^h,
\end{equation}
where $\boldsymbol{X}_{j}(t)=(x_{j}(t),y_{j}(t))^T:=\boldsymbol{X}^h(\rho_{j},t)$ for $j=0,1,\cdot\cdot\cdot,N$.
Then we have
\begin{theorem}[Area conservation and perimeter decrease]\label{Them: Area-perimeter semi}\\ Let $\left(\boldsymbol{X}^h(\cdot,t),\kappa^h(\cdot,t)\right)\in \left[\mathbb{V}^h\right]^2\times \mathbb{V}^h $ be a solution of the PFEM semi-discrete scheme \eqref{Model,semi,eq} with the assumption \eqref{assumption on kappa, semi} holds on $\kappa^h(\cdot, t)$. Then the area $A^{h}(t)$ is conserved and the perimeter $L^{h}(t)$ is decreasing, i.e., for $t\geq t_{1}\geq 0$, there holds
\begin{equation}\label{Area-preserve, semi}
 A^h(t)= A^h(t_{1})\equiv A^h(0):=\frac{1}{2}\sum_{j=1}^{N}\left[x_{0}(s_{j})-x_{0}(s_{j-1})\right]\left[y_{0}(s_{j})+y_{0}(s_{j-1})\right],
\end{equation}
\begin{equation}\label{perimeter-preserve, semi}
 L^{h}(t)\leq L^{h}(t_{1})\leq L^{h}(0):=\sum_{j=1}^{N}|\boldsymbol{h}_{j}(0)|.
\end{equation}
\end{theorem}
\begin{proof}\quad
By applying proposition 3.1 in \cite{zhao2021energy}, we obtain
\begin{equation}
\label{Area-preserve, semi, eq1}
\begin{aligned}
\frac{dA^{h}(t)}{dt}=\frac{1}{2}\sum_{j=1}^{N}(\partial_{t}\boldsymbol{X}_{j}^{h}+\partial_{t}\boldsymbol{X}_{j-1}^{h})\cdot(|\boldsymbol{h}_{j}|{\boldsymbol{n}}_{j}^{h})=(\partial_{t}\boldsymbol{X}^{h}\cdot\boldsymbol{n}^{h},1)_{\Gamma^{h}}^h.
\end{aligned}
\end{equation}
Taking $\phi^{h}\equiv1$ in \eqref{Model,semi,eq}, and employing \eqref{Model,semi,lag} and \eqref{Area-preserve, semi, eq1}, we are able to derive
\begin{equation}
\frac{dA^{h}(t)}{dt}=\left(\beta(\kappa^{h})^{\alpha}-\lambda^h,1\right)_{\Gamma^{h}}^h
=\left(\beta(\kappa^{h})^{\alpha},1\right)_{\Gamma^{h}}^h-\frac{\left(\beta(\kappa^{h})^{\alpha},1\right)_{\Gamma^{h}}^h}{(1,1)_{\Gamma^{h}}^h}(1,1)_{\Gamma^{h}}^h\equiv0,
\end{equation}
which leads to \eqref{Area-preserve, semi} immediately. In other words, $A^{h}(t)$ is conserved.

At the same time, differentiating $L^{h}(t)$ with respect to $t$, we get
\begin{equation}\label{perimeter-different, semi eq1}
\begin{aligned}
\frac{dL^{h}(t)}{dt}&=\frac{d}{dt}\left(\sum_{j=1}^{N}|\boldsymbol{h}_{j}|\right)=\sum_{j=1}^{N}\frac{\boldsymbol{h}_{j}}{|\boldsymbol{h}_{j}|}\cdot\frac{d\boldsymbol{h}_{j}}{dt}
=\sum_{j=1}^{N}|\boldsymbol{h}_{j}|\boldsymbol{\tau}^h_{j}\cdot\frac{1}{|\boldsymbol{h}_{j}|}\frac{d\boldsymbol{h}_{j}}{dt}\\
&=\sum_{j=1}^{N}|\boldsymbol{h}_{j}|(\partial_{s}\boldsymbol{X}^{h})\mid_{I_{j}}\cdot\,(\partial_{s}(\partial_{t}\boldsymbol{X}^{h}))\mid_{I_{j}}
=(\partial_{s}\boldsymbol{X}^{h},\partial_{s}\partial _{t}\boldsymbol{X}^{h})_{\Gamma^{h}}^h.
\end{aligned}
\end{equation}
Taking $\phi^{h}=\kappa^{h}$ and $\boldsymbol{\omega}^{h}=\partial_{t}\boldsymbol{X}^{h}$ in \eqref{Model,semi,eq}, combining \eqref{Model,semi,lag} and \eqref{perimeter-different, semi eq1}, we can deduce that
\begin{equation}\label{perimeter-different semi eq2}
\begin{aligned}
\frac{dL^h(t)}{dt}&=\left(\kappa^{h}{\boldsymbol{n}^{h},\partial _{t}\boldsymbol{X}^{h}}\right)_{\Gamma^{h}}^h=\left(\beta(\kappa^{h})^{\alpha}-\lambda^h,\kappa^{h}\right)_{\Gamma^{h}}^h\\
&=\beta\left[\left((\kappa^{h})^{\alpha},\kappa^{h}\right)_{\Gamma^{h}}^h-\frac{\left((\kappa^{h})^{\alpha},1\right)_{\Gamma^{h}}^h(1,\kappa^{h})_{\Gamma^{h}}^h}{(1,1)_{\Gamma^{h}}^h}\right].
\end{aligned}
\end{equation}

When $|\alpha|=1$, it follows from the Cauchy's inequality and the arithmetic\\-geometric/-harmonic mean inequality that
\begin{equation}
\begin{aligned}\label{dis alpha=1}
        &(\kappa^{h},1)_{\Gamma^{h}}^h(\kappa^{h},1)_{\Gamma^{h}}^h=\left(\frac{1}{2}\sum_{j=1}^{N}|\boldsymbol{h}_{j}|
\left(\kappa^{h}(\rho_j^{-})+\kappa^{h}(\rho_{j-1}^{+})\right)\right)^2\\
        &\leq\left(\sum_{j=1}^{N}|\boldsymbol{h}_{j}|\right)\,\left(\frac{1}{4}\sum_{j=1}^{N}|\boldsymbol{h}_{j}|
\left(\kappa^{h}(\rho_j^{-})+\kappa^{h}(\rho_{j-1}^{+})\right)^2\right)\\
        &\leq \left(\sum_{j=1}^{N}|\boldsymbol{h}_{j}|\right)\,\left(\frac{1}{4}\sum_{j=1}^{N}|\boldsymbol{h}_{j}|2
\left((\kappa^{h}(\rho_j^{-}))^2+(\kappa^{h}(\rho_{j-1}^{+}))^2\right)\right)\\
        &=(1,1)_{\Gamma^{h}}^h(\kappa^{h},\kappa^{h})_{\Gamma^{h}}^h,
\end{aligned}
\end{equation}
and
\begin{equation}
\begin{aligned}\label{dis alpha=-1}
&(1,1)_{\Gamma^{h}}^h((\kappa^{h})^{-1},\kappa^{h})_{\Gamma^{h}}^h=(1,1)_{\Gamma^{h}}^h(1,1)_{\Gamma^{h}}^h=\left(\sum_{j=1}^{N}|\boldsymbol{h}_{j}|\right)^2\\
&\leq \left(\sum_{j=1}^{N}|\boldsymbol{h}_{j}|\frac{\frac{1}{\kappa^h}(\rho_j^-)+\frac{1}{\kappa^h}(\rho_{j-1}^+)}{2}\right)\left(\sum_{j=1}^{N}|\boldsymbol{h}_{j}|\frac{2}{\frac{1}{\kappa^h}(\rho_j^-)+\frac{1}{\kappa^h}(\rho_{j-1}^+)}\right)\\
&\leq \left(\sum_{j=1}^{N}|\boldsymbol{h}_{j}|\frac{\frac{1}{\kappa^h}(\rho_j^-)+\frac{1}{\kappa^h}(\rho_{j-1}^+)}{2}\right)\left(\sum_{j=1}^{N}|\boldsymbol{h}_{j}|\frac{\kappa^h(\rho_j^-)+\kappa^h(\rho_{j-1}^+)}{2}\right)\\
&=\left((\kappa^{h})^{-1},1\right)_{\Gamma^{h}}^h(\kappa^{h},1)_{\Gamma^{h}}^h.
\end{aligned}
\end{equation}

Substituting \eqref{dis alpha=1}--\eqref{dis alpha=-1} into \eqref{perimeter-different semi eq2} for $\alpha=\pm 1$, respectively, and using $\alpha\,\beta<0$, we deduce that
\begin{equation}
\begin{aligned}
\frac{dL^h(t)}{dt}=\beta\left[\left((\kappa^{h})^{\alpha},\kappa^{h}\right)_{\Gamma^{h}}^h-\frac{\left((\kappa^{h})^{\alpha},1\right)_{\Gamma^{h}}^h(1,\kappa^{h})_{\Gamma^{h}}^h}{(1,1)_{\Gamma^{h}}^h}\right]
\leq 0,
\end{aligned}
\end{equation}
which leads to the desired perimeter decrease result \eqref{perimeter-preserve, semi} for $|\alpha|=1$.

When $0<|\alpha|\neq1$, similar to the proof of theorem \ref{Them: Area-perimeter vari}, to show $\frac{dL^h(t)}{dt}\leq0$, we need the discrete version of power mean inequality \cite{hardy1934inequalities}[theorem 16]:
\begin{lemma}[Power mean inequality, discrete]\label{lem: power mean dis}
Suppose $n\in \mathbb{Z}^+$ and the non-zero parameters $p,q$ satisfy $-\infty<p<q<\infty$. Let $x_1, x_2, \ldots,x_n$ be non-negative numbers if $p>0$, or positive numbers if $p<0$. For any $\omega_1, \omega_2,\ldots, \omega_n$ satisfying $0\leq \omega_i< +\infty\,\forall i$ with $\sum\limits_{i=1}^n\omega_i>0$,  we have
\begin{equation}\label{eq: power mean dis}
    \left(\frac{\sum\limits_{i=1}^n \omega_ix_i^p}{\sum\limits_{i=1}^n\omega_i}\right)^{\frac{1}{p}}\leq  \left(\frac{\sum\limits_{i=1}^n \omega_ix_i^q}{\sum\limits_{i=1}^n\omega_i}\right)^{\frac{1}{q}}.
\end{equation}
\end{lemma}

For the case $\alpha>0, \alpha\neq 1$, we employ the discrete power mean inequality for $n=2N, \omega_{2j-1}=\omega_{2j}=\frac{1}{2}|\boldsymbol{h}_j|, x_{2j-1}=(\kappa^h)^\alpha(\rho_j^-), x_{2j}=(\kappa^h)^\alpha(\rho_{j-1}^+),\, \forall\, j=1, 2,\ldots, N,\\ p=1, q=\frac{\alpha+1}{\alpha}$, and $x_{2j-1}=(\kappa^h)(\rho_j^-), x_{2j}=(\kappa^h)(\rho_{j-1}^+),\, \forall\, j=1, 2,\ldots, N,\, p=1, q=\alpha{+1}$, respectively, we obtain
\begin{subequations}
\label{dis alpha>0}
\begin{align}
\label{dis alpha>0 eq1}
    &\qquad  \frac{\sum_{j=1}^{N}\frac{1}{2} |\boldsymbol{h}_j| \left((\kappa^h)^\alpha(\rho_{j}^-)+(\kappa^h)^\alpha(\rho_{j-1}^+)\right)}{L^h(t)}\\
    &\qquad  \leq \left(\frac{\sum_{j=1}^{N}\frac{1}{2} |\boldsymbol{h}_j| \left((\kappa^h)^{\alpha+1}(\rho_{j}^-)+(\kappa^h)^{\alpha+1}(\rho_{j-1}^+)\right)}{L^h(t)}\right)^{\frac{\alpha}{\alpha+1}},\nonumber\\
\label{dis alpha>0 eq2}
    &\qquad  \frac{\sum_{j=1}^{N}\frac{1}{2} |\boldsymbol{h}_j| \left((\kappa^h)(\rho_{j}^-)+(\kappa^h)(\rho_{j-1}^+)\right)}{L^h(t)}\\
    &\qquad  \leq \left(\frac{\sum_{j=1}^{N}\frac{1}{2} |\boldsymbol{h}_j| \left((\kappa^h)^{\alpha+1}(\rho_{j}^-)+(\kappa^h)^{\alpha+1}(\rho_{j-1}^+)\right)}{L^h(t)}\right)^{\frac{1}{\alpha+1}}.\nonumber
\end{align}
\end{subequations}
Combining \eqref{dis alpha>0 eq1} and \eqref{dis alpha>0 eq2}, we have
\begin{equation}\label{dis alpha>0 eq3}
\begin{aligned}
&\frac{1}{L^h(t)}((\kappa^{h})^{\alpha},\kappa^{h})_{\Gamma^{h}}^h
\\&=\left(\frac{\sum_{j=1}^{N}\frac{1}{2} |\boldsymbol{h}_j| \left((\kappa^h)^{\alpha+1}(\rho_{j}^-)+(\kappa^h)^{\alpha+1}(\rho_{j-1}^+)\right)}{L^h(t)}\right)\\
&=\left(\frac{\sum_{j=1}^{N}\frac{1}{2} |\boldsymbol{h}_j| \left((\kappa^h)^{\alpha+1}(\rho_{j}^-)+(\kappa^h)^{\alpha+1}(\rho_{j-1}^+)\right)}{L^h(t)}\right)^{\frac{\alpha}{\alpha+1}}
\\& \qquad \cdot\left(\frac{\sum_{j=1}^{N}\frac{1}{2} |\boldsymbol{h}_j| \left((\kappa^h)^{\alpha+1}(\rho_{j}^-)+(\kappa^h)^{\alpha+1}(\rho_{j-1}^+)\right)}{L^h(t)}\right)^{\frac{1}{\alpha+1}}
\end{aligned}
\end{equation}
\begin{equation}
\begin{aligned}
&\geq \left(\frac{\sum_{j=1}^{N}\frac{1}{2} |\boldsymbol{h}_j| \left((\kappa^h)^\alpha(\rho_{j}^-)+(\kappa^h)^\alpha(\rho_{j-1}^+)\right)}{L^h(t)}\right)
\nonumber\\& \qquad \cdot\left(\frac{\sum_{j=1}^{N}\frac{1}{2} |\boldsymbol{h}_j| \left((\kappa^h)(\rho_{j}^-)+(\kappa^h)(\rho_{j-1}^+)\right)}{L^h(t)}\right)\nonumber\\
&=\frac{1}{L^h(t)}\left((\kappa^{h})^{\alpha},1\right)_{\Gamma^{h}}^h\frac{1}{L^h(t)}(\kappa^{h},1)_{\Gamma^{h}}^h.\nonumber
\end{aligned}
\end{equation}
Similarly, for the case $\alpha<0, \alpha\neq -1$, we employ the discrete power mean inequality for $n=2N, \omega_{2j-1}=\omega_{2j}=\frac{1}{2}|\boldsymbol{h}_j|, x_{2j-1}=(\kappa^h)^\alpha(\rho_j^-), x_{2j}=(\kappa^h)^\alpha(\rho_{j-1}^+),\, \forall\, j=1, 2,\ldots, N,\, p=\frac{\alpha+1}{\alpha}, q=1$, and $x_{2j-1}=(\kappa^h)(\rho_j^-), x_{2j}=(\kappa^h)(\rho_{j-1}^+),\, \forall\, j=1, 2,\ldots, N, p=\alpha{+1}, q=1$, respectively, we obtain
\begin{subequations}
\label{dis alpha<0}
\begin{align}
\label{dis alpha<0 eq1}
    &\qquad \left(\frac{\sum_{j=1}^{N}\frac{1}{2} |\boldsymbol{h}_j| \left((\kappa^h)^{\alpha+1}(\rho_{j}^-)+(\kappa^h)^{\alpha+1}(\rho_{j-1}^+)\right)}{L^h(t)}\right)^{\frac{\alpha}{\alpha+1}}
\\
    &\qquad \leq \frac{\sum_{j=1}^{N}\frac{1}{2} |\boldsymbol{h}_j| \left((\kappa^h)^\alpha(\rho_{j}^-)+(\kappa^h)^\alpha(\rho_{j-1}^+)\right)}{L^h(t)},\nonumber\\
\label{dis alpha<0 eq2}
    &\qquad  \left(\frac{\sum_{j=1}^{N}\frac{1}{2} |\boldsymbol{h}_j| \left((\kappa^h)^{\alpha+1}(\rho_{j}^-)+(\kappa^h)^{\alpha+1}(\rho_{j-1}^+)\right)}{L^h(t)}\right)^{\frac{1}{\alpha+1}}
\\
    &\qquad \leq \frac{\sum_{j=1}^{N}\frac{1}{2} |\boldsymbol{h}_j| \left((\kappa^h)(\rho_{j}^-)+(\kappa^h)(\rho_{j-1}^+)\right)}{L^h(t)}.\nonumber
\end{align}
\end{subequations}
Using \eqref{dis alpha<0 eq1} and \eqref{dis alpha<0 eq2}, by the same argument as \eqref{dis alpha>0 eq3}, we have
\begin{equation}\label{dis alpha<0 eq3}
\frac{1}{L^h(t)}\left((\kappa^{h})^{\alpha},\kappa^{h}\right)_{\Gamma^{h}}^h\leq \frac{1}{L^h(t)}\left((\kappa^{h})^{\alpha},1\right)_{\Gamma^{h}}^h\frac{1}{L^h(t)}(\kappa^{h},1)_{\Gamma^{h}}^h
\end{equation}
Combining \eqref{perimeter-different semi eq2}, \eqref{dis alpha>0 eq3}, \eqref{dis alpha<0 eq3} and $\alpha\,\beta<0$, we obtain
\begin{equation}
\begin{aligned}
\frac{dL^h(t)}{dt}=\beta\left[\left((\kappa^{h})^{\alpha},\kappa^{h}\right)_{\Gamma^{h}}^h-\frac{\left((\kappa^{h})^{\alpha},1\right)_{\Gamma^{h}}^h(1,\kappa^{h})_{\Gamma^{h}}^h}{(1,1)_{\Gamma^{h}}^h}\right]
\leq 0,
\end{aligned}
\end{equation}
which means that the desired perimeter decrease result \eqref{perimeter-preserve, semi} is also valid for $0<|\alpha|\neq 1$. To sum up, the PFEM semi-discrete scheme \eqref{Model,semi,eq} is energy decreasing for all $\alpha$ cases. The proof is completed.
\end{proof}
\color{black}
\section{ A structure-preserving PFEM and its properties}
\setcounter{section}{4} \setcounter{equation}{0}
\sloppy{}
\subsection{ A fully-discrete scheme}
Let $\tau>0$ be the time step size and $t_{m}=m\tau$ be the discrete time levels for $m\geq 0$. Assume that $\Gamma^m:=\boldsymbol{X}^m(\cdot)=\left(x^m(\cdot),y^m(\cdot)\right)^T\in \left[\mathbb{V}^h\right]^2$ and $\kappa^m(\cdot)\in \mathbb{V}^h$ are the numerical approximations of $\Gamma^{h}(t_{m}):=\boldsymbol{X}^h(\cdot,t_{m})$ and $\kappa^h(\cdot,t_m)$, respectively, here $\left(\boldsymbol{X}^h(\cdot,t),\kappa^h(\cdot,t)\right) \in \left[\mathbb{V}^h\right]^2\times \mathbb{V}^h$ is the solution of the semi-discretization \eqref{Model,semi,eq}. Similarly, $\Gamma^m$ is composed by segments $\left\{\boldsymbol{h}_j^m\right\}_{j=1}^N$, which satisfies
\begin{equation}
\underset{1 \leq j \leq N}{\min}|\boldsymbol{h}_j^m|>0\quad \mathrm{with}~~ \boldsymbol{h}_j^m:=\boldsymbol{X}^m(\rho_j)-\boldsymbol{X}^m(\rho_{j-1}),\quad \forall j=1,2,\cdot\cdot\cdot,N,
\end{equation}
and $\kappa^m{(\cdot)}$ satisfies the following assumption:
\begin{equation}\label{assumption on kappa, full}
     \kappa^m(\rho_j)\geq 0~ \forall\, 0<\alpha\neq 1 ~ \text{  and  } ~ \kappa^m(\rho_j)>0 ~\forall\, \alpha<0,\quad \forall j=0,1,\cdot\cdot\cdot,N,~ m\geq 0.
 \end{equation}

And the unit tangential vector $\boldsymbol{\tau}^m$ and the outward unit normal vector $\boldsymbol{n}^m$ of the curve $\Gamma^m$ are constant vectors on each interval $I_j$, which can be computed as
\begin{align}
\boldsymbol{\tau}^m|_{I_j}=\frac{\boldsymbol{h}_j^m}{|\boldsymbol{h}_j^m|}:=\boldsymbol{\tau}_j^m,~~
\boldsymbol{n}^m|_{I_j}=-(\boldsymbol{\tau}_j^m)^\perp=-\frac{(\boldsymbol{h}_j^m)^\perp}{|\boldsymbol{h}_j^m|}:=\boldsymbol{n}_j^m, ~~ \forall j=1,2,\cdot\cdot\cdot,N.
\end{align}

Then a structure-preserving parametric finite element fully discrete scheme (SP-PFEM) based on the backward Euler method in time and a proper approximation of the unit normal vector is constructed as: Given the initial curve $\Gamma^0:=\Gamma^h(0)\in \left[\mathbb{V}^h\right]^2$, find the closed curve $\Gamma^{m+1}:=\boldsymbol{X}^{m+1}(\cdot)=\left(x^{m+1}(\cdot),y^{m+1}(\cdot)\right)^T\in \left[\mathbb{V}^h\right]^2$ and the mean curvature $\kappa^{m+1}(\cdot)\in \mathbb{V}^h$ for $m\geq0$ such that
\begin{align}\label{Model full eq}
\left\{\begin{array}{ll}
\displaystyle \left(\frac{\boldsymbol{X}^{m+1}-\boldsymbol{X}^{m}}{\tau}\cdot\boldsymbol{n}^{m+\frac{1}{2}},\phi^{h}\right)_{\Gamma^{m}}^{h}-\left(\beta(\kappa^{m+1})^{\alpha}-{\lambda}^{m+1},\phi^{h}\right)_{\Gamma^{m}}^{h}=0,~\phi^{h}\in \mathbb{V}^h,\\
\displaystyle \left (\kappa^{m+1}\boldsymbol{n}^{m+\frac{1}{2}},\boldsymbol{\omega}^{h}\right)_{\Gamma^{m}}^{h}-\left(\partial_{s}\boldsymbol{X}^{m+1},\partial_{s}\boldsymbol{\omega}^{h}\right)_{\Gamma^{m}}^{h}=0,~~~~~~~~~~~~~~~~~~~~~~\boldsymbol{\omega}^{h}\in \left[\mathbb{V}^h\right]^2,
\end{array}\right.
\end{align}
where
\begin{align}\label{Model full lag}
\displaystyle{\lambda}^{m+1}=\frac{\left(\beta(\kappa^{m+1})^{\alpha},1\right)_{\Gamma^{m}}^{h}}{(1,1)_{\Gamma^{m}}^{h}},
\end{align}
and $$\boldsymbol{n}^{m+\frac{1}{2}}=-\frac{1}{2}\left(\partial_s\boldsymbol{X}^m(s)+\partial_s\boldsymbol{X}^{m+1}(s)\right)^{\bot}
=-\frac{\left(\partial_\rho\boldsymbol{X}^m(\rho)+\partial_\rho\boldsymbol{X}^{m+1}(\rho)\right)^{\bot}}{2|\partial_\rho\boldsymbol{X}^m(\rho)|}.$$
\subsection{Area conservation and perimeter decrease}

Let $A^m$ be the total enclosed area and
$L^m$ be the perimeter of $\Gamma^m$, which can be written as
\begin{equation}\label{Area-perimeter full}
A^{m}=\frac{1}{2}\sum_{j=1}^{N}(x^{m}_{j}-x^{m}_{j-1})(y^{m}_{j}+y^{m}_{j-1}),\qquad
L^{m}=\sum_{j=1}^{N}|\boldsymbol{h}^{m}_{j}|=(1,1)_{\Gamma^{m}}^{h},\hspace{0.5cm} m\geq0.
\end{equation}
Then we can prove that the fully discrete scheme \eqref{Model full eq} preserves the two geometric structures.
\begin{theorem}[Area conservation and perimeter decrease]\label{Them: Area-perimeter full} \\Let $\left(\boldsymbol{X}^{m+1}(\cdot), \kappa^{m+1}(\cdot)\right)\in \left[\mathbb{V}^h\right]^2\times \mathbb{V}^h$ be a solution of the proposed SP-PFEM \eqref{Model full eq} with the assumption \eqref{assumption on kappa, full} holds on $\kappa^{m+1}(\cdot)$. Then there holds
\begin{align}\label{preserve, full}
A^{m+1}=A^{m}\equiv A^{0},
\qquad \qquad L^{m}\leq L^{m-1}\leq\cdot\cdot\cdot\leq L^{0},
\end{align}
i.e., the area is preserving and the perimeter is decreasing.
\end{theorem}
\begin{proof}\quad
By using a linear interpolation of $\boldsymbol{X}^{m+1}$ and $\boldsymbol{X}^{m}$, we can introduce the approximate solution $\Gamma^h(\vartheta)=\boldsymbol{X}^h(\rho,\vartheta)$ as
\begin{equation}\label{9127}
\Gamma^h(\vartheta)=(1-\vartheta)\boldsymbol{X}^m(\rho)+\vartheta\boldsymbol{X}^{m+1}(\rho), \qquad 0\leq\rho\leq 1,\quad 0\leq\vartheta\leq 1.
\end{equation}
Let $B(\vartheta)$ be the area enclosed by $\Gamma^h(\vartheta)$.
By applying theorem 2.1 in \cite{bao2021structure} and taking $\phi^{h}\equiv1$ in \eqref{Model full eq}, we obtain
\begin{equation}
\begin{aligned}
B(1)-B(0)&=-\int_0^1[\boldsymbol{X}^{m+1}-\boldsymbol{X}^{m}]\cdot\left[\frac{1}{2}\partial_\rho\boldsymbol{X}^m(\rho)+\frac{1}{2}\partial_\rho\boldsymbol{X}^{m+1}(\rho)\right]^{\bot}d\rho
\\&=\left((\boldsymbol{X}^{m+1}-\boldsymbol{X}^{m})\cdot\boldsymbol{n}^{m+\frac{1}{2}},1\right)_{\Gamma^{m}}^h
\\&={\tau\beta}\left[\left((\kappa^{m+1})^{\alpha},1\right)_{\Gamma^{m}}^{h}
-\frac{\left((\kappa^{m+1})^{\alpha},1\right)_{\Gamma^{m}}^{h}(1,1)_{\Gamma^{m}}^{h}}{(1,1)_{\Gamma^{m}}^{h}}\right]
\equiv 0,
\end{aligned}
\end{equation}
which implies $B(1)=B(0)$, i.e., $A^{m+1}=A^{m}$. Thus the first formula in \eqref{preserve, full} is true.

Now we aim to get the perimeter decrease property.
Taking $\phi^{h}=\kappa^{m+1}$ and $\boldsymbol{\omega}^{h}=\boldsymbol{X}^{m+1}-\boldsymbol{X}^{m}$ in \eqref{Model full eq}, we can arrive at
\begin{equation}
\begin{aligned}
&\left(\partial_{s}\boldsymbol{X}^{{m+1}},\partial_{s}(\boldsymbol{X}^{{m+1}}-\boldsymbol{X}^{{m}})\right)_{\Gamma^{m}}^{h}
=\left(\kappa^{m+1}\boldsymbol{n}^{m+\frac{1}{2}},(\boldsymbol{X}^{m+1}-\boldsymbol{X}^{m})\right)_{\Gamma^{m}}^{h}
\\&=\tau\left(\beta(\kappa^{m+1})^{\alpha}-{\lambda}^{m+1},{\kappa^{m+1}}\right)_{\Gamma^{m}}^{h}\\
&={\tau\beta}\left[\left((\kappa^{m+1})^{\alpha},{\kappa^{m+1}}\right)_{\Gamma^{m}}^{h}
-\frac{\left((\kappa^{m+1})^{\alpha},1\right)_{\Gamma^{m}}^{h}({\kappa^{m+1}},1)_{\Gamma^{m}}^{h}}{(1,1)_{\Gamma^{m}}^{h}}\right].
\end{aligned}
\end{equation}
By using a similar argument as the proof of the perimeter decrease part in theorem \ref{Them: Area-perimeter semi}, we know that
\begin{equation}\label{perimeter full eq1}
\begin{aligned}
&\left(\partial_{s}\boldsymbol{X}^{{m+1}},\partial_{s}(\boldsymbol{X}^{{m+1}}-\boldsymbol{X}^{{m}})\right)_{\Gamma^{m}}^{h}\\
&={\tau\beta}\left[\left((\kappa^{m+1})^{\alpha},{\kappa^{m+1}}\right)_{\Gamma^{m}}^{h}-\frac{\left((\kappa^{m+1})^{\alpha},1\right)_{\Gamma^{m}}^{h}({\kappa^{m+1}},1)_{\Gamma^{m}}^{h}}{(1,1)_{\Gamma^{m}}^{h}}\right]
\leq0.
\end{aligned}
\end{equation}
On the other hand, since
$\boldsymbol{u}\cdot\boldsymbol{v}\leq \frac{1}{2}(|\boldsymbol{u}|^2+|\boldsymbol{v}|^2), \,\forall\, \boldsymbol{u}, \boldsymbol{v}\in \mathbb{R}^2$, there holds
\begin{equation}\label{perimeter full eq2}
\begin{aligned}
&\left(\partial_{s}\boldsymbol{X}^{{m+1}},\partial_{s}(\boldsymbol{X}^{{m+1}}-\boldsymbol{X}^{{m}})\right)_{\Gamma^{m}}^{h}\\
&=\sum_{j=1}^{N}|\boldsymbol{h}_{j}^{m}|\frac{\boldsymbol{h}_{j}^{m+1}}{|\boldsymbol{h}_{j}^{m}|}\cdot\frac{\boldsymbol{h}_{j}^{m+1}-\boldsymbol{h}_{j}^{m}}{|\boldsymbol{h}_{j}^{m}|}\\
&\geq\sum_{j=1}^{N}\frac{1}{2}\left(|\frac{\boldsymbol{h}_{j}^{m+1}}{\boldsymbol{h}_{j}^{m}}|^{2}-1\right)|\boldsymbol{h}_{j}^{m}|
\geq\sum_{j=1}^{N}\left(\frac{|\boldsymbol{h}_{j}^{m+1}|}{|\boldsymbol{h}_{j}^{m}|}-1\right)|\boldsymbol{h}_{j}^{m}|\\
&=\sum_{j=1}^{N}\left(|\boldsymbol{h}_{j}^{m+1}|-|\boldsymbol{h}_{j}^{m}|\right)=L^{m+1}-L^{m}.
\end{aligned}
\end{equation}
Then the desired second result in \eqref{preserve, full} follows from \eqref{perimeter full eq1} and \eqref{perimeter full eq2} immediately. Therefore, the fully-discrete scheme \eqref{Model full eq} is structure-preserving. The proof is completed.
\end{proof}
\subsection{The iterative solver}
To solve the nonlinear system \eqref{Model full eq}, the
Newton's iterative method is adopted to derive $\left(\boldsymbol{X}^{m+1}(\cdot), \kappa^{m+1}(\cdot)\right)\in \left[\mathbb{V}^h\right]^2\times \mathbb{V}^h$.
%By using the first-order Taylor
%expansion of the nonlinear system \eqref{Model full eq}) at the point $(\boldsymbol{X}^{m+1,i},\kappa^{m+1,i})$ and setting
%$\boldsymbol{X}^{m+1,i+1}=\boldsymbol{X}^{m+1,i}+\boldsymbol{X}^{\delta}$
In the $i-$th iteration,
given $\left(\boldsymbol{X}^{m+1,i}(\cdot),\kappa^{m+1,i}(\cdot)\right)\in \left[\mathbb{V}^h\right]^2\times \mathbb{V}^h$, for any $(\boldsymbol{\omega}^{h},\phi^{h})\in \left[\mathbb{V}^h\right]^2\times \mathbb{V}^h$, the Newton direction $\left(\boldsymbol{X}^{\delta}(\cdot),\kappa^{\delta}(\cdot)\right)\in \left[\mathbb{V}^h\right]^2\times \mathbb{V}^h$ can be obtained by
\begin{subequations}
\begin{align}
\label{Newton:eq1}
&\displaystyle\left(\frac{\boldsymbol{X}^{\delta}}{\tau}\cdot\boldsymbol{n}^{m+\frac{1}{2},i},\phi^{h}\right)_{\Gamma^{m}}^{h}
+\left(\frac{\boldsymbol{X}^{m+1,i}-\boldsymbol{X}^{m}}{\tau}\cdot \frac{\left(-\partial_\rho\boldsymbol{X}^{\delta}\right)^{\bot}}{2|\partial_\rho\boldsymbol{X}^m|},\phi^{h}\right)_{\Gamma^{m}}^{h}
\\
&-\alpha\beta\left((\kappa^{m+1,i})^{\alpha-1}\kappa^{\delta},\phi^{h}\right)_{\Gamma^{m}}^{h}
+\displaystyle{\alpha\beta\left(\frac{\left((\kappa^{m+1,i})^{\alpha-1}\kappa^{\delta},1\right)_{\Gamma^{m}}^{h}}{(1,1)_{\Gamma^{m}}^{h}},\phi^{h}\right)_{\Gamma^{m}}^{h}}\nonumber\\
&=\displaystyle-\left(\frac{\boldsymbol{X}^{m+1,i}-\boldsymbol{X}^{m}}{\tau}\cdot\boldsymbol{n}^{m+\frac{1}{2},i},\phi^{h}\right)_{\Gamma^{m}}^{h}
+\beta\left((\kappa^{m+1,i})^{\alpha}-\frac{\left((\kappa^{m+1,i})^{\alpha},1\right)_{\Gamma^{m}}^{h}}{(1,1)_{\Gamma^{m}}^{h}},\phi^{h}\right)_{\Gamma^{m}}^{h},\nonumber\\
\label{Newton:eq2}
&\displaystyle\left(\kappa^{\delta}\boldsymbol{n}^{m+\frac{1}{2},i},\boldsymbol{\omega}^{h}\right)_{\Gamma^{m}}^{h}
+\left(\kappa^{m+1,i}\frac{\left(-\partial_\rho\boldsymbol{X}^{\delta}\right)^{\bot}}{2|\partial_\rho\boldsymbol{X}^m|},\boldsymbol{\omega}^{h}\right)_{\Gamma^{m}}^{h}
-\left(\partial_{s}\boldsymbol{X}^{\delta},\partial_{s}\boldsymbol{\omega}^{h}\right)_{\Gamma^{m}}^{h}\\
&=-\left(\kappa^{m+1,i}\boldsymbol{n}^{m+\frac{1}{2},i},\boldsymbol{\omega}^{h}\right)_{\Gamma^{m}}^{h}
+\left(\partial_{s}\boldsymbol{X}^{m+1,i},\partial_{s}\boldsymbol{\omega}^{h}\right)_{\Gamma^{m}}^{h},\nonumber
\end{align}
\end{subequations}
where
$$\boldsymbol{n}^{m+\frac{1}{2},i}=-\frac{1}{2|\partial_\rho\boldsymbol{X}^m|}\left(\partial_\rho\boldsymbol{X}^{m}+\partial_\rho\boldsymbol{X}^{m+1,i}\right)^{\bot}.$$
For each $m\geq 0$, we set
$\boldsymbol{X}^{m+1,i+1}=\boldsymbol{X}^{m+1,i}+\boldsymbol{X}^{\delta}$, $\kappa^{m+1,i+1}=\kappa^{m+1,i}+\kappa^{\delta}$ and
typically choose the initial guess $\boldsymbol{X}^{m+1,0}=\boldsymbol{X}^{m}, \kappa^{m+1,0}=\kappa^{m}$. Then we repeat the iteration \eqref{Newton:eq1}--\eqref{Newton:eq2} until $$\max_{0\leq j\leq N}\left(|\boldsymbol{X}^{m+1,i+1}(\rho_j)-\boldsymbol{X}^{m+1,i}(\rho_j)|+|\kappa^{m+1,i+1}(\rho_j)-\kappa^{m+1,i}(\rho_j)|\right)\leq tol,$$ where $tol$ is the chosen tolerance.

The Picard iteration method can also be taken as an alternative solver. In the $i$-th iteration, given $\kappa^{m+1,i}(\cdot)\in\mathbb{V}^h$, for any $(\boldsymbol{\omega}^{h},\phi^{h})\in \left[\mathbb{V}^h\right]^2\times \mathbb{V}^h$, we seek $\left(\boldsymbol{X}^{m+1,i+1}(\cdot),\kappa^{m+1,i+1}(\cdot)\right)\in \left[\mathbb{V}^h\right]^2\times \mathbb{V}^h$ such that
%\begin{align}\label{eq: Picard}
%\left\{\begin{array}{ll}
%\displaystyle \left(\frac{\boldsymbol{X}^{m+1,i+1}}{\tau}\cdot\boldsymbol{n}^{m+\frac{1}{2},i},\phi^{h}\right)_{\Gamma^{m}}^{h}
%-\beta\left((\kappa^{m+1,i})^{\alpha-1}\kappa^{m+1,i+1},\phi^{h}\right)_{\Gamma^{m}}^{h}\\[3mm]
%\quad\displaystyle+{\beta\left(\frac{\left((\kappa^{m+1,i})^{\alpha-1}\kappa^{m+1,i+1},1\right)_{\Gamma^{m}}^{h}}{(1,1)_{\Gamma^{m}}^{h}},\phi^{h}\right)_{\Gamma^{m}}^{h}}
%=\displaystyle\left(\frac{\boldsymbol{X}^{m}}{\tau}\cdot\boldsymbol{n}^{m+\frac{1}{2},i},\phi^{h}\right)_{\Gamma^{m}}^{h},\\[3mm]
%\displaystyle\left(\kappa^{m+1,i+1}\boldsymbol{n}^{m+\frac{1}{2},i},\boldsymbol{\omega}^{h}\right)_{\Gamma^{m}}^{h}
%-\left(\partial_{s}\boldsymbol{X}^{m+1,i+1},\partial_{s}\boldsymbol{\omega}^{h}\right)_{\Gamma^{m}}^{h}=0.
%\end{array}\right.
%\end{align}
\begin{subequations}
\begin{align}
\label{Picard:eq1}
&\displaystyle \left(\frac{\boldsymbol{X}^{m+1,i+1}}{\tau}\cdot\boldsymbol{n}^{m+\frac{1}{2},i},\phi^{h}\right)_{\Gamma^{m}}^{h}
-\beta\left((\kappa^{m+1,i})^{\alpha-1}\kappa^{m+1,i+1},\phi^{h}\right)_{\Gamma^{m}}^{h}\\
&\quad\displaystyle+{\beta\left(\frac{\left((\kappa^{m+1,i})^{\alpha-1}\kappa^{m+1,i+1},1\right)_{\Gamma^{m}}^{h}}{(1,1)_{\Gamma^{m}}^{h}},\phi^{h}\right)_{\Gamma^{m}}^{h}}
=\displaystyle\left(\frac{\boldsymbol{X}^{m}}{\tau}\cdot\boldsymbol{n}^{m+\frac{1}{2},i},\phi^{h}\right)_{\Gamma^{m}}^{h},\nonumber\\
\label{Picard:eq2}
&\displaystyle\left(\kappa^{m+1,i+1}\boldsymbol{n}^{m+\frac{1}{2},i},\boldsymbol{\omega}^{h}\right)_{\Gamma^{m}}^{h}
-\left(\partial_{s}\boldsymbol{X}^{m+1,i+1},\partial_{s}\boldsymbol{\omega}^{h}\right)_{\Gamma^{m}}^{h}=0.
\end{align}
\end{subequations}

\section{Numerical results}
\setcounter{section}{5} \setcounter{equation}{0}
\sloppy{}
\par
In this section, numerical experiments are carried out to verify the performance of the proposed SP-PFEM \eqref{Model full eq}. We will investigate the convergent rates and area conservation, examine the perimeter decrease property and mesh quality, and simulate the morphological evolutions of closed curves.

In order to measure the numerical error between $\Gamma^m$ and $\Gamma(t_m)$, we adopt the manifold distance defined in \cite{zhao2021energy} as
$M(\Gamma_1,\Gamma_2)=|(\Omega_1\setminus\Omega_2)\cup(\Omega_2\setminus\Omega_1)|=|\Omega_1|+|\Omega_2|-2|\Omega_1\cap\Omega_2|,$
where $\Gamma_1$ and $\Gamma_2$ denote any two closed curves, $\Omega_1$ and $\Omega_2$ are the regions enclosed by $\Gamma_1$ and $\Gamma_2$, respectively, and $|\Omega|$ denotes the area of $\Omega$. Then the numerical error $e^h(t_m)$ is defined by
$$e^h(t)\mid_{t=t_m}:=M(\Gamma^m,\Gamma(t=t_m)),~~m\geq0.$$

To further verify the properties of the proposed SP-PFEM \eqref{Model full eq}, we introduce the relative area loss function $\frac{\Delta A^h(t)}{A^h(0)}$, the normalized perimeter function $\frac{W^h(t)}{W^h(0)}$ and the mesh ratio $R^h(t)$ as
$$
\frac{\Delta A^h(t)}{A^h(0)}\mid_{t=t_m}:=\frac{ A^m-A^0}{A^0},\qquad \frac{W^h(t)}{W^h(0)}\mid_{t=t_m}:=\frac{L^m}{L^0},\qquad R^h(t)\mid_{t=t_m}:=\frac{h^m_{\max}}{h^m_{\min}},$$\color{black}
where $A^m$ and $L^m$ have been given in \eqref{Area-perimeter full}, $h^m_{\max}=\underset{1 \leq j \leq N}{\max}|\boldsymbol{h}_j^m|$, $h^m_{\min}=\underset{1 \leq j \leq N}{\min}|\boldsymbol{h}_j^m|,$ and $R^h(t)$ can be used to measure the mesh
quality.

In the following numerical simulations, since formally the proposed SP-PFEM \eqref{Model full eq} is first-order accuracy in temporal direction and second-order accuracy in spatial direction, we can take the size parameters as $\tau=\mathcal{{O}}(h^2)$, e.g. $\tau=h^2$, except where noted. In view of the fact that the exact solution $\Gamma(t)$ is usually not available, we replace $\Gamma(t=t_m)$ with the numerical approximation based on very small size parameters $h=2^{-8}$ and $\tau=2^{-16}$. In addition, at each time step, the nonlinear system \eqref{Model full eq} is solved by the Newton's iteration method \eqref{Newton:eq1}--\eqref{Newton:eq2}, where the tolerance is chosen as $10^{-12}$.

\subsection{Convergent rates, area conservation, perimeter decrease and mesh quality}
\par
In this subsection, the initial shape is taken as an ellipse curve given by $\frac{x^2}{{3^2}}+y^2=1$ and the parameter $\beta$ in the area-conserved generalized mean curvature flow \eqref{Model,cont,eq}--\eqref{Model,cont,lag} is chosen as $|\beta|=1$ satisfying $\alpha\beta<0$.

Figure \ref{fig:error} plots the spatial convergence rates of the proposed SP-PFEM \eqref{Model full eq} at different times $t$ for different parameters $\alpha$ and $\beta$. We find the order of convergence can reach about $2$ in spatial discretization.
\begin{figure}[!ht]
\centering
\includegraphics[width=.32\textwidth]{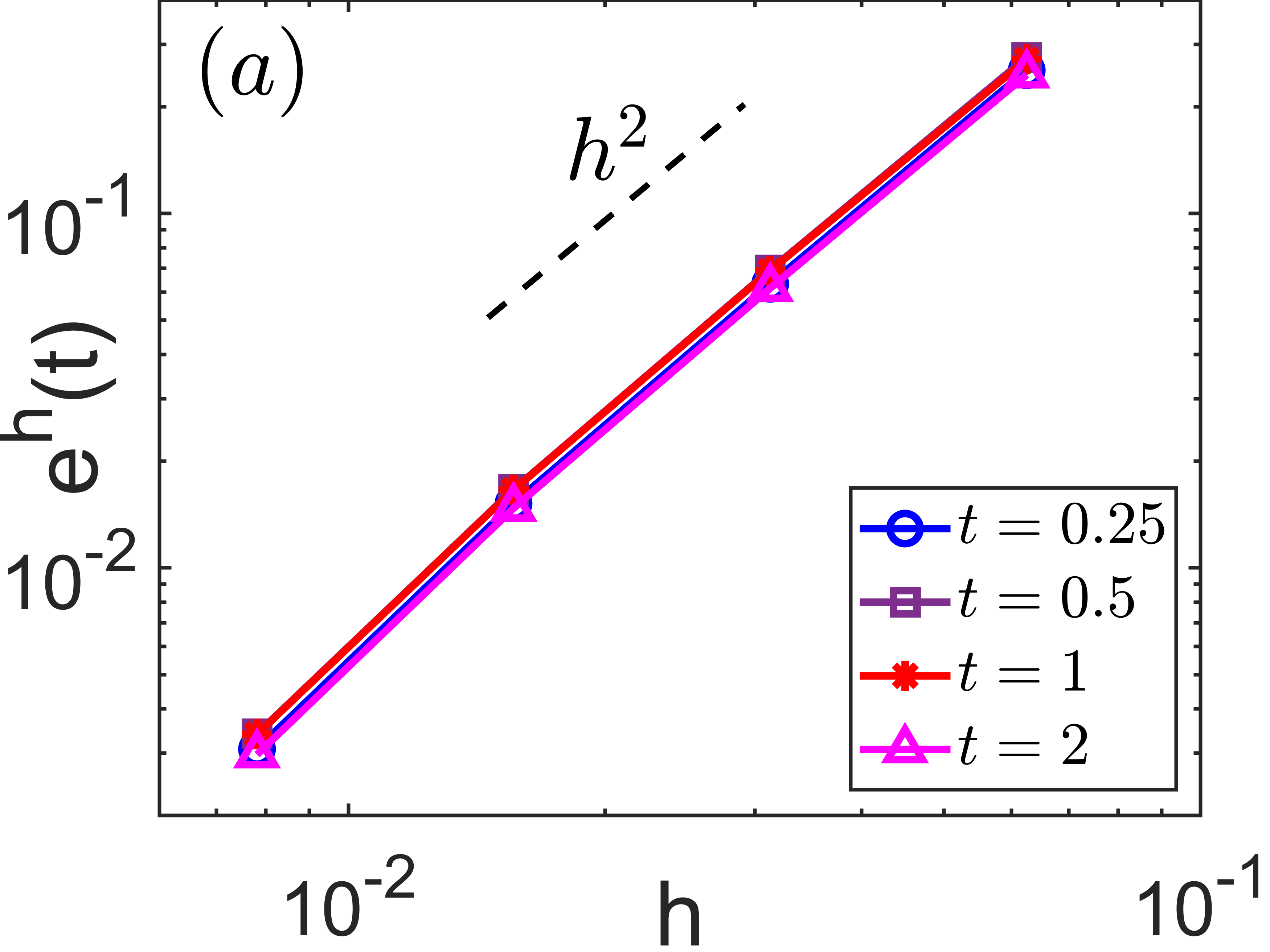}
\includegraphics[width=.32\textwidth]{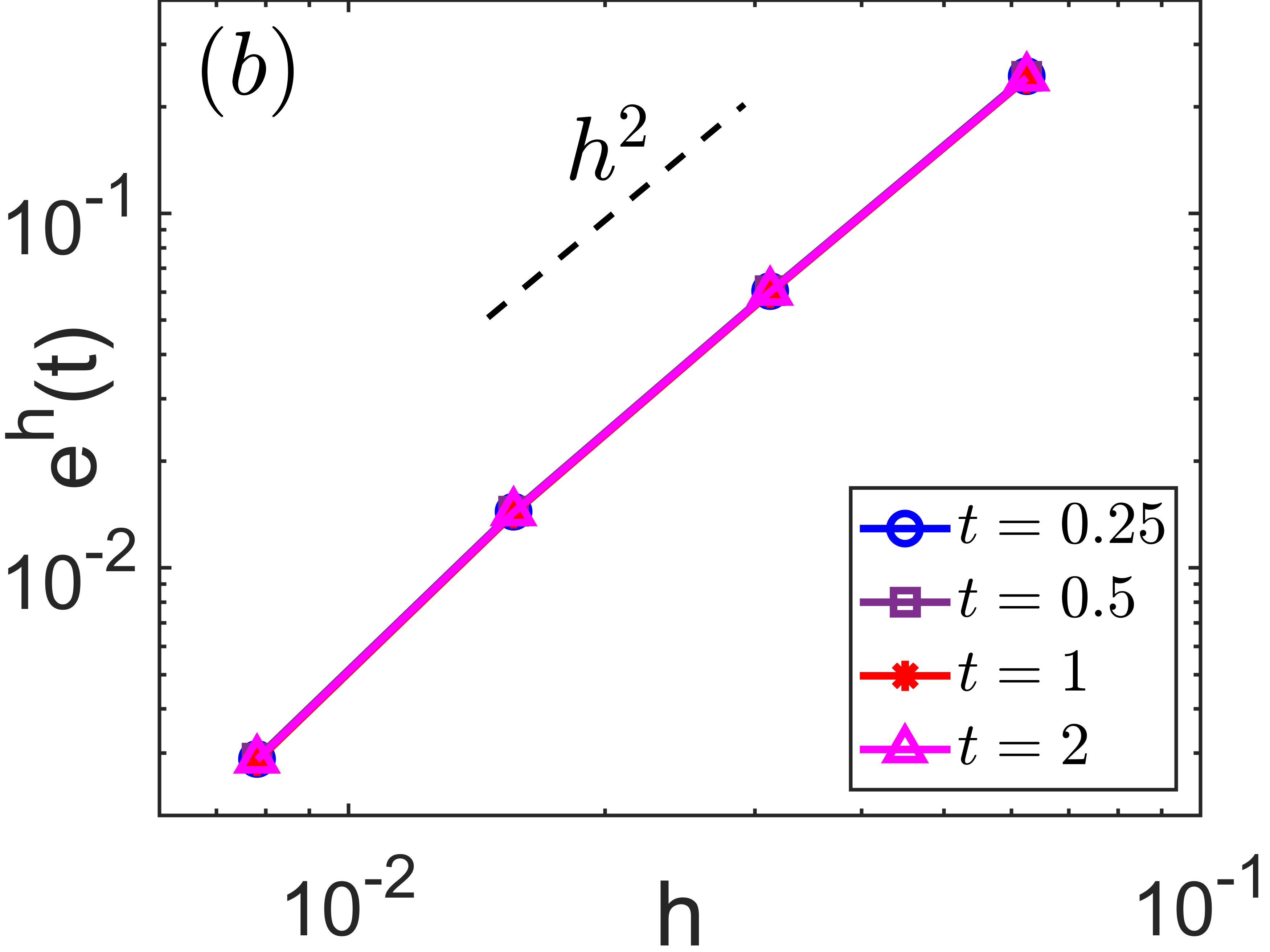}
\includegraphics[width=.32\textwidth]{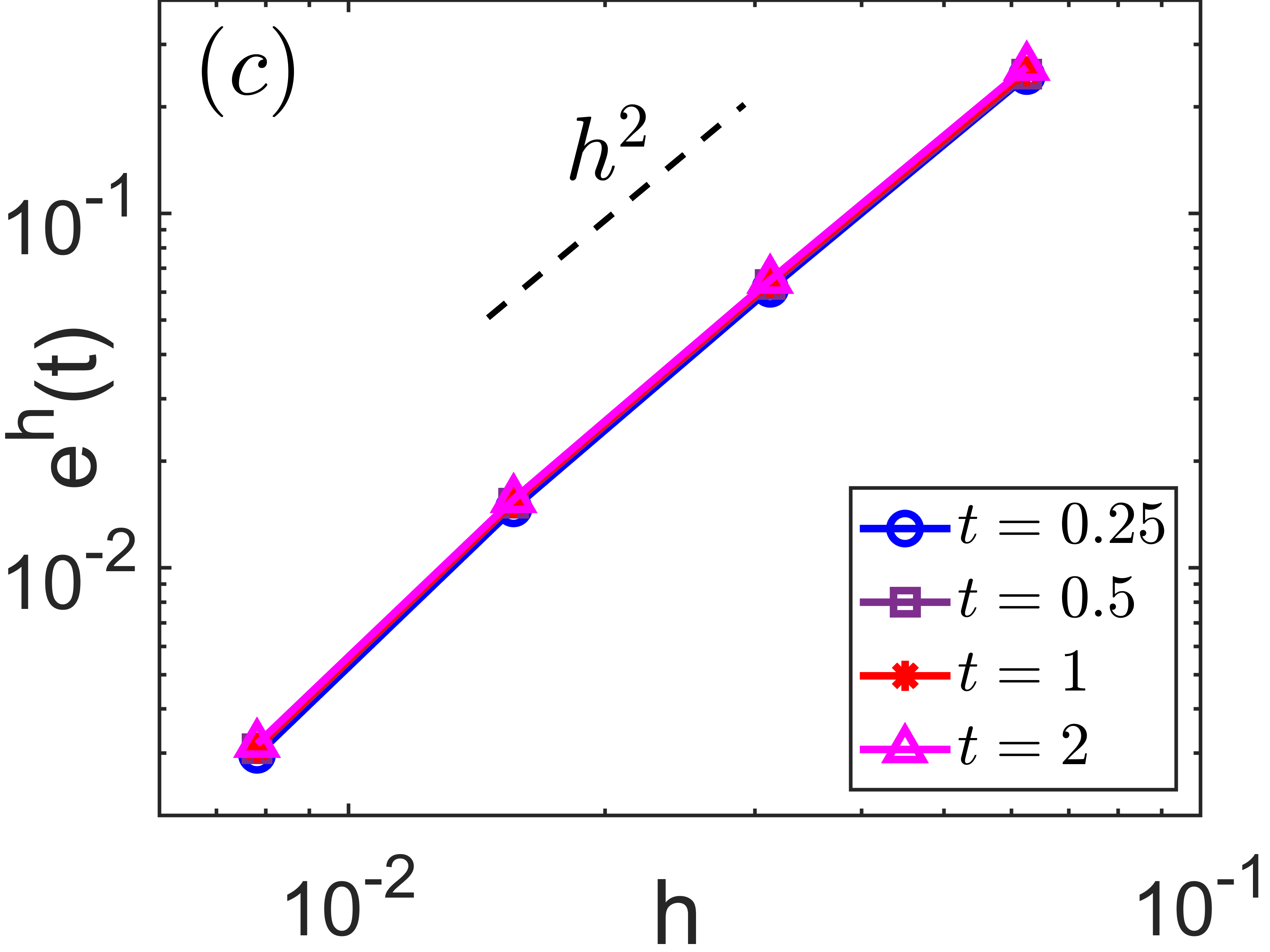}\\
 \vspace{0.5cm}
\includegraphics[width=.32\textwidth]{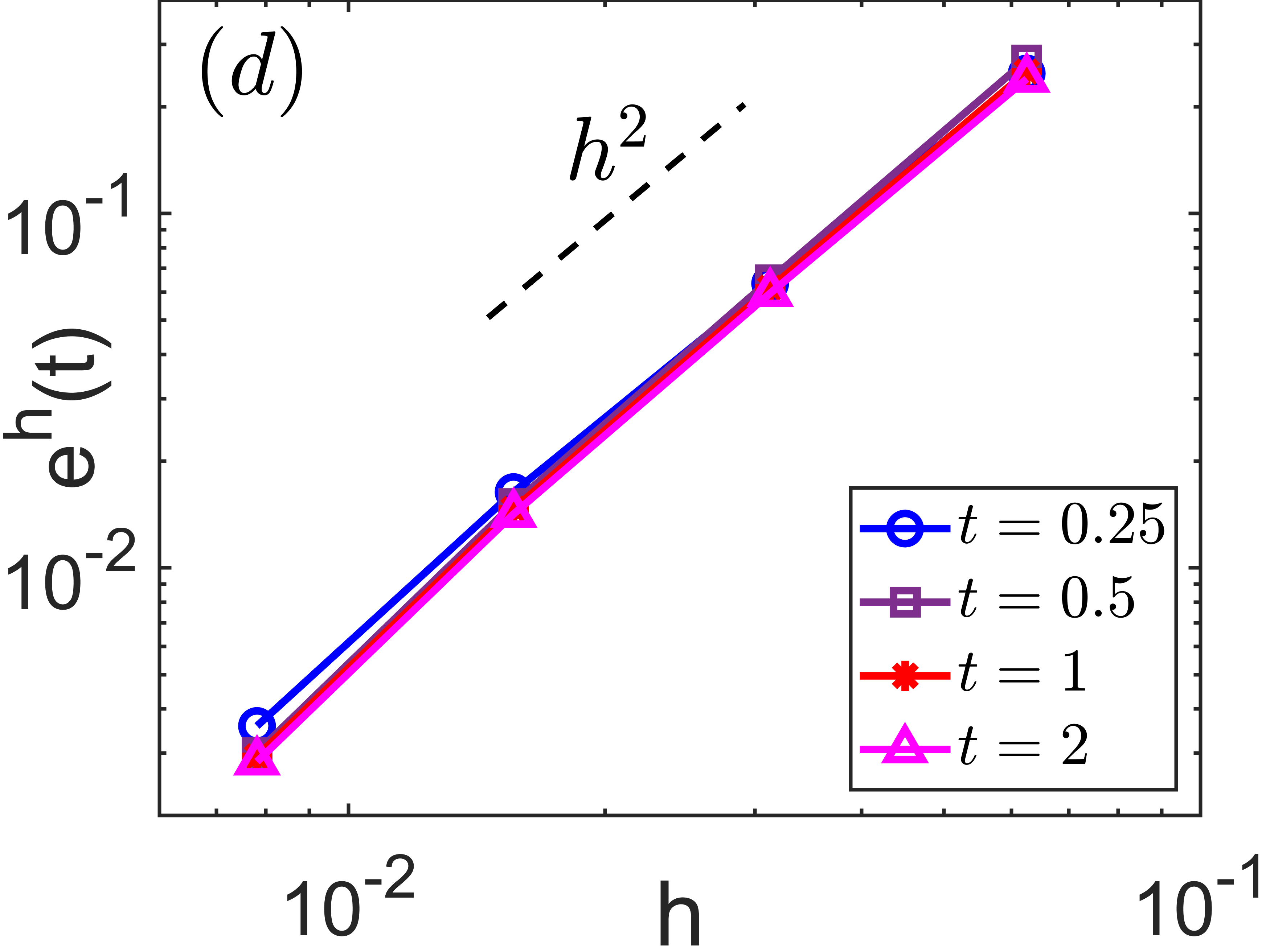}
\includegraphics[width=.32\textwidth]{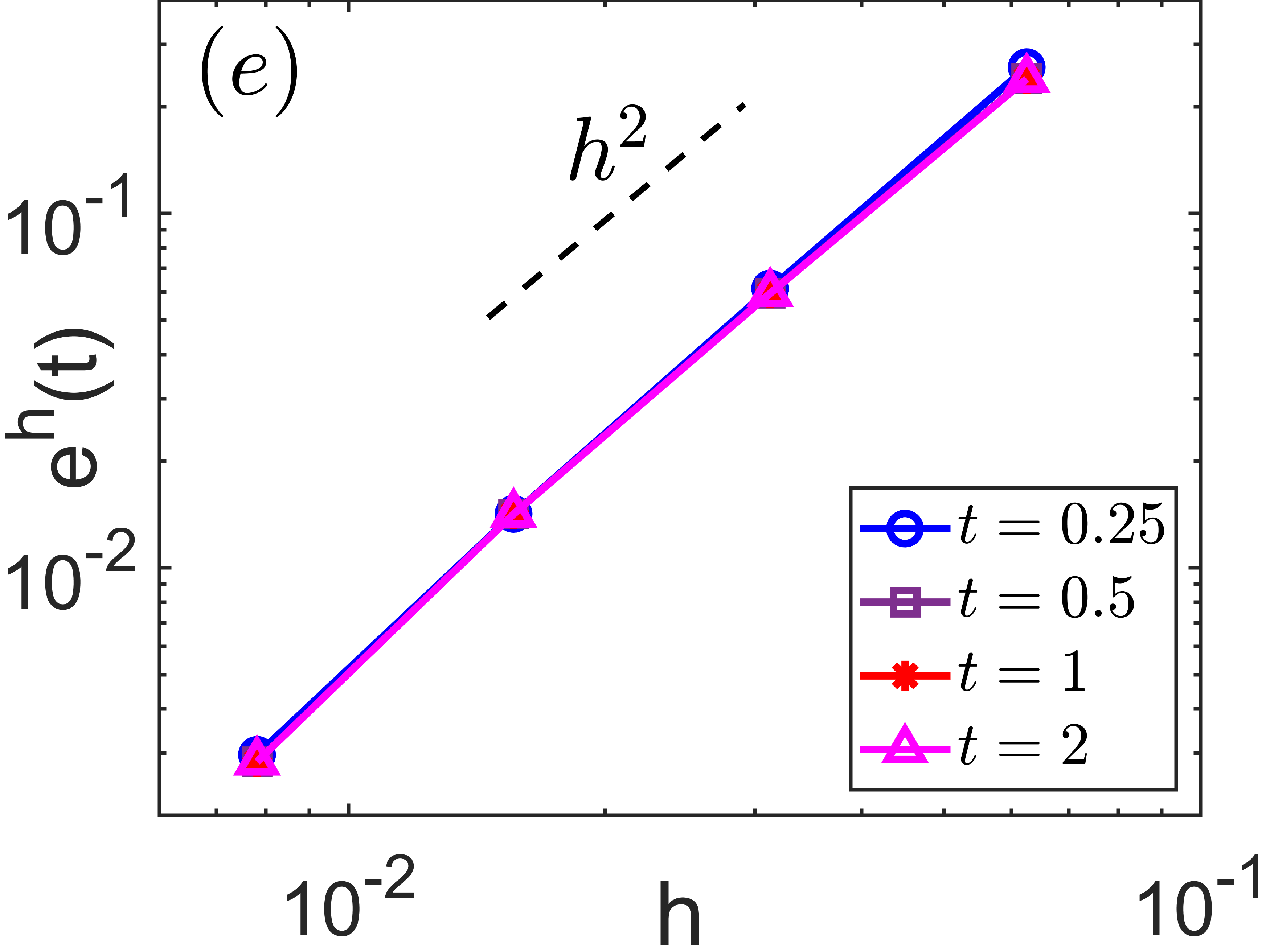}
\includegraphics[width=.32\textwidth]{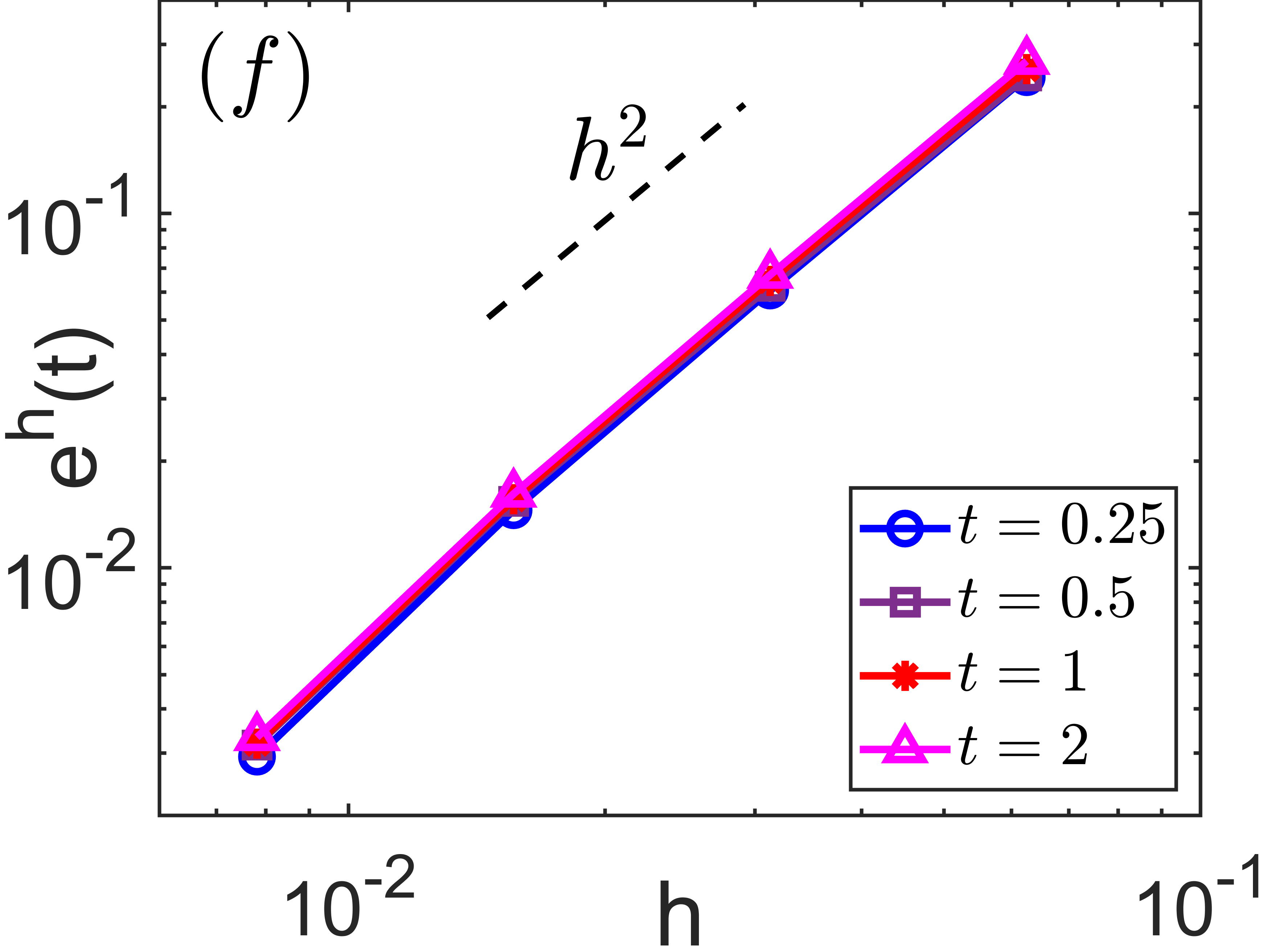}
\caption{ {Spatial convergence rates of the SP-PFEM \eqref{Model full eq} with $\tau=h^2$ at different times $t$ for different parameters $\alpha$ and $\beta$: (a) $\alpha=1$, $\beta=-1$; (b) $\alpha=2$, $\beta=-1$; (c) $\alpha=1/3$, $\beta=-1$; (d) $\alpha=-1$, $\beta=1$; (e) $\alpha=-2$, $\beta=1$; (f) $\alpha=-1/3$, $\beta=1$.}}
\label{fig:error}
\end{figure}

Figure \ref{fig:Ah} shows the time evolutions of the relative area loss $\frac{\Delta A^h(t)}{A^h(0)}$
for different parameters $\alpha$ and $\beta$ with fixed $h=2^{-5}$ and $\tau=\frac{2}{25}h^2$, and depicts the iteration numbers of Newton's iterative \eqref{Newton:eq1}--\eqref{Newton:eq2} simultaneously. Numerical results confirm the area conservation property of the proposed SP-PFEM \eqref{Model full eq} in theorem \ref{Them: Area-perimeter full} for the order of magnitude of the relative area loss is at around $10^{-15}$, which is close to the machine epsilon at around $10^{-16}$. In addition, we observe that the number of Newton's iteration at each time step is around 1 to 3, which implies that it is efficient.

\begin{figure}[!htb]
\centering
\includegraphics[width=.32\textwidth]{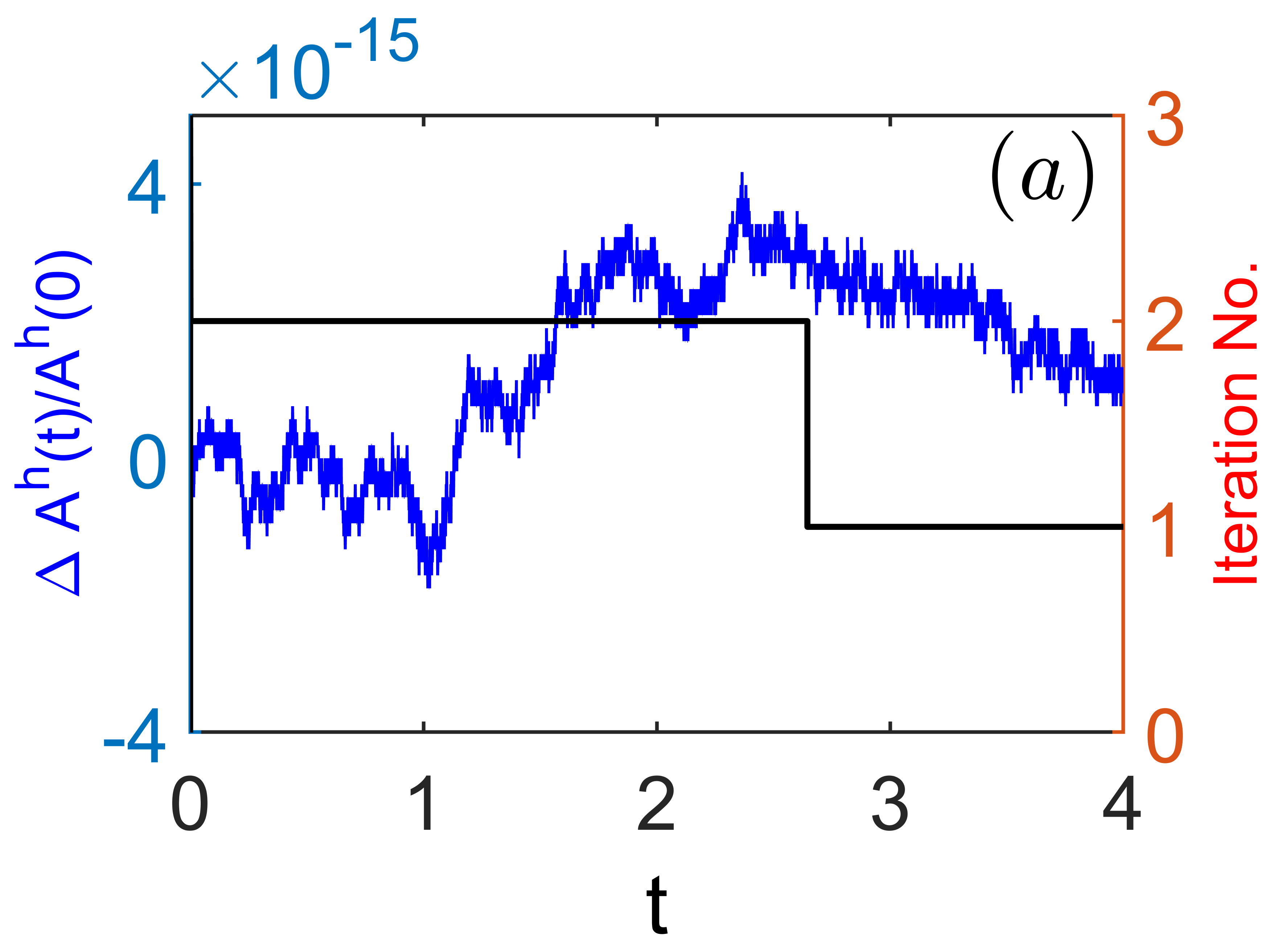}
\includegraphics[width=.32\textwidth]{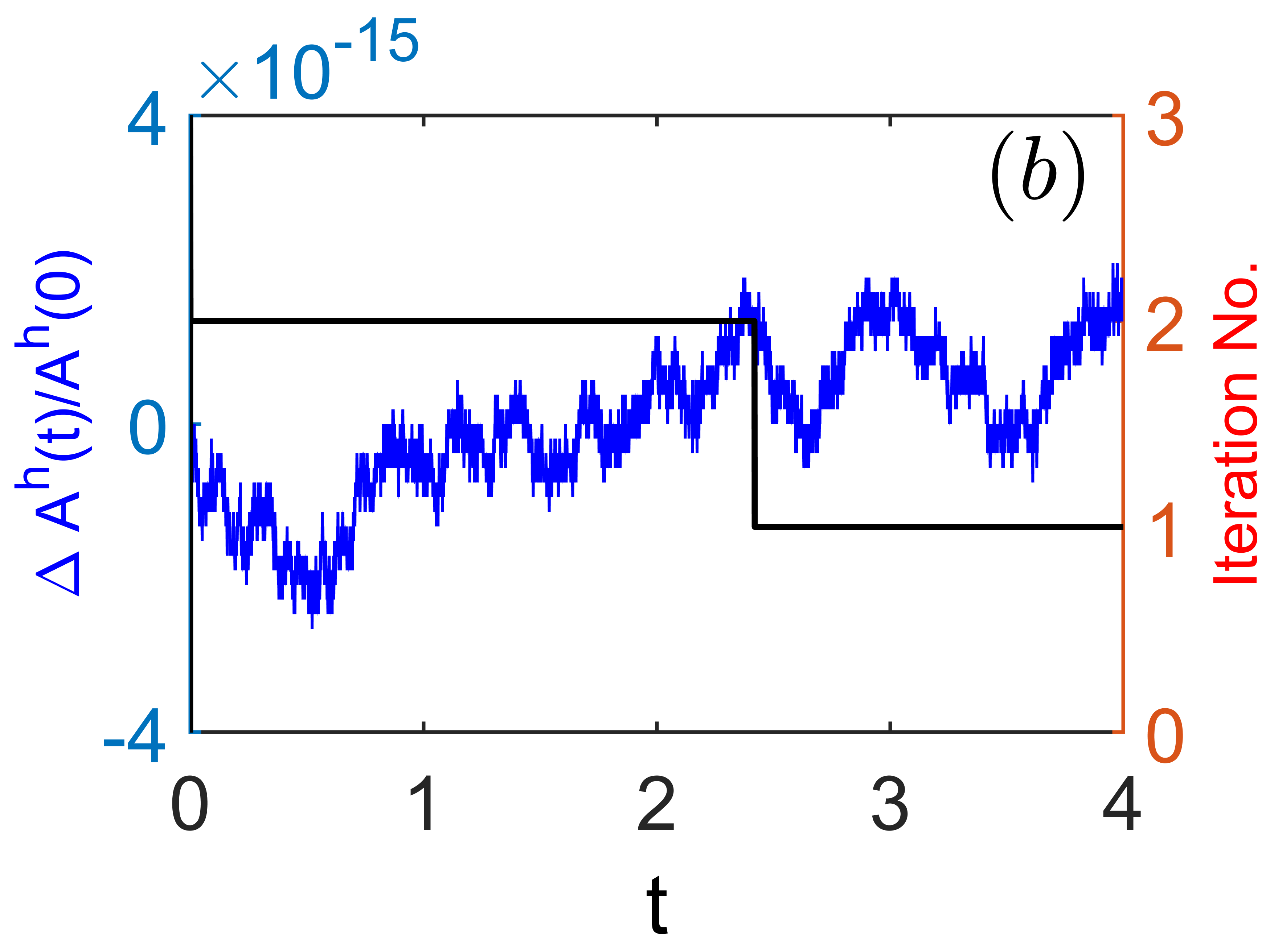}
\includegraphics[width=.32\textwidth]{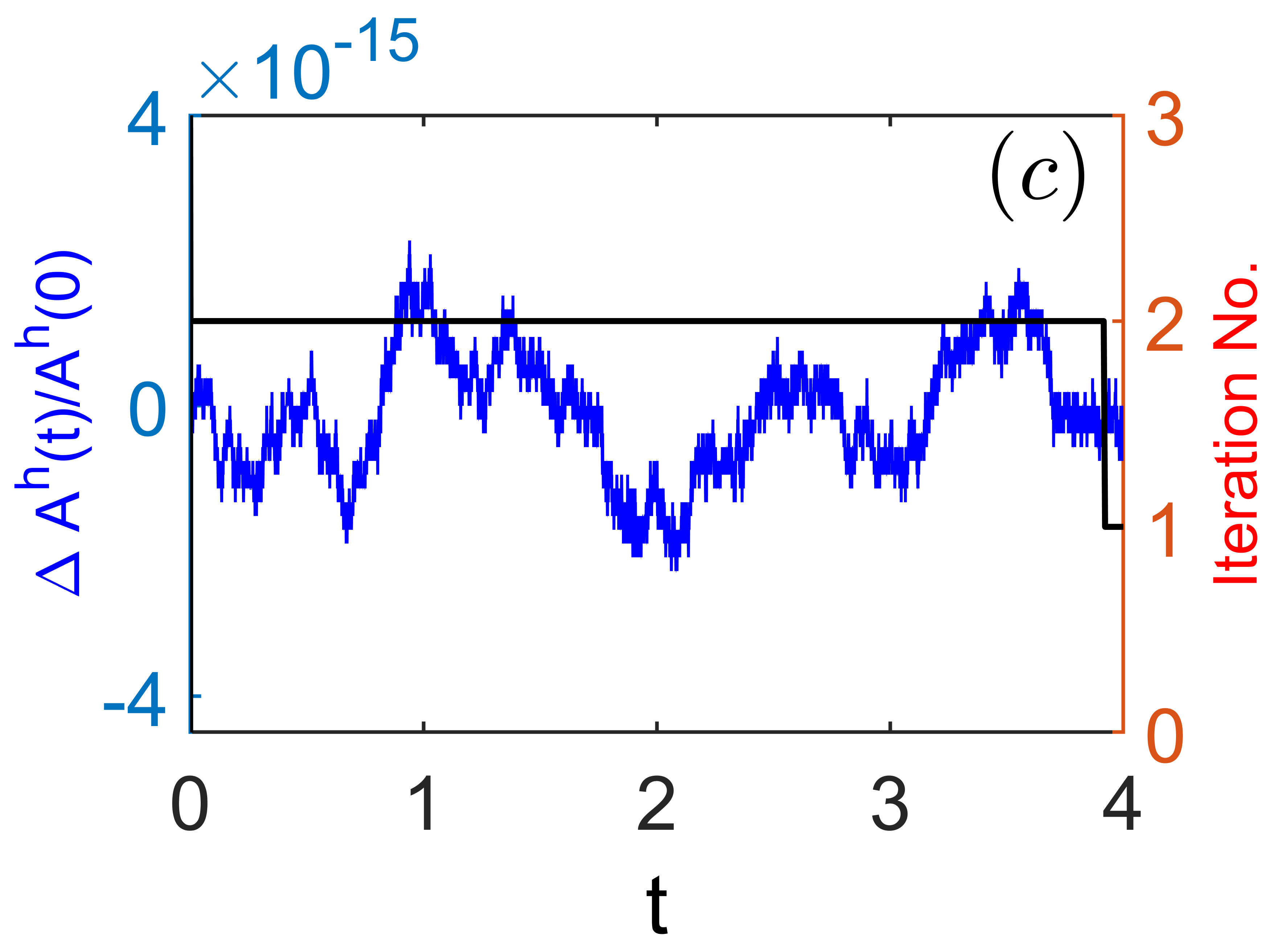}\\
 \vspace{0.1cm}
\includegraphics[width=.32\textwidth]{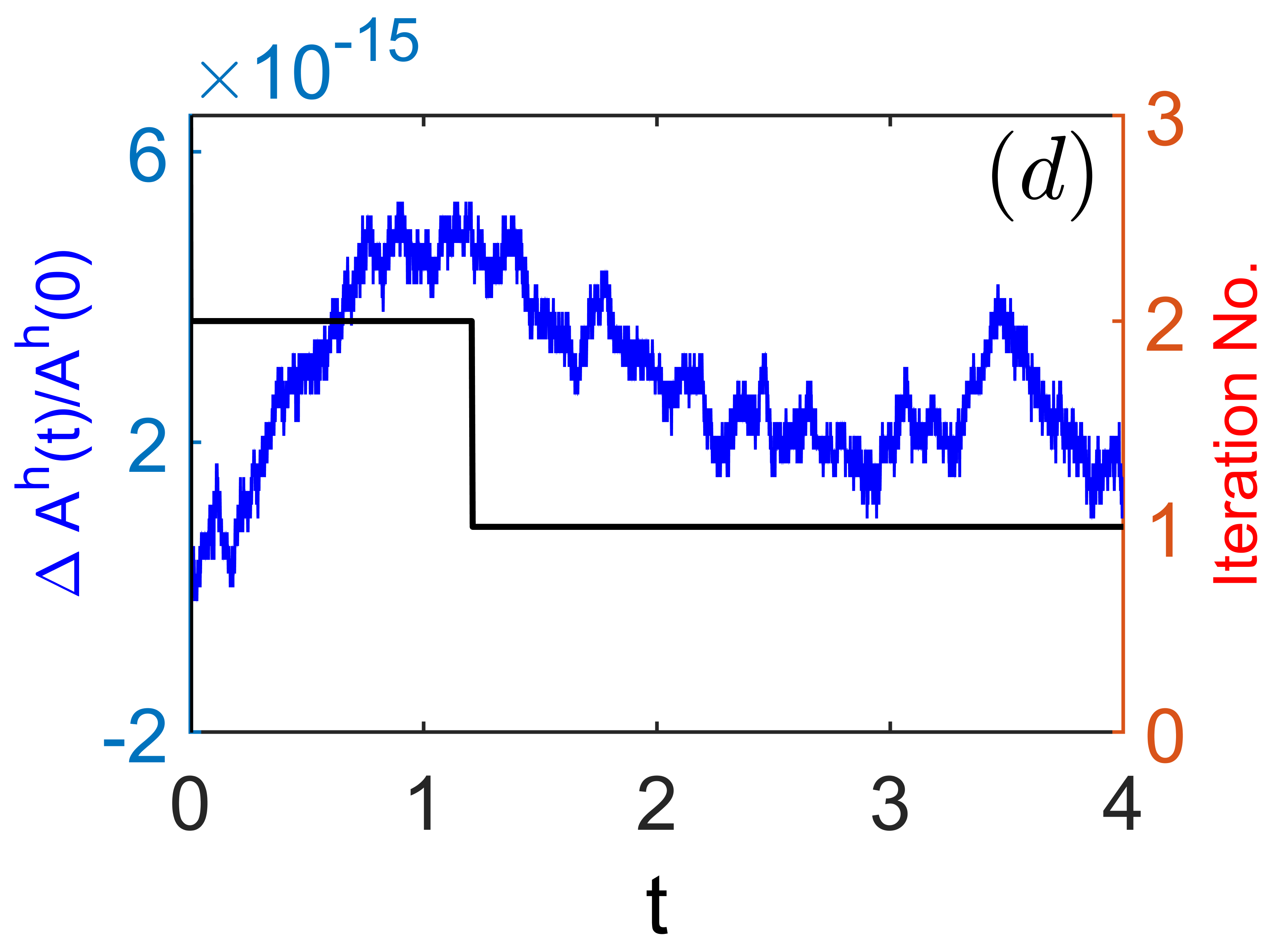}
\includegraphics[width=.32\textwidth]{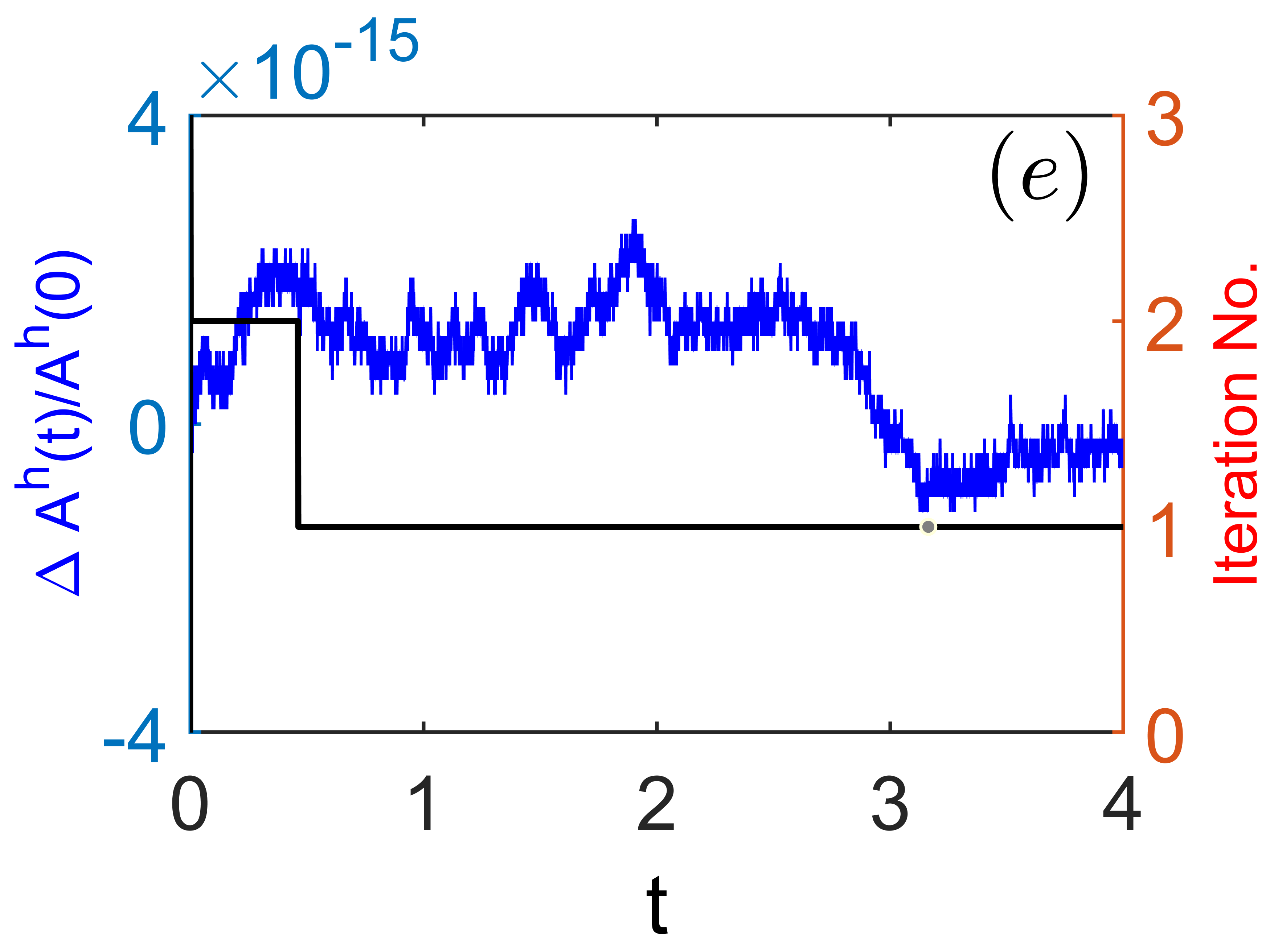}
\includegraphics[width=.32\textwidth]{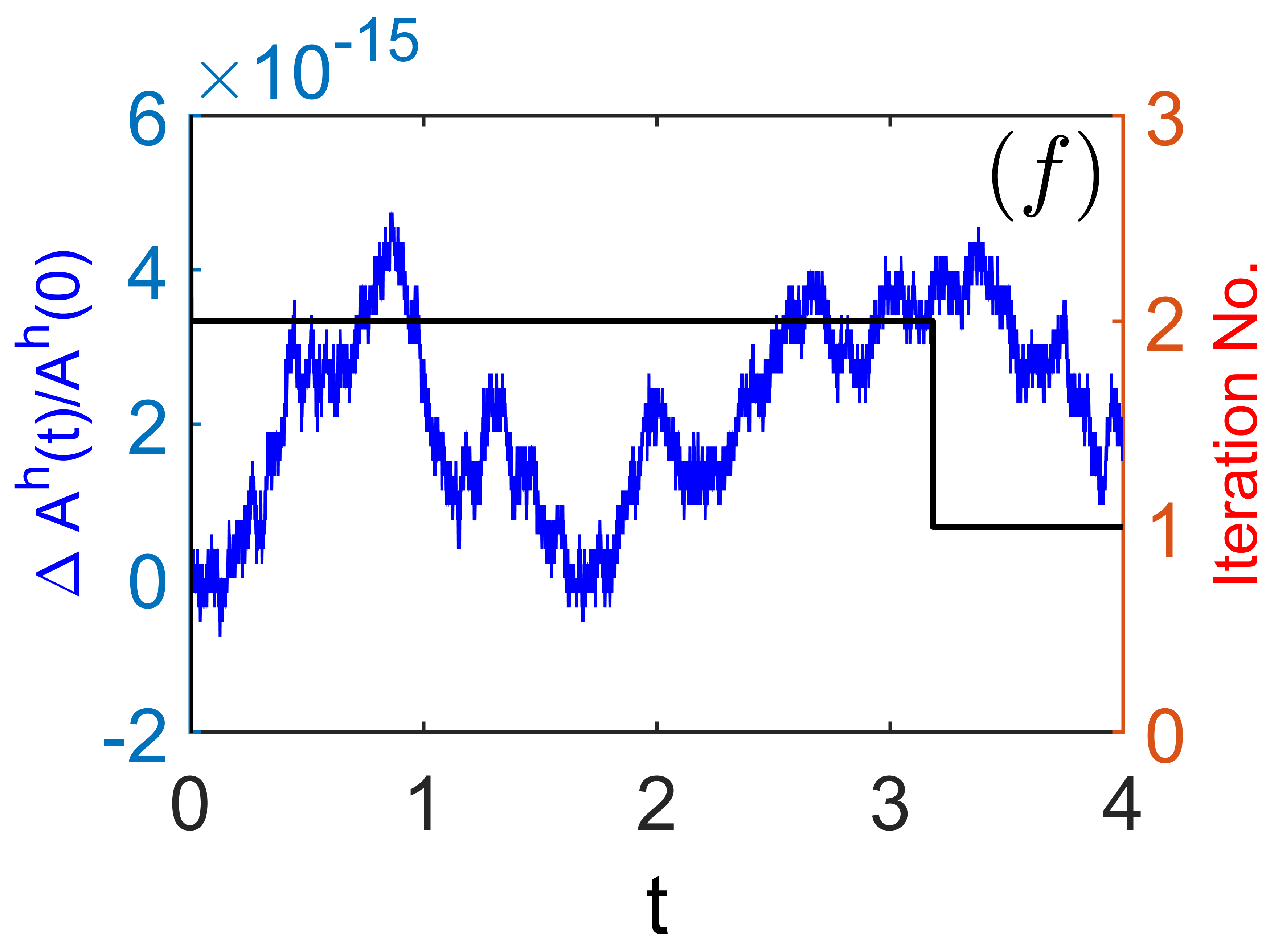}
\caption{{ Time evolutions of the normalized area loss $\frac{\Delta A^h(t)}{A^h(0)}$ (blue line) and iteration number (black line) of the SP-PFEM \eqref{Model full eq} for fixed $h=2^{-5}$ and $\tau=\frac{2}{25}h^2$, where we choose the following parameters $\alpha$ and $\beta$: (a) $\alpha=1$, $\beta=-1$; (b) $\alpha=2$, $\beta=-1$; (c) $\alpha=1/3$, $\beta=-1$; (d) $\alpha=-1$, $\beta=1$; (e) $\alpha=-2$, $\beta=1$; (f) $\alpha=-1/3$, $\beta=1$.}}
\label{fig:Ah}
\end{figure}

To examine the unconditionally perimeter decrease, we show the normalized perimeter $\frac {W^h(t)}{W^h(0)}$ with fixed $h=2^{-4}$ for different time step $\tau$ in Figure \ref{fig:Wh}. It can be seen from Figure \ref{fig:Wh} that the normalized perimeter indeed decreases in time under different $\tau$, which numerically substantiates the theoretical analysis in theorem \ref{Them: Area-perimeter full}.

\begin{figure}[!htb]
\centering
\includegraphics[width=.32\textwidth]{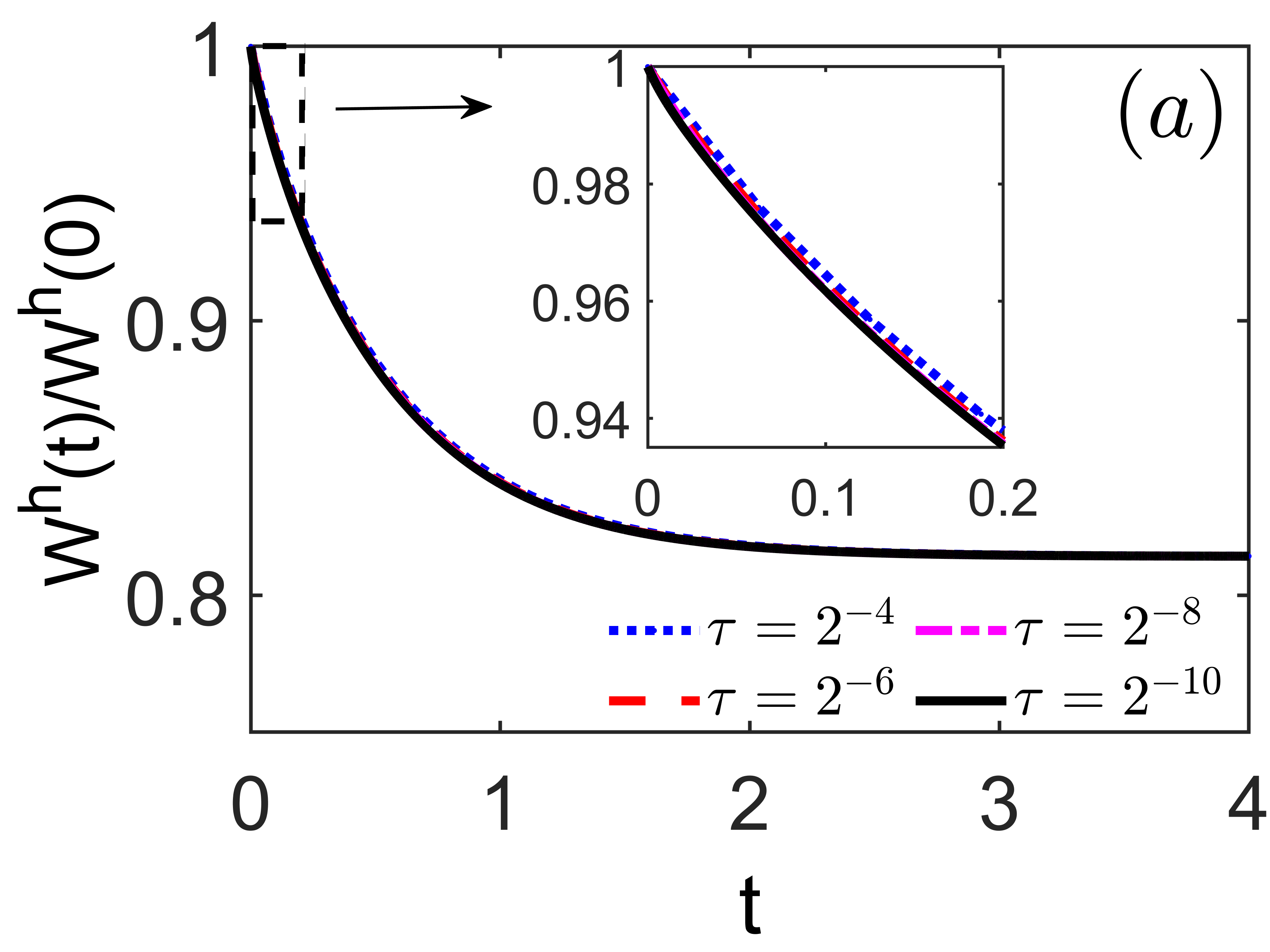}
\includegraphics[width=.32\textwidth]{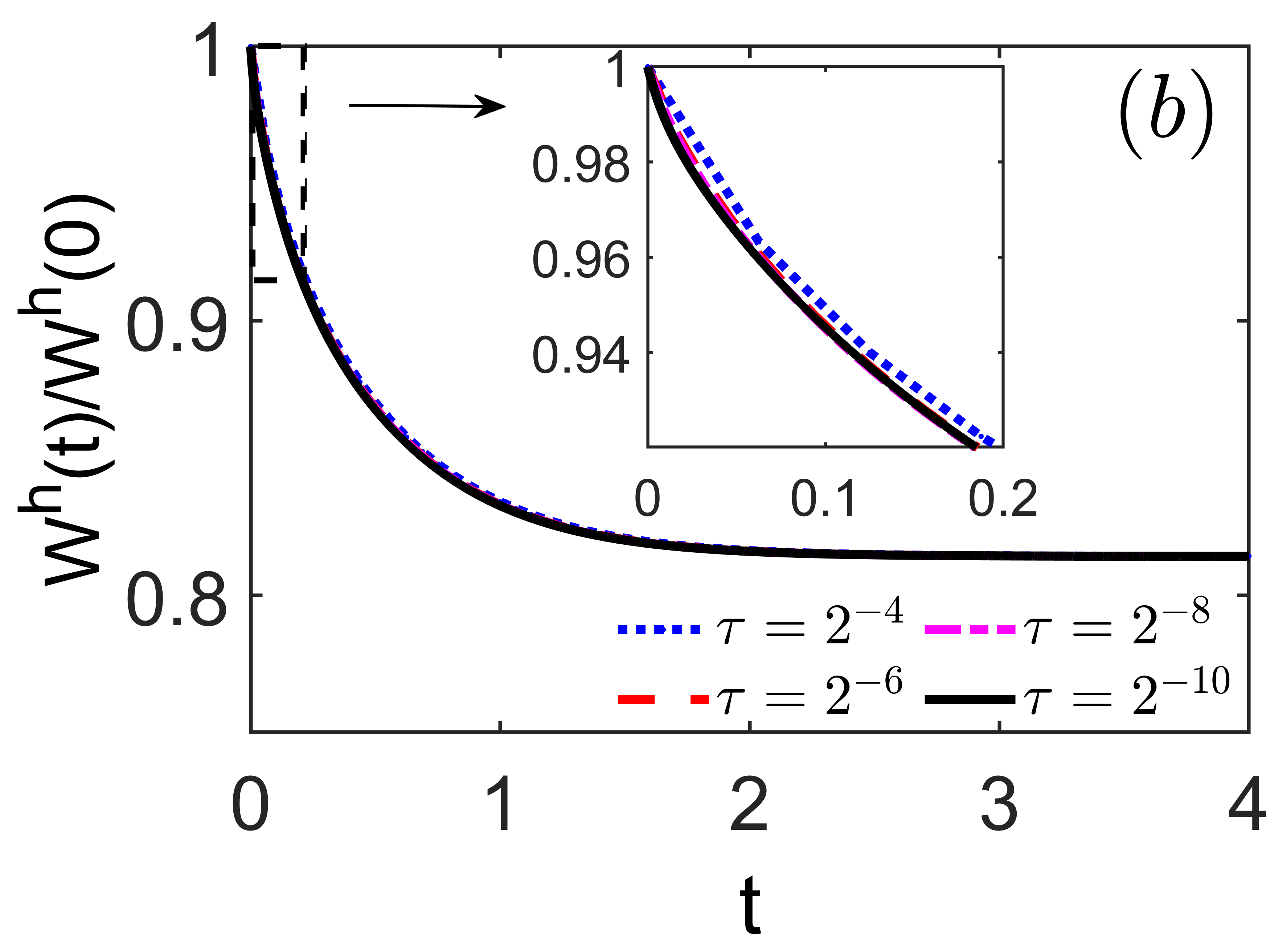}
\includegraphics[width=.32\textwidth]{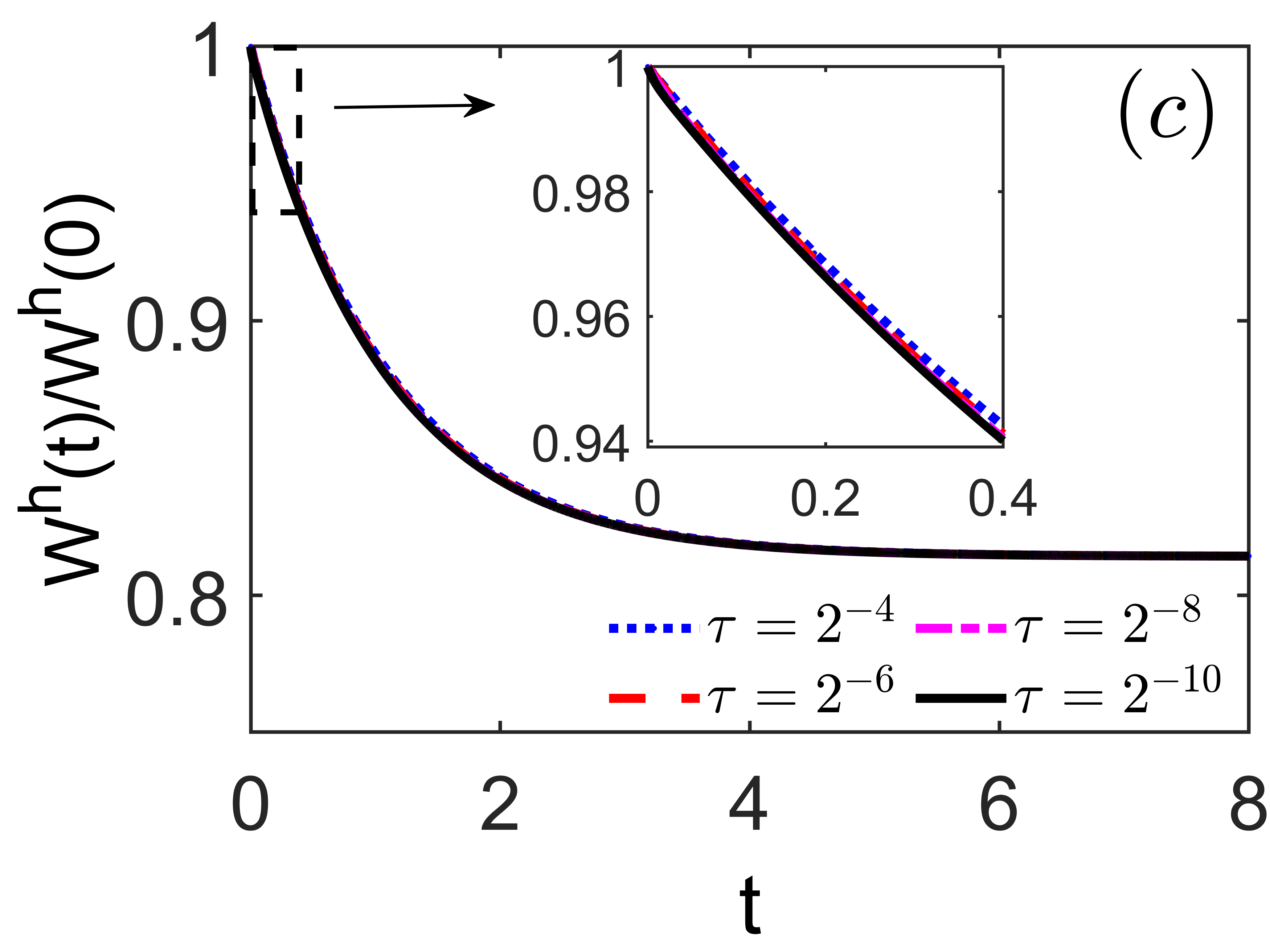}\\
 \vspace{0.1cm}
\includegraphics[width=.32\textwidth]{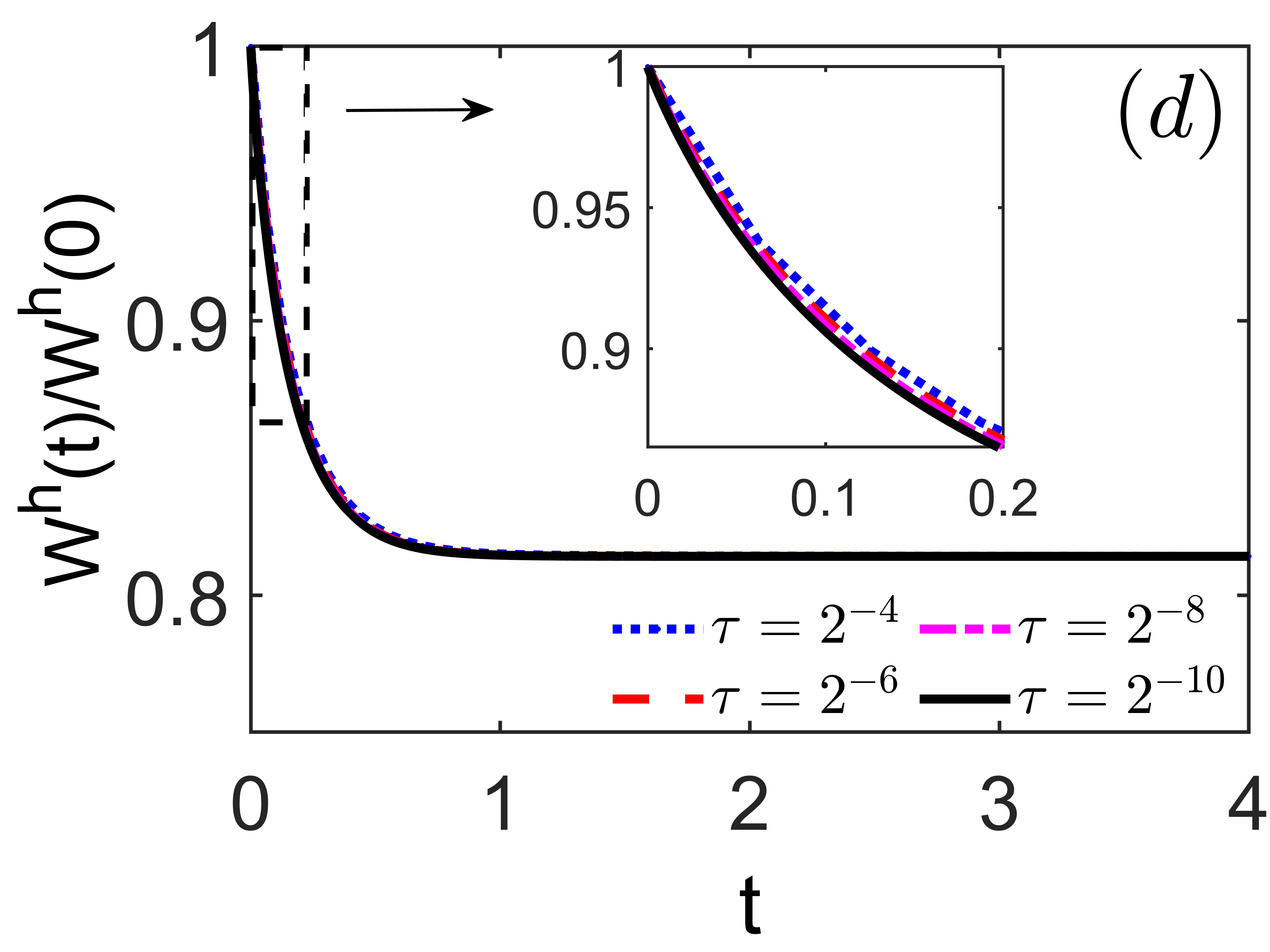}
\includegraphics[width=.32\textwidth]{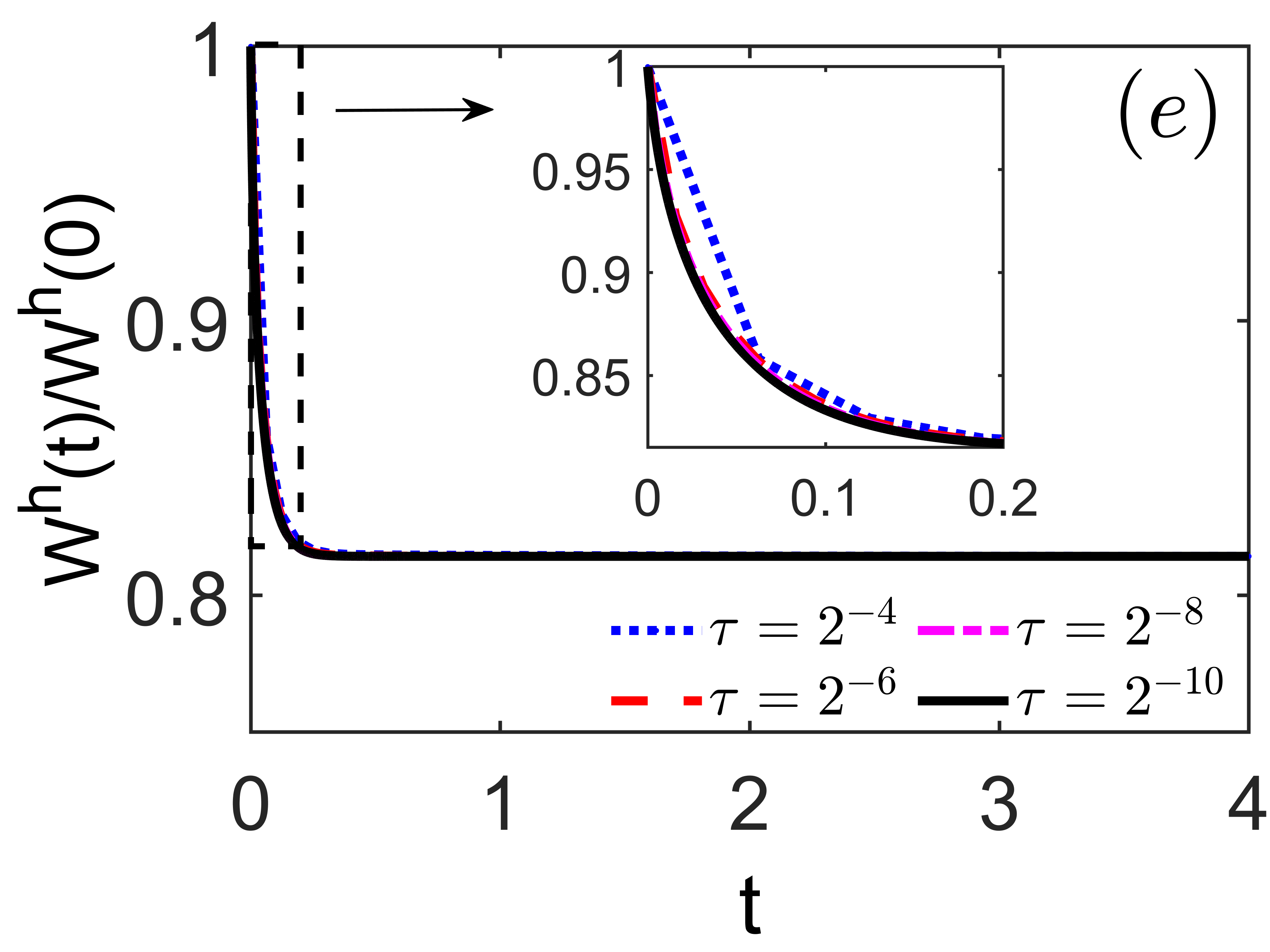}
\includegraphics[width=.32\textwidth]{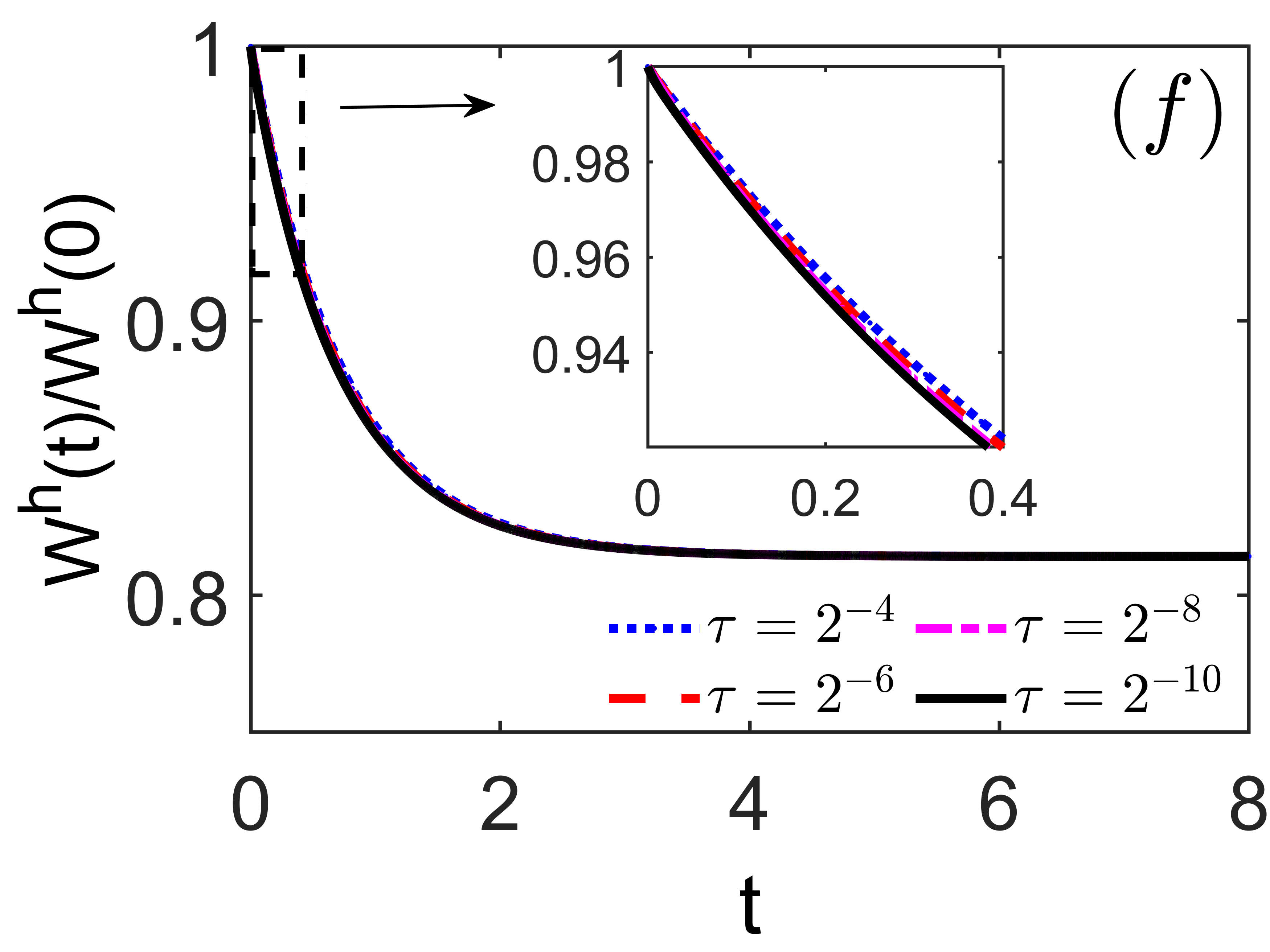}
\caption{{Time evolutions of the normalized perimeter $\frac{ W^h(t)}{W^h(0)}$ of the SP-PFEM \eqref{Model full eq} with fixed $h=2^{-4}$ for different $\tau$, where we choose the following parameters $\alpha$ and $\beta$: (a) $\alpha=1$, $\beta=-1$; (b) $\alpha=2$, $\beta=-1$; (c) $\alpha=1/3$, $\beta=-1$; (d) $\alpha=-1$, $\beta=1$; (e) $\alpha=-2$, $\beta=1$; (f) $\alpha=-1/3$, $\beta=1$.}}
\label{fig:Wh}
\end{figure}

Figure \ref{fig:Rh} depicts the time evolutions of the mesh ratio indicator $R^h(t)$ for different $h$. By using a uniform partition of the polar angle, the initial curve is discretized into a polygonal curve with non-uniform distribution with respect to the arc length. As time evolves, we can see the mesh ratio indicator $R^h(t)$ gradually decreases to approximate 1, which implies that the mesh points on the polygonal curve eventually be equally distributed, i.e., asymptotic equal mesh distribution. This property can follow from proposition 2.1 in \cite{bao2021structure}.

\begin{figure}[!htp]
\centering
\includegraphics[width=.32\textwidth]{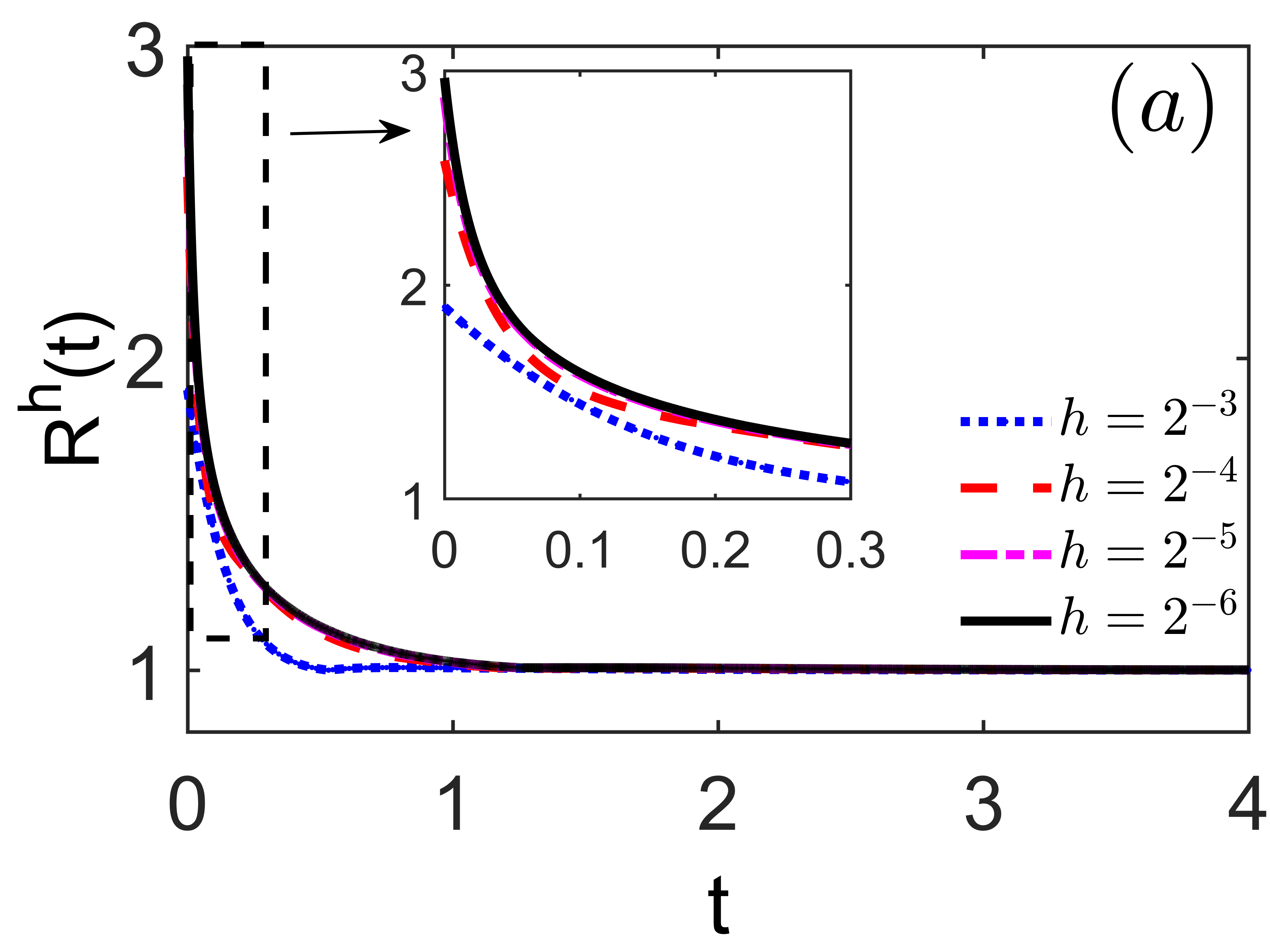}
\includegraphics[width=.32\textwidth]{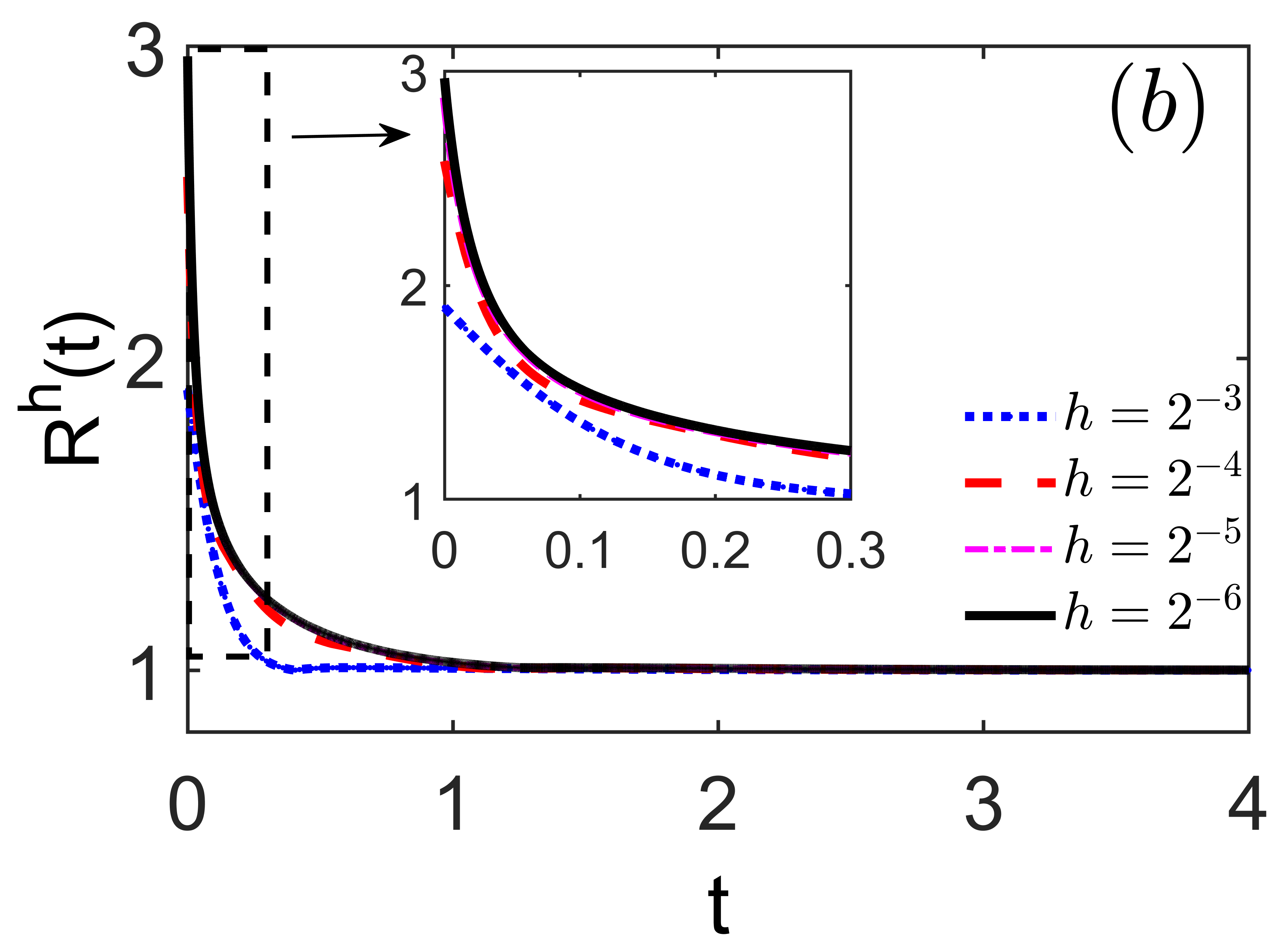}
\includegraphics[width=.32\textwidth]{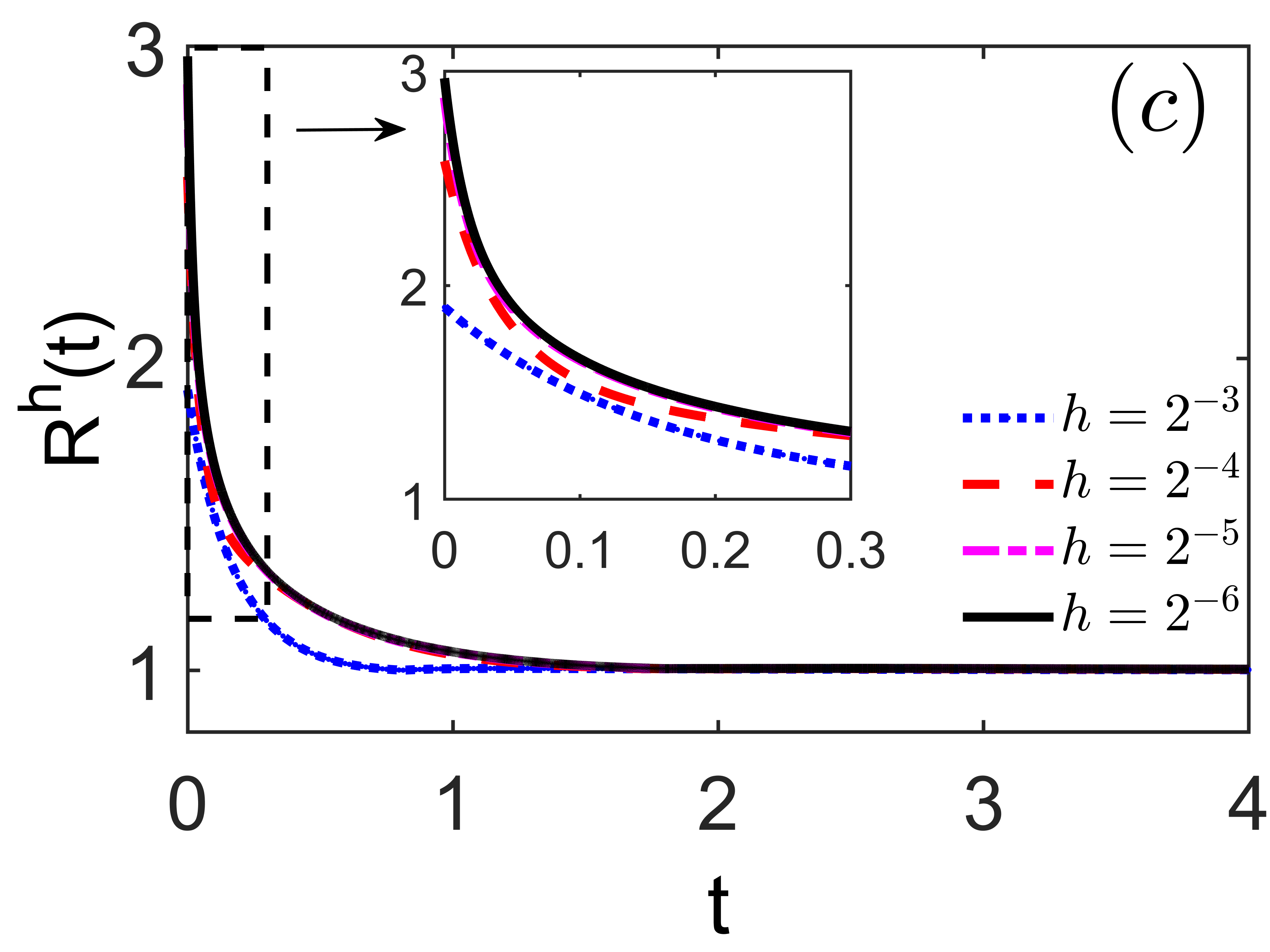}\\
 \vspace{0.1cm}
\includegraphics[width=.32\textwidth]{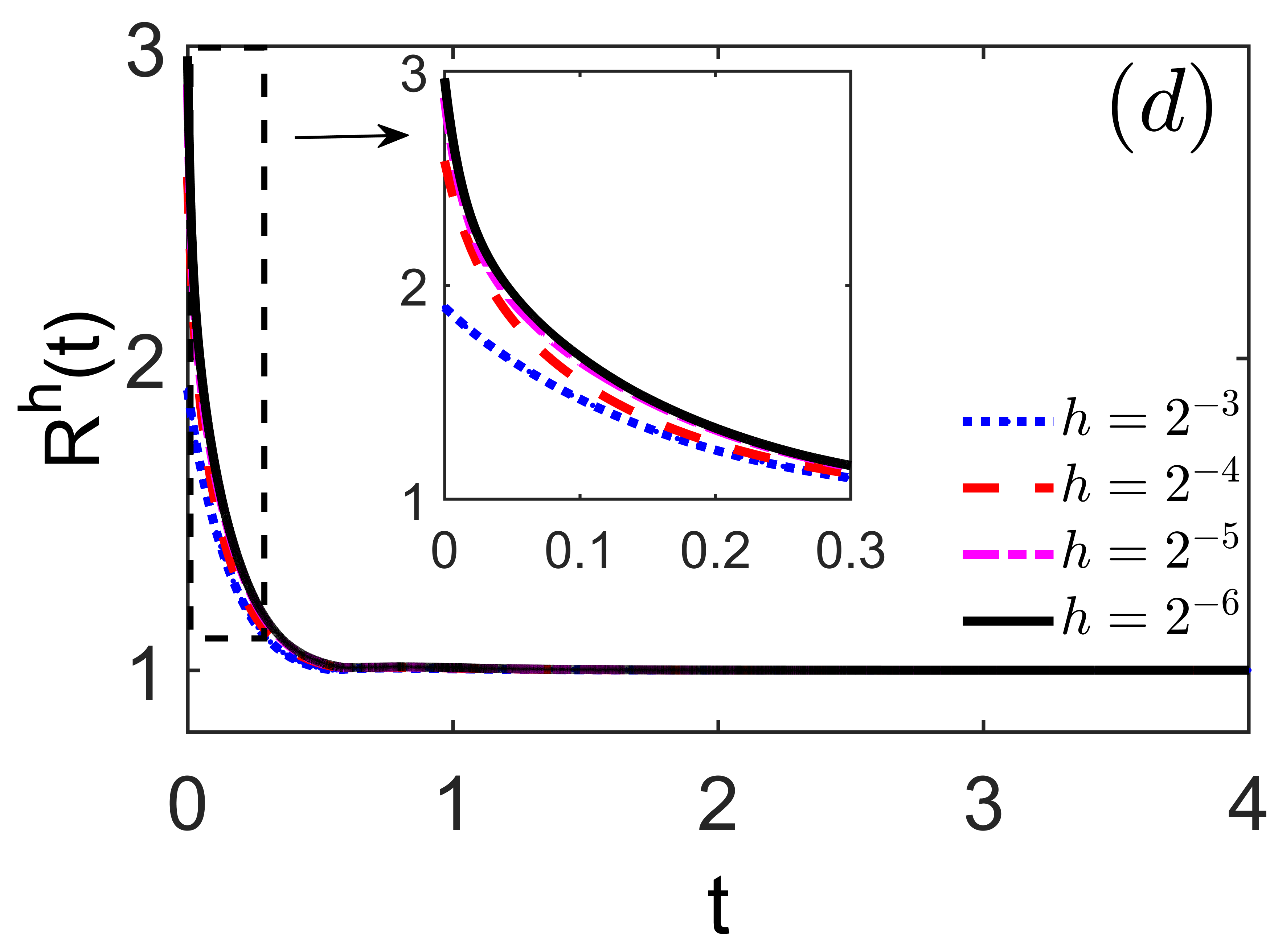}
\includegraphics[width=.32\textwidth]{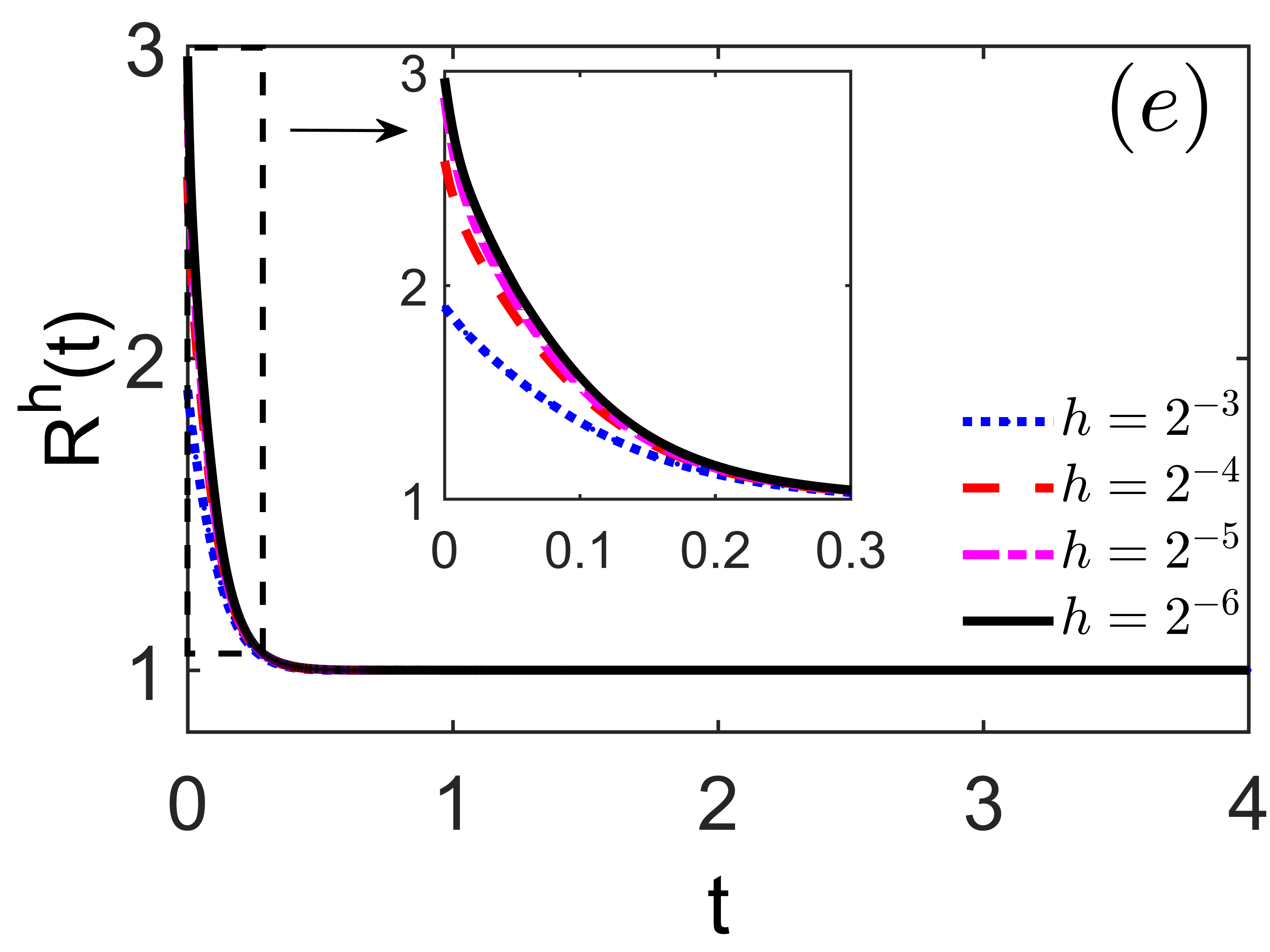}
\includegraphics[width=.32\textwidth]{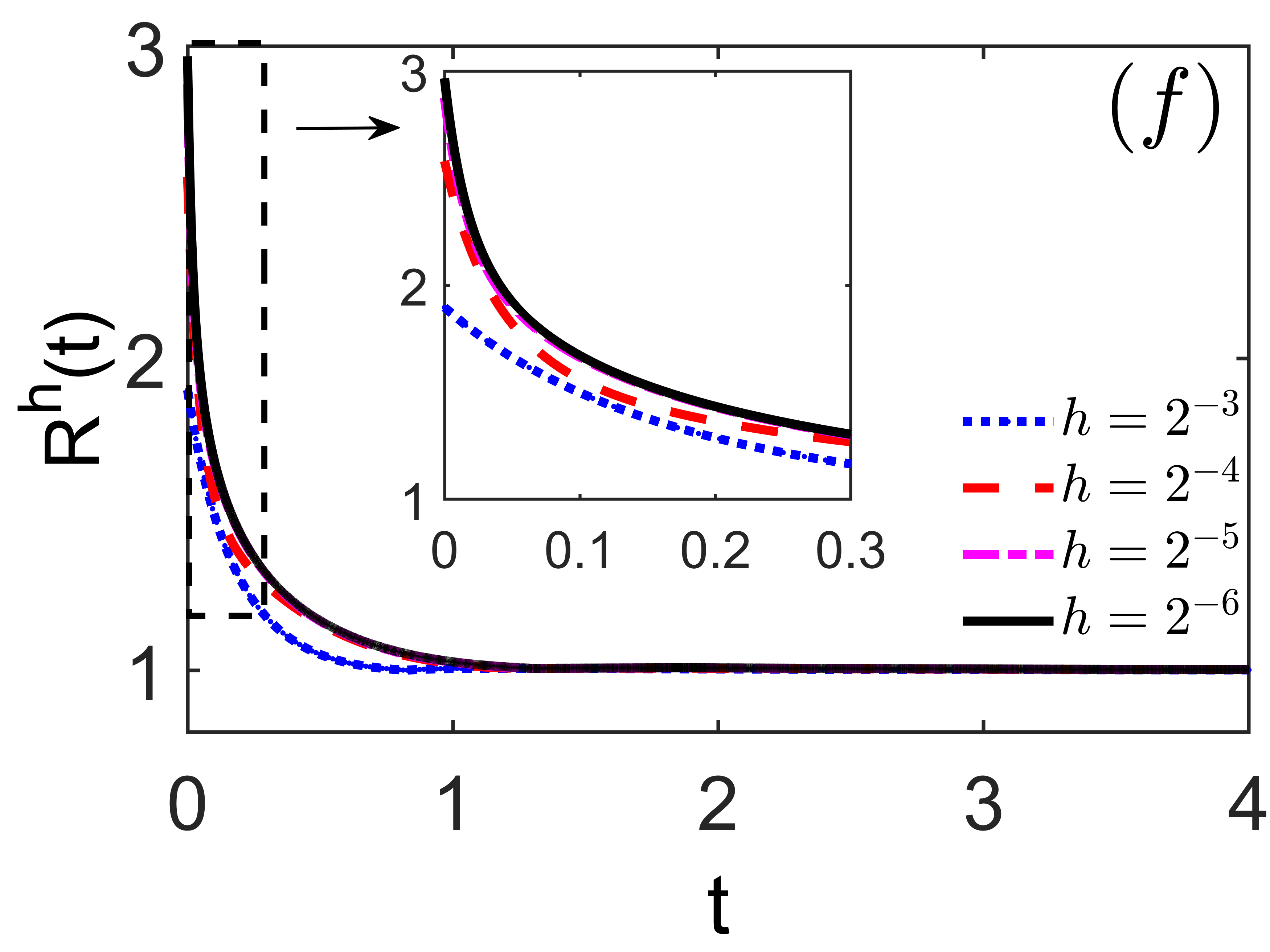}
\caption{ {Time evolutions of the mesh ratio $R^h(t)$ of the SP-PFEM \eqref{Model full eq} with $\tau=h^2$ for different $h$, where we choose the following parameters $\alpha$ and $\beta$: (a) $\alpha=1$, $\beta=-1$; (b) $\alpha=2$, $\beta=-1$; (c) $\alpha=1/3$, $\beta=-1$; (d) $\alpha=-1$, $\beta=1$; (e) $\alpha=-2$, $\beta=1$; (f) $\alpha=-1/3$, $\beta=1$.}}
\label{fig:Rh}
\end{figure}
\subsection{Application of the SP-PFEM \eqref{Model full eq} for morphological evolutions}
 Here we will use the SP-PFEM \eqref{Model full eq} to simulate morphological evolutions of closed curves governed by the area-conserved generalized mean curvature flow \eqref{Model,cont,eq}--\eqref{Model,cont,lag}. The evolutions of the initial ellipse curve $\frac{x^2}{{3^2}}+y^2=1$ for different parameters $\alpha$ and $\beta$ are depicted in Figure \ref{fig:envoleovel}. We observe that all curves evolve into a circle finally.
\begin{figure}[!htb]
\centering
\includegraphics[width=.32\textwidth]{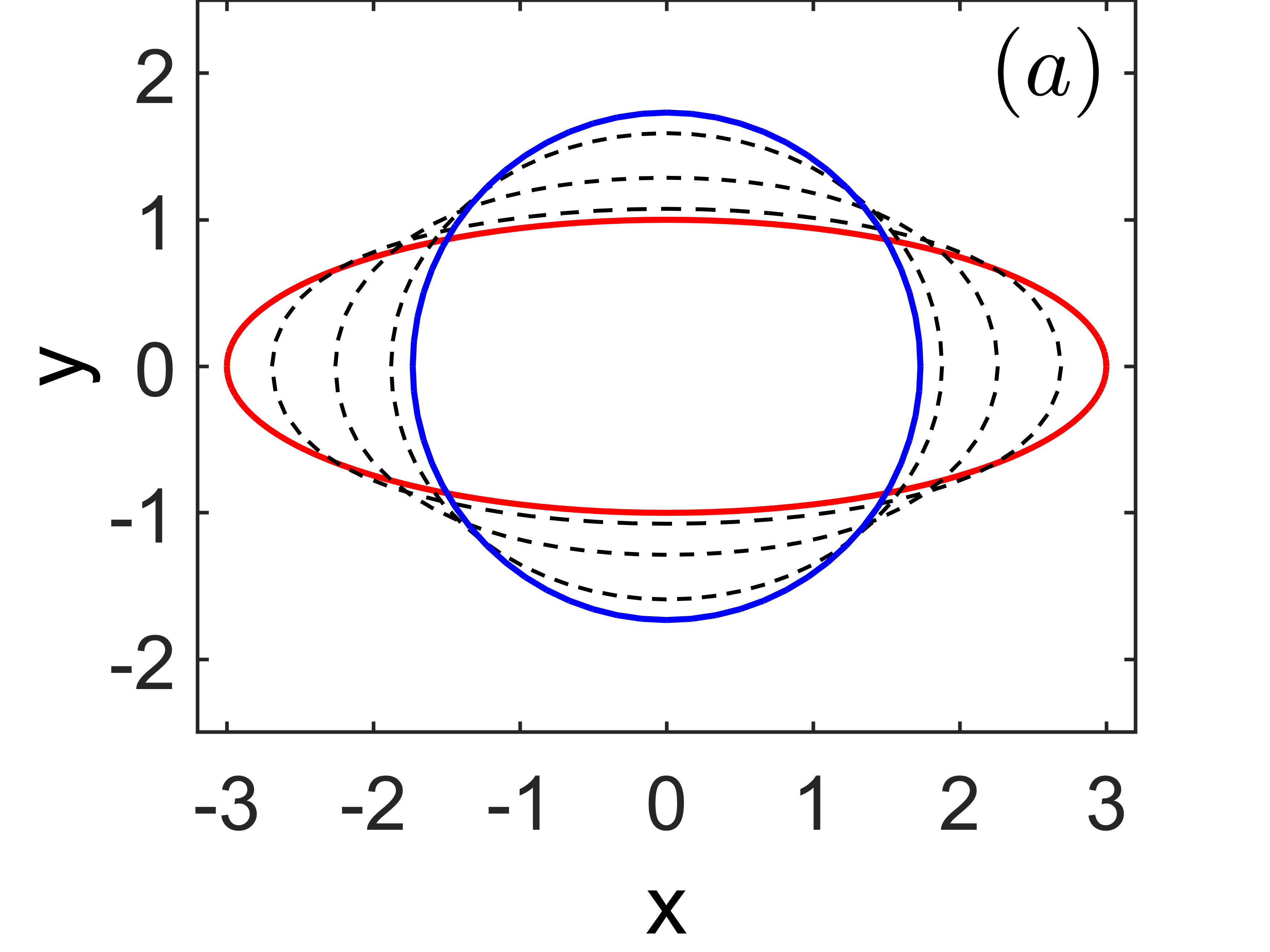}
\includegraphics[width=.32\textwidth]{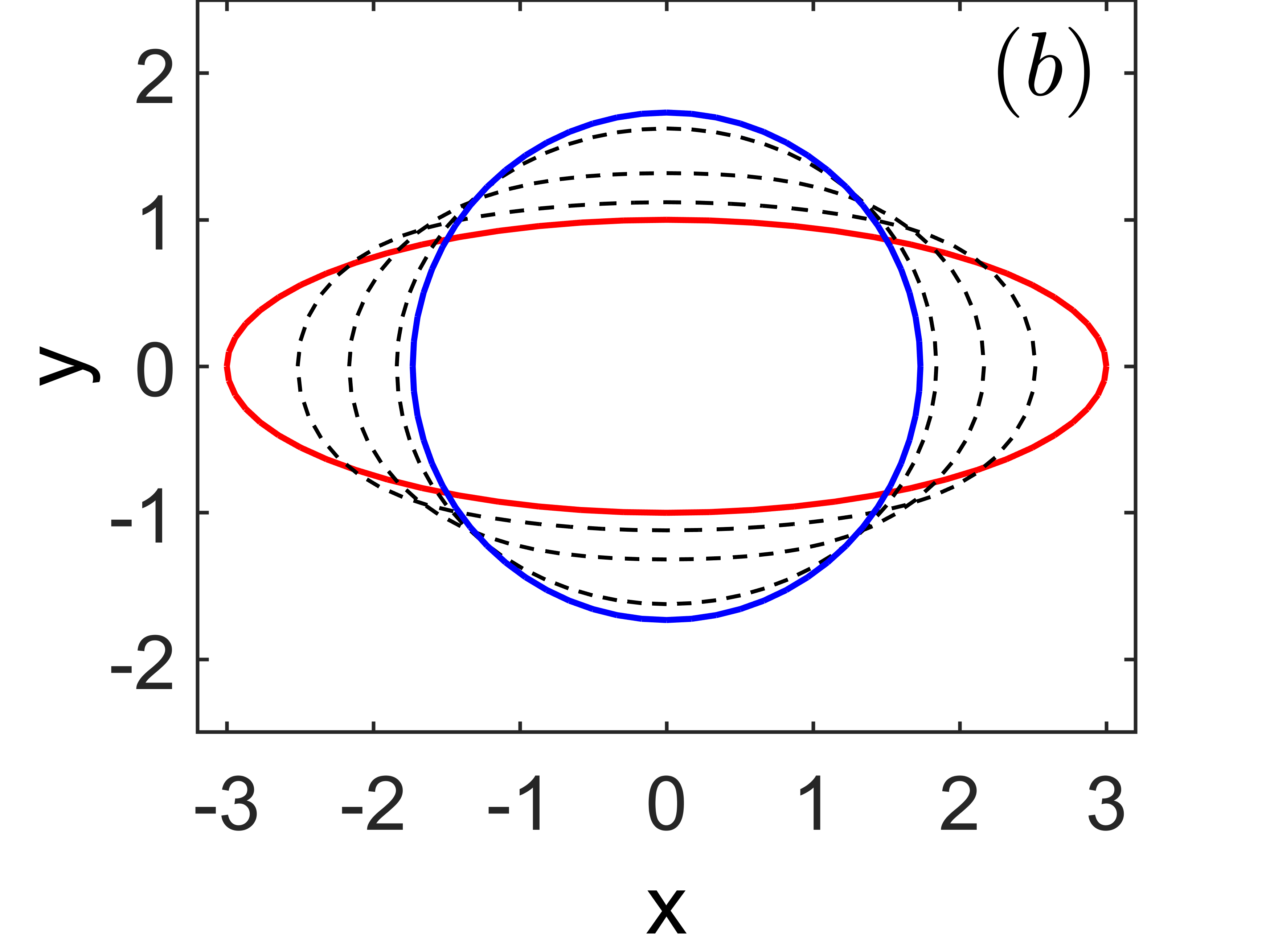}
\includegraphics[width=.32\textwidth]{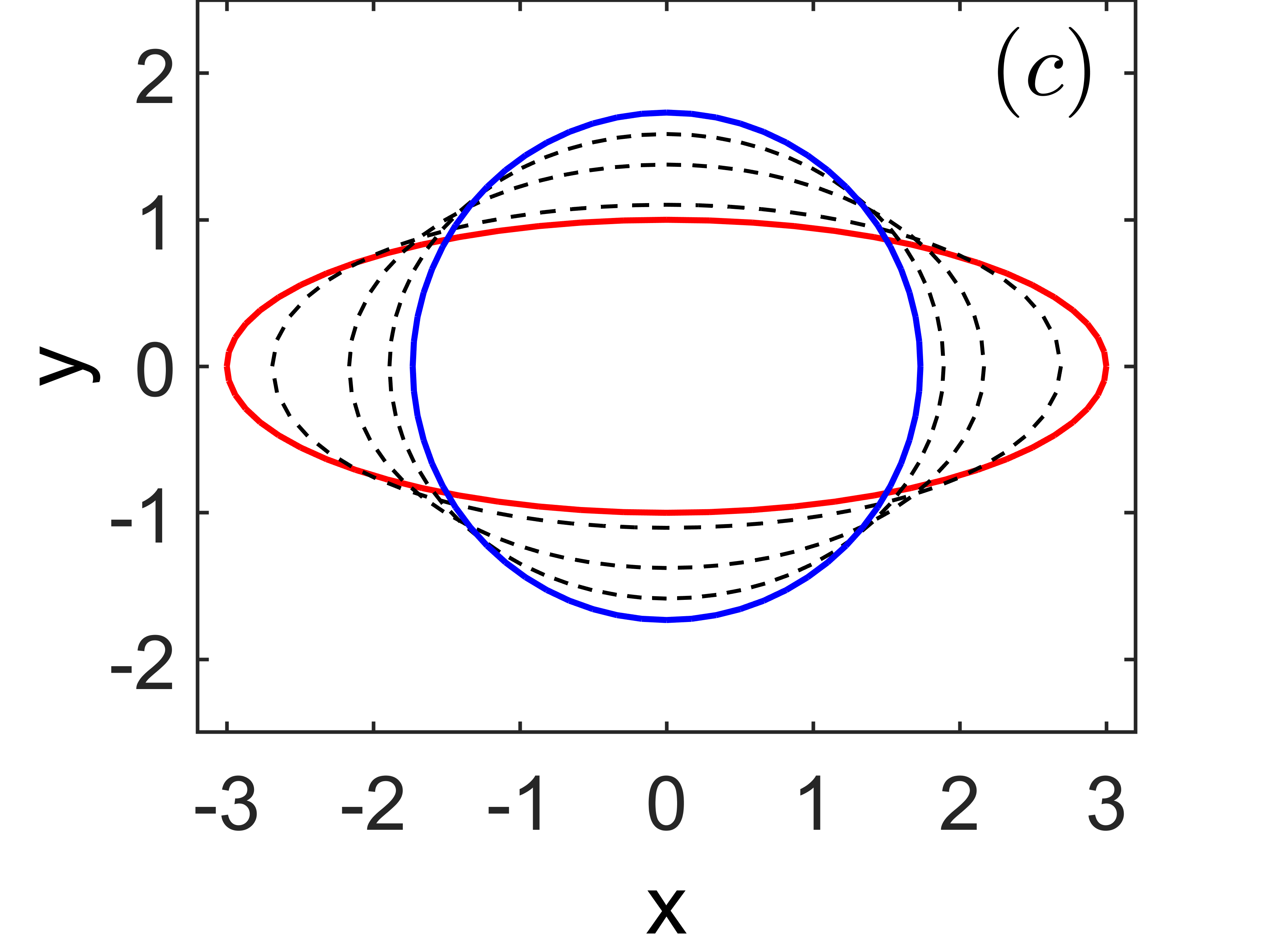}\\
 \vspace{0.2cm}
\includegraphics[width=.32\textwidth]{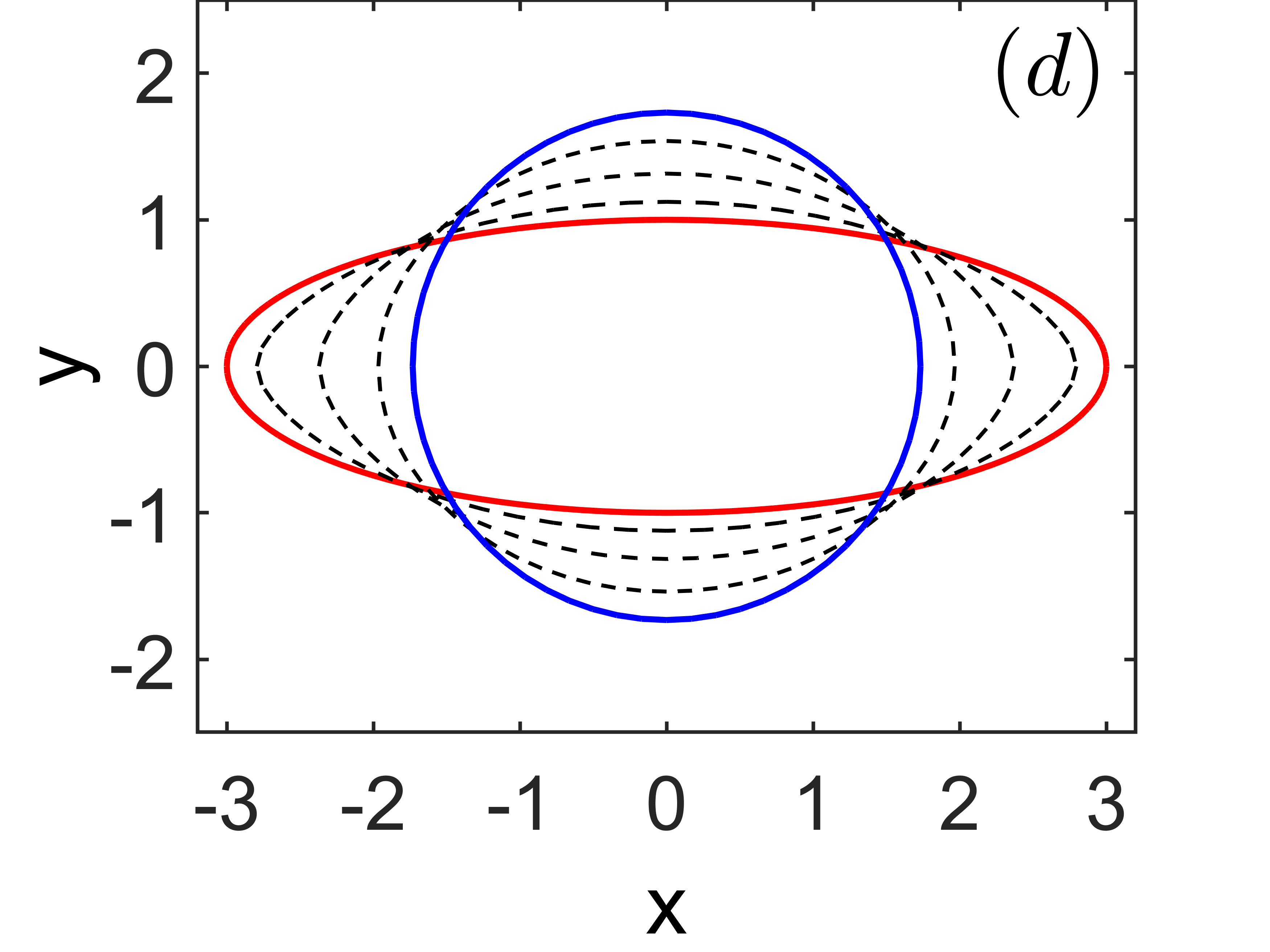}
\includegraphics[width=.32\textwidth]{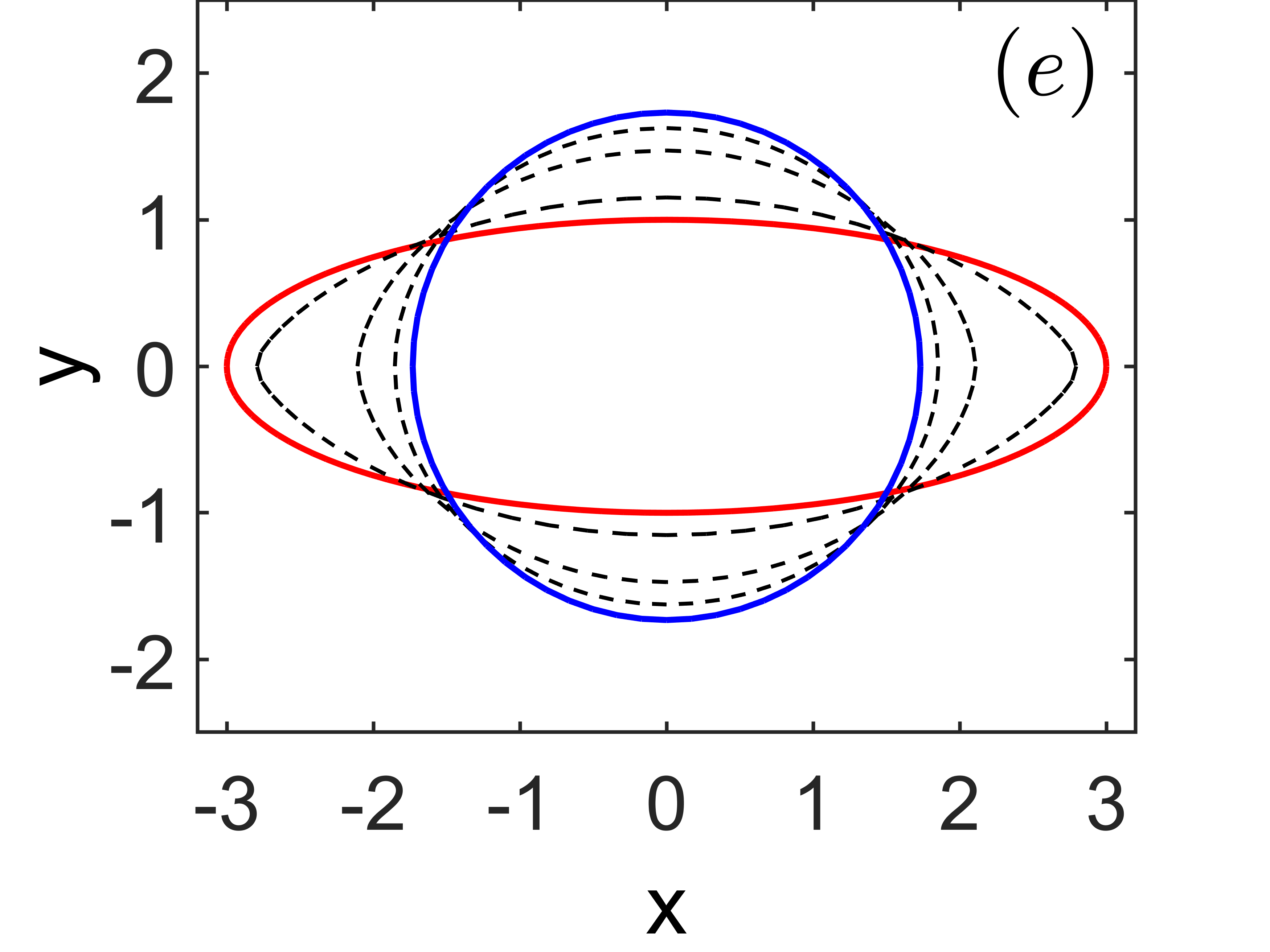}
\includegraphics[width=.32\textwidth]{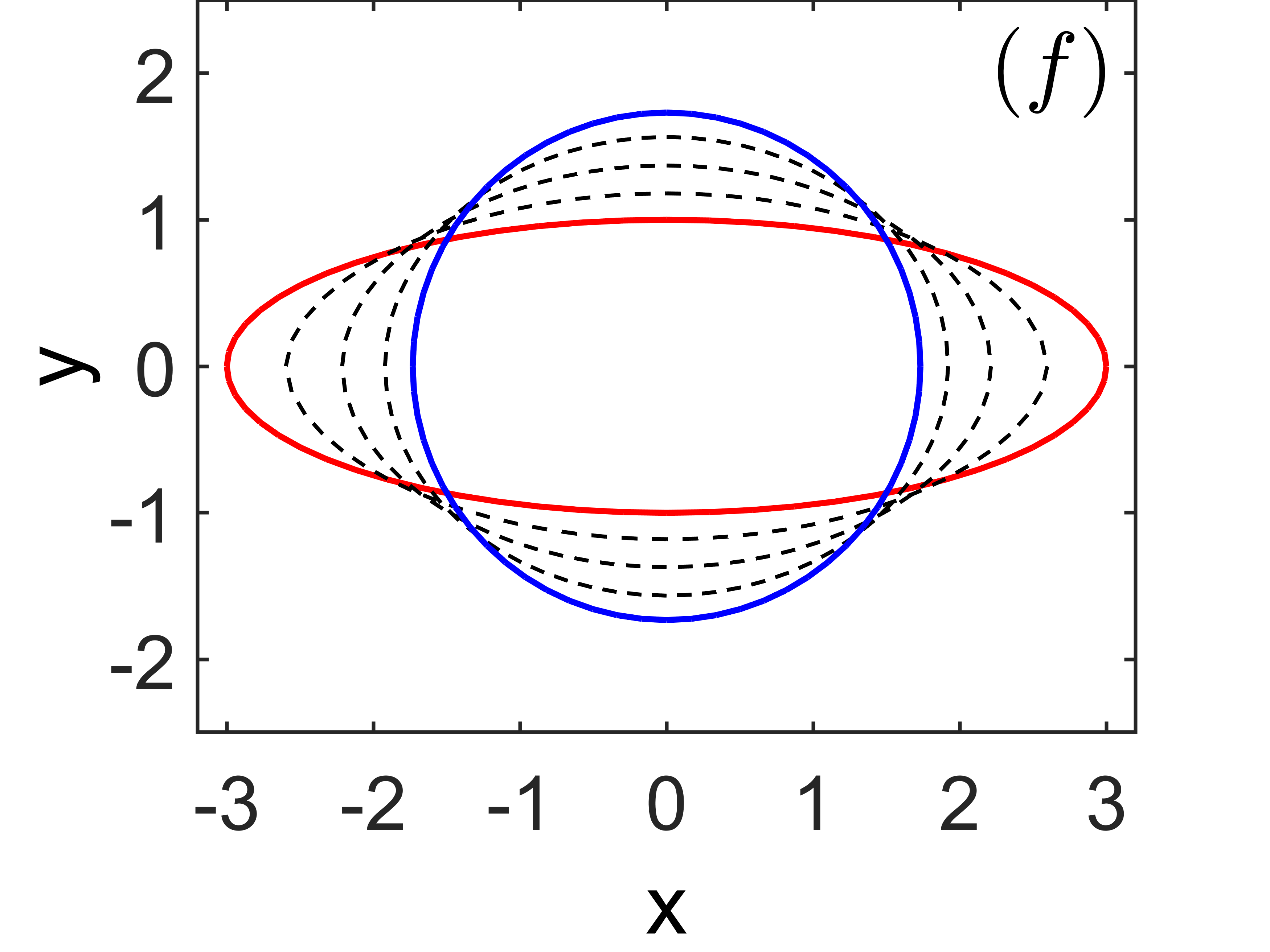}
\caption{{ Morphological evolutions of the initial ellipse curve (red solid line) towards its equilibrium shape (blue solid line) at different times (dashed lines) for different parameters $\alpha$ and $\beta$ as follows: (a) $\alpha=1$, $\beta=-1$, $t\in\{0.25,0.75,2\}$; (b) $\alpha=2$, $\beta=-1$, $t\in\{0.25,0.75,2\}$; (c) $\alpha=1/3$, $\beta=-1$, $t\in\{0.5,2,4\}$; (d) $\alpha=-1$, $\beta=1$, $t\in\{0.05,0.2,0.5\}$; (e) $\alpha=-2$, $\beta=1$, $t\in\{0.01,0.1,0.2\}$; (f) $\alpha=-1/3$, $\beta=1$, $t\in\{0.5,1.25,2.5\}$. Here we choose $h=2^{-6}$ and $\tau=h^2$.}}
\label{fig:envoleovel}
\end{figure}
Furthermore, when $\alpha=1$ and $\beta=-1$, we apply the SP-PFEM \eqref{Model full eq} to three more complex initial curves given by the following cases:
\begin{align}
&\mbox{(Case I):}~~~~
\displaystyle x=(2+\cos(6\theta))\cos\theta, ~~~~
y=(2+\cos(6\theta))\sin\theta,\nonumber\\
&\mbox{(Case II):}~~~
\displaystyle x=\cos\theta,~~~~~~~~~~~~~~~~~~~~~
y=2\sin\theta-1.9\sin^3\theta,\nonumber\\
&\mbox{(Case III):}~
\displaystyle x=\cos\theta,~
y=\frac{1}{2}\sin\theta+\sin (\cos\theta)+(0.2+\sin\theta \sin^2(3\theta))\sin\theta,\nonumber
\end{align}
where $\theta\in [0,2\pi]$. Figure \ref{fig:envolesixflower}, Figure \ref{fig:envoledumbleline} and Figure \ref{fig:envolepowerline} depict the curve evolutions of these initial shapes, respectively. It can be seen in Figure \ref{fig:envolesixflower} that the six petals of the curve (Case I) gradually disappear and a circle is finally formed. Analogous numerical results for the complex curves (Case II) and (Case III) can be observed in Figure \ref{fig:envoledumbleline} and Figure \ref{fig:envolepowerline}, respectively. These numerical examples further demonstrate the applicability and reliability of the SP-PFEM \eqref{Model full eq}.

\begin{figure}[!htb]
\centering
  \includegraphics[width=.68\textwidth]{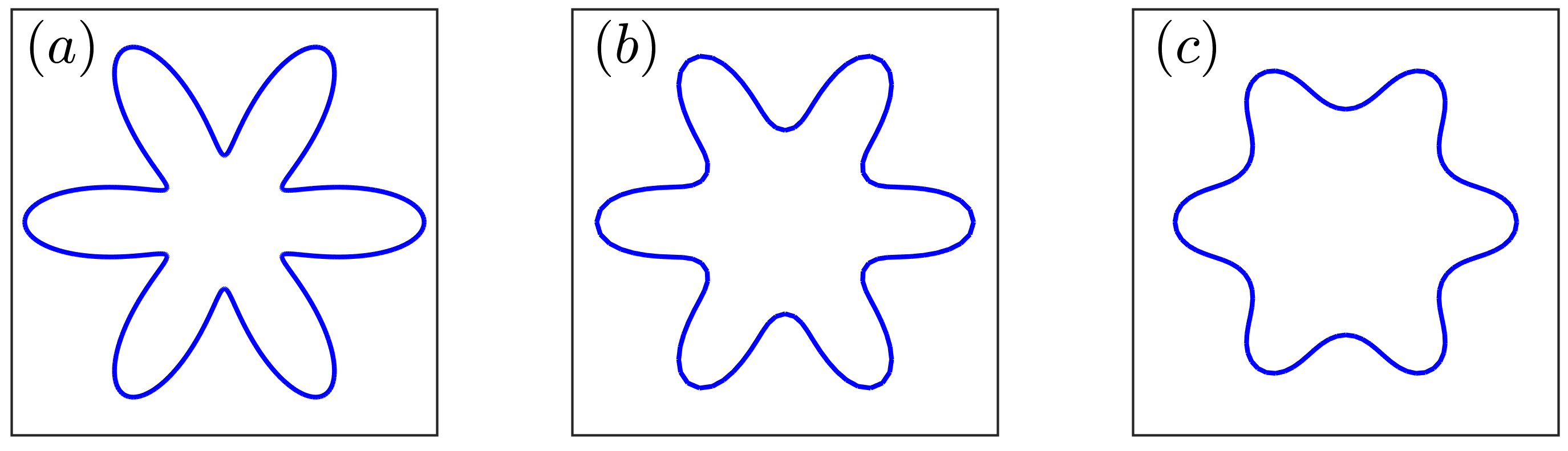}\\
   \includegraphics[width=.68\textwidth]{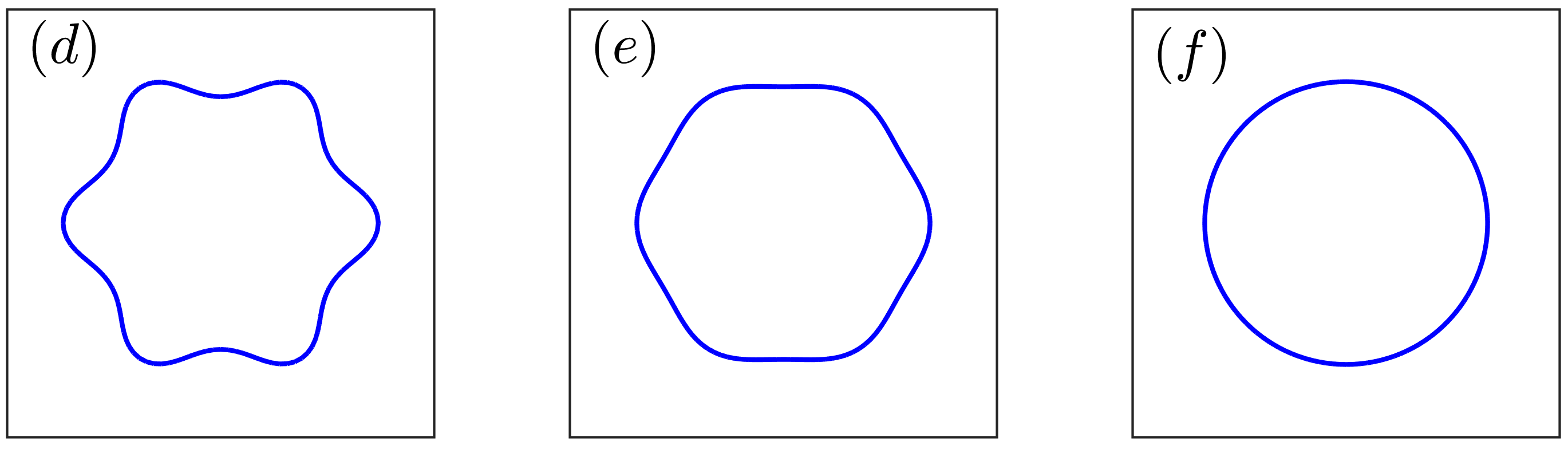}
\caption{{Evolution of the initial curve (Case I) at different times: (a) $t=0$, (b) $t=0.05$, (c) $t=0.15$, (d) $t=0.25$, (e) $t=0.4$, (f) $t=1$. Parameters are chosen as $\alpha=1$, $\beta=-1$, $h=2^{-7}$ and $\tau=h^2$.}}
\label{fig:envolesixflower}
\end{figure}
%%%这个版本可以调png，那我就不用生成pdf了
\begin{figure}[!htb]
\centering
  \includegraphics[width=.68\textwidth]{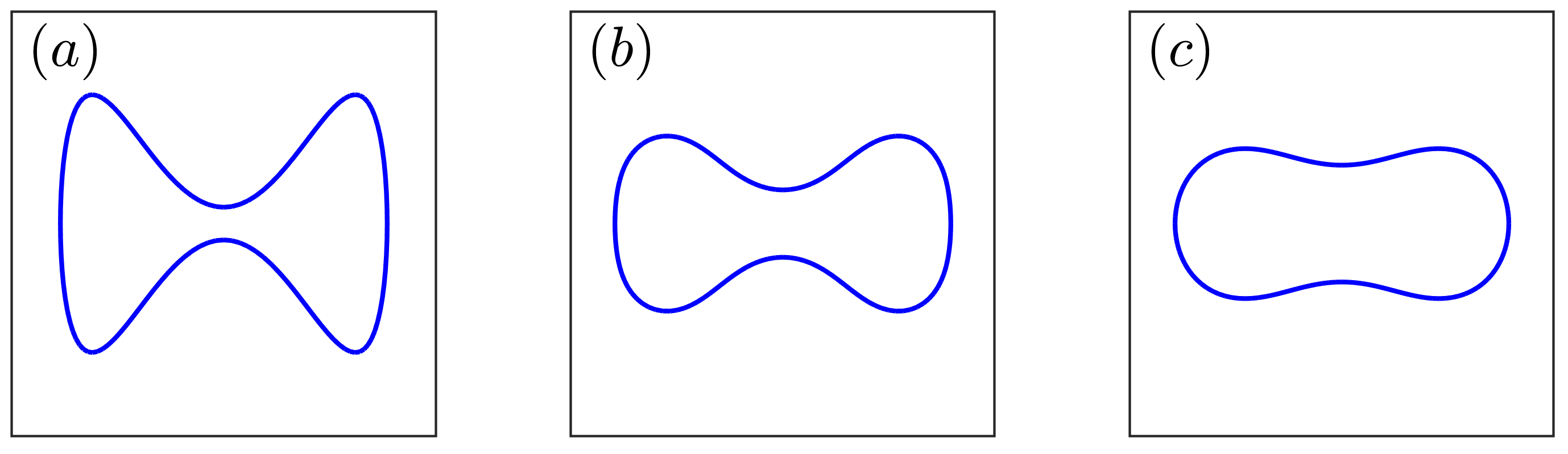}\\
   \includegraphics[width=.68\textwidth]{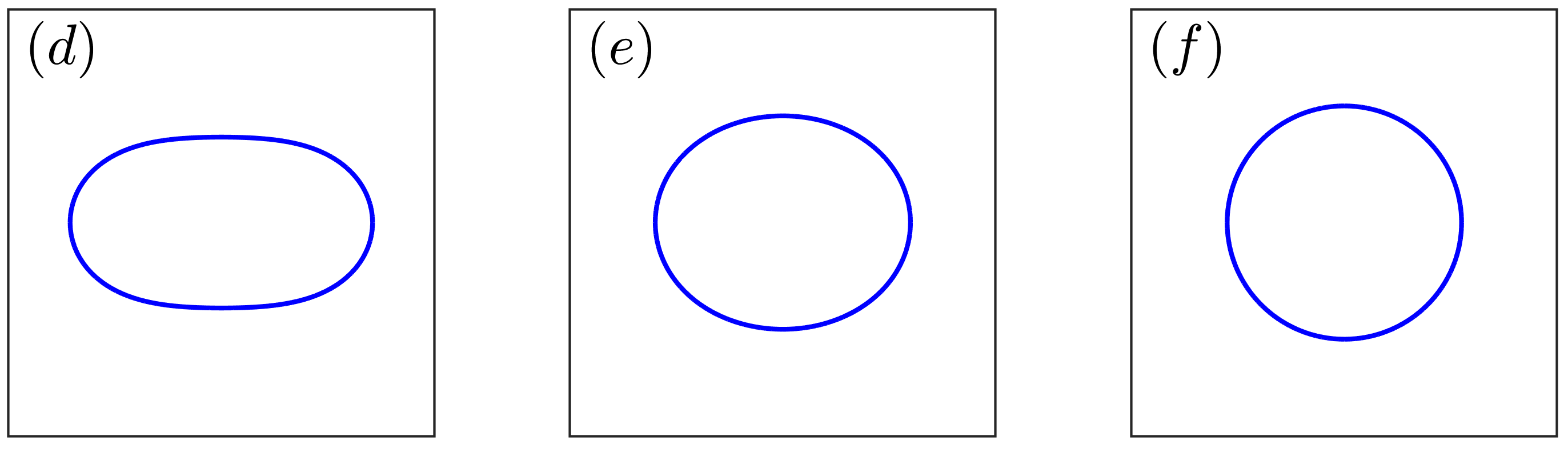}
\caption{{Evolution of the initial curve (Case II) at different times: (a) $t=0$, (b) $t=0.05$, (c) $t=0.1$, (d) $t=0.2$, (e) $t=0.4$, (f) $t=1$. Parameters are chosen as $\alpha=1$, $\beta=-1$, $h=2^{-7}$ and $\tau=h^2$.}}
\label{fig:envoledumbleline}
\end{figure}

\begin{figure}[!htb]
\centering
  \includegraphics[width=.68\textwidth]{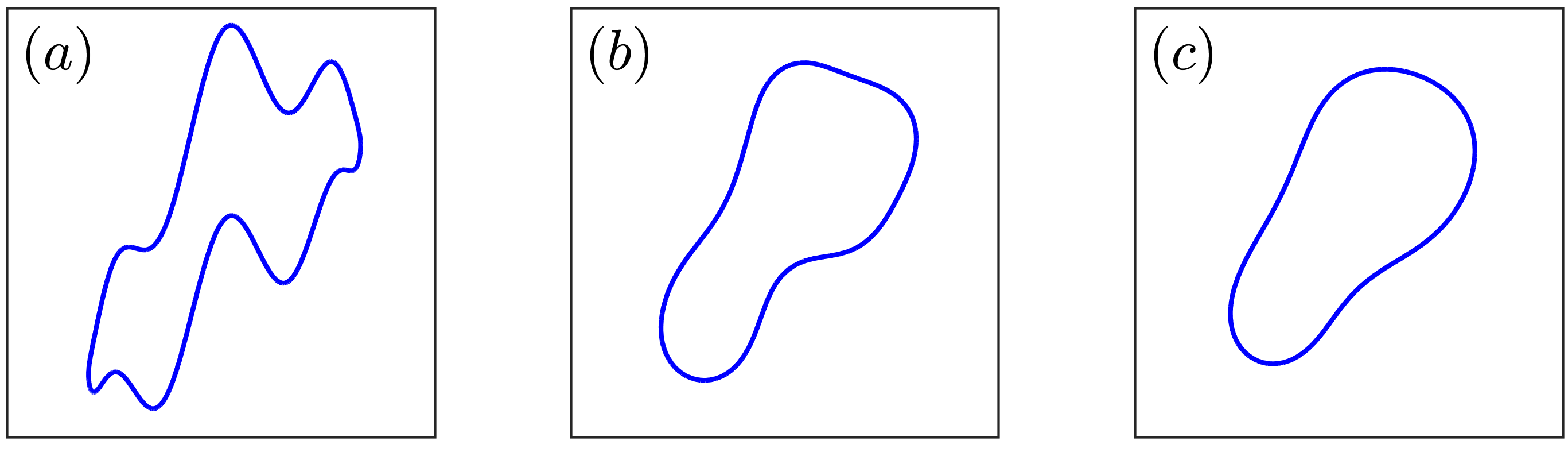}\\
   \includegraphics[width=.68\textwidth]{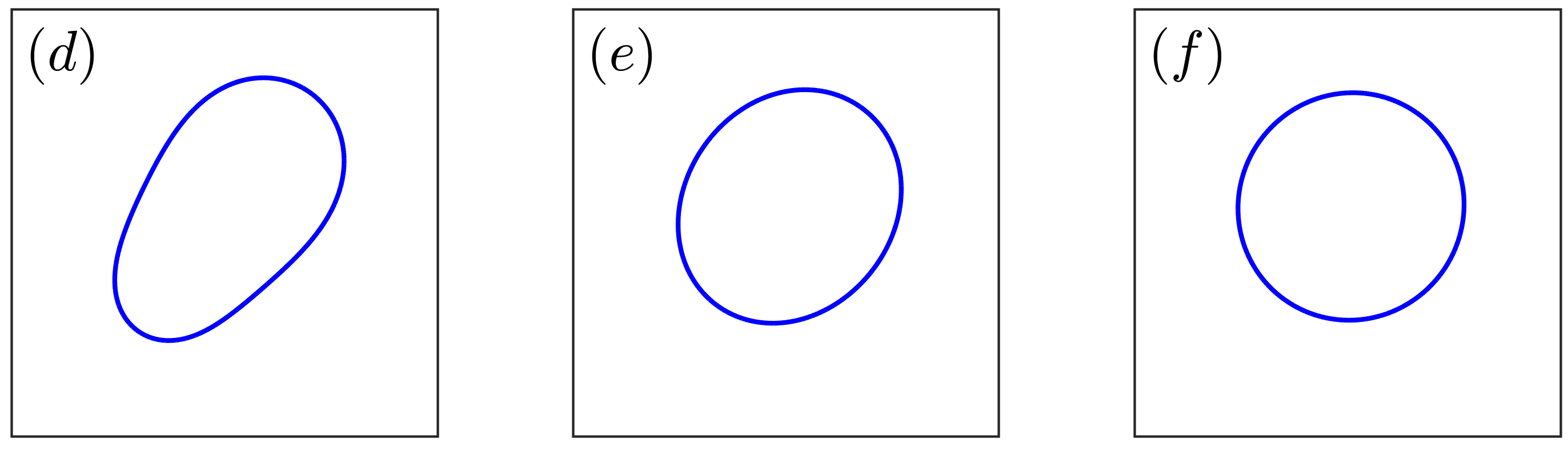}
\caption{{Evolution of the initial curve (Case III) at different times: (a) $t=0$, (b) $t=0.05$, (c) $t=0.1$, (d) $t=0.2$, (e) $t=0.5$, (f) $t=1$. Parameters are chosen as $\alpha=1$, $\beta=-1$, $h=2^{-7}$ and $\tau=h^2$.}}
\label{fig:envolepowerline}
\end{figure}
\section{Conclusions}
\sloppy{}
This study investigated numerical approximation of the area-conserved generalized mean curvature flow. We present a variational formulation and prove that it preserves two fundamental geometric structures of flows: the conservation of the area enclosed by the closed curve and the decrease of the perimeter of the curve. Then the variational problem is discretized in space by PFEM, and the structure-preserving property is also established for the semi-discrete scheme. Finally, a fully-discrete scheme is constructed by using the backward Euler method in time and piecewise linear parametric finite element approximation in space. We prove that the proposed method preserves the two essential geometric structures simultaneously at the discrete level. In the end, numerical results confirm these conclusions and suggest that the proposed SP-PFEM is second-order accurate in space and enjoys asymptotic equal mesh distribution during the evolution. Therefore,  our work provides a reliable and powerful tool for the simulation of area-conserved generalized mean curvature flow in 2D. In the future, we will further explore the extension of the new SP-PFEM for volume-preserving generalized mean curvature flow in three dimensions.

\bibliographystyle{spmpsci}      % mathematics and physical sciences
\bibliography{thebibJSC}

\section*{Statements and Declarations}
\section*{Funding}
\sloppy{}

This research was supported by National Natural Science Foundation of China (Grant No. [11701523]).
\section*{Competing Interests}
\sloppy{}

The authors have no relevant financial or non-financial interests to disclose.
\section*{Author Contributions}
\sloppy{}

All authors contributed to the study conception and design.

\textbf{Lifang Pei:} Paper writing, Theoretical analysis, Numerical tests.

\textbf{Yifei Li: } Theoretical analysis, Writing-review and editing.

All authors read and approved the final manuscript.
\section*{Data Availability}
\sloppy{}

%Data will be made available on reasonable request.
The data generated or analyzed during the current study are available from the corresponding author on reasonable request.

\end{document}